\documentclass[11pt,a4paper]{amsart}
\usepackage[english]{babel}
\usepackage[utf8]{inputenc}
\usepackage[T1]{fontenc}
\usepackage{csquotes} 
\usepackage{lmodern}
\usepackage{amssymb}
\usepackage{float}
\usepackage{amscd}
\usepackage{amsthm}
\usepackage[top=1.9cm, bottom=1.9cm, left=1.9cm, right=1.9cm, twoside=false]{geometry}
\usepackage{array}
\usepackage{longtable}
\usepackage{mathtools}
\usepackage{graphicx}
\usepackage{wrapfig}
\usepackage{color}
\usepackage[usenames,x11names]{xcolor}
\usepackage[all]{xy}
\usepackage{enumerate}
\usepackage{tikz}
\usetikzlibrary{calc,intersections}
\tikzset{
	thick/.style=      {line width=1.2pt},
	add/.style args={#1 and #2}{to path={%
			($(\tikztostart)!-#1!(\tikztotarget)$)--($(\tikztotarget)!-#2!(\tikztostart)$)%
			\tikztonodes},add/.default={.2 and .2}}
}
\usepackage{tikz-cd}

\usepackage{multirow}
\usepackage{units} 
\usepackage{yfonts}
\usepackage{mathrsfs}
\usepackage{caption}
\usepackage{subcaption}
\usepackage[export]{adjustbox}

\usepackage[hyphens]{url} 
\usepackage[hypertexnames=false,breaklinks]{hyperref} 
\usepackage[hyphenbreaks]{breakurl}
\hypersetup{linkcolor  =DodgerBlue3, citecolor  = teal, urlcolor   = teal, colorlinks = true, hyperfootnotes =false}

\usepackage{etoolbox}

\usepackage[normalem]{ulem}

\newcommand{\quadtext}[1]{\quad\text{#1}\quad}

\usepackage[scr=boondoxo]{mathalfa}
\newcommand{\mm}{\mathscr{m}}
\renewcommand{\ll}{\mathscr{l}}

\newcommand{\kk}{\mathscr{k}}

\renewcommand{\to}{\longrightarrow}
\newcommand{\map}{\dashrightarrow}

\def\dim{\operatorname{dim}}

\def\coeff{\operatorname{coeff}}
\def\cf{\operatorname{cf}}

\def\Spec{\operatorname{Spec}}

\def\span{\operatorname{span}}

\def\Pic{\operatorname{Pic}}
\def\Cl{\operatorname{Cl}}

\def\Proj{\operatorname{Proj}}

\newcommand{\F}{\mathbb{F}}
\newcommand{\N}{\mathbb{N}}
\newcommand{\A}{\mathbb{A}}

\newcommand{\Z}{\mathbb{Z}}
\newcommand{\Q}{\mathbb{Q}}
\renewcommand{\P}{\mathbb{P}}

\renewcommand{\tilde}{\widetilde}

\renewcommand{\epsilon}{\varepsilon}
\renewcommand{\phi}{\varphi}

\newcommand{\id}{\mathrm{id}}
\newcommand{\ov}{\overline}

\newcommand{\tip}[1]{\mathrm{tip}(#1)}

\newcommand{\NE}{\operatorname{NE}}
\renewcommand{\path}{\operatorname{path}}
\newcommand{\Exc}{\operatorname{Exc}}

\newcommand{\ctr}{\operatorname{ctr}}
\newcommand{\Supp}{\operatorname{Supp}}
\newcommand{\Bk}{\operatorname{Bk}}
\newcommand{\redd}{_{\mathrm{red}}}
\newcommand{\am}{{\mathrm{am}}}

\renewcommand{\leq}{\leqslant}
\renewcommand{\geq}{\geqslant}
\newcommand{\ld}{\operatorname{ld}}

\newcommand{\tcf}{\operatorname{tcf}}
\newcommand{\mult}{\operatorname{mult}}

\theoremstyle{plain}
\newtheorem{thm}{Theorem}[section]
\newtheorem{lem}[thm]{Lemma}
\newtheorem{prop}[thm]{Proposition}
\newtheorem*{thm*}{Theorem}

\theoremstyle{definition}
\newtheorem{dfn}[thm]{Definition}

\newtheorem{cor}[thm]{Corollary}
\newtheorem{ex}[thm]{Example}

\newtheorem{notation}[thm]{Notation}
\newtheorem{rem}[thm]{Remark}

\theoremstyle{remark}

\newtheorem*{claim*}{Claim}

\newtheorem*{case*}{Case}
\newtheorem*{rem*}{Remark} 

\makeatletter \def\subsection{\@startsection{subsection}{3}
	\z@{.5\linespacing\@plus.7\linespacing}{.5\linespacing}
	{\bfseries\itshape}} \makeatother

\makeatletter \renewenvironment{proof}[1][\proofname]{
	\par\pushQED{\qed}\normalfont
	\topsep6\p@\@plus6\p@\relax
	\trivlist\item[\hskip\labelsep\bfseries#1\@addpunct{.}]
	\ignorespaces}{
	\popQED\endtrivlist\@endpefalse} \makeatother

\def\:{\colon}
\def\.{\cdot}
\numberwithin{equation}{section}

\renewcommand{\bar}{\overline}

\newcommand{\cE}{\mathcal{E}}
\newcommand{\cF}{\mathcal{F}}
\newcommand{\cG}{\mathcal{G}}

\newcommand{\cO}{\mathcal{O}}
\newcommand{\cS}{\mathcal{S}}

\newcommand{\rA}{\mathrm{A}}
\newcommand{\Ast}{\mathrm{A}^{\star}}
\newcommand{\rD}{\mathrm{D}}
\newcommand{\rE}{\mathrm{E}}

\newcommand{\trp}{^{\scriptscriptstyle{\top}}}
\newcommand{\ind}{\operatorname{ind}}

\usepackage{xcolor}

\def\mono{\hookrightarrow}

\def\8{\infty}

\newcommand{\cp}[1]{^{(#1)}}
\newcommand{\cha}{\operatorname{char}}

\author{Karol Palka}
\thanks{The author was supported by the National Science Centre, Poland, grants number 2015/18/E/ST1/00562 and 2021/41/B/ST1/02062. For the purpose of Open Access, the author has applied a CC-BY public copyright license to any Author Accepted Manuscript version arising from this submission.}
\address{Institute of Mathematics, Polish Academy of Sciences, \'{S}niadeckich 8, 00-656 Warsaw, Poland}	\email{palka@impan.pl}

\begin{document}
\title[Almost minimal models of log surfaces]{Almost minimal models of log surfaces}
\begin{abstract} 
We generalize Miyanishi's theory of almost minimal models of log smooth surfaces with reduced boundary to the case of arbitrary log surfaces defined over an algebraically closed field. Given an MMP run of a log surface $(X,D)$ we define and construct its almost minimal model, whose underlying surface has singularities not worse than $X$ and which differs from a minimal model by a contraction of some curves supported in the boundary only. For boundaries of type $rD$, where $D$ is reduced and $r\in [0,1]\cap \mathbb{Q}$, we show that if $X$ is smooth or $r\in [0,\frac{1}{2}]$ then the construction respects $(1-r)$-divisorial log terminality and $(1-r)$-log canonicity. We show that the assumptions are optimal, too.
\end{abstract}

\maketitle

\section{Main results}

An important tool in the study of quasi-projective surfaces and of log surfaces is the logarithmic version of the Minimal Model Program (MMP) \cite{KollarMori-bir_geom}, which finds a birational model whose log canonical divisor has uniform numerical properties. It is well known that in dimensions bigger than two a minimal model of a smooth projective variety can be singular \cite[Example 3.1.3]{Matsuki_MMP_intro}. The same problem appears naturally for quasi-projective surfaces and log surfaces with nonzero boundary \cite[Example 3.49]{KollarMori-bir_geom}. Whenever the MMP is used to understand the geometry of a smooth variety, passing to a singular model too quickly makes the analysis more difficult. To avoid this, for log smooth surfaces with reduced boundary Miyanishi developed the notion of an \emph{almost minimal model} \cite[2.3.11]{Miyan-OpenSurf}. It is related to a minimal model by a well described morphism, which we call a \emph{peeling}, see Definition \ref{def:peel_squeeze}, contracting only some curves supported in the boundary. For a log smooth surface with reduced boundary, an almost minimal model, unlike a minimal model, has always a smooth underlying surface. Moreover, one proves that an almost minimal model it is in fact log smooth. Understanding the process of almost minimalization gives an effective tool to analyze log surfaces. In particular, it helped to obtain various structure theorems, see \cite[\S 2-3]{Miyan-OpenSurf}.

We show that the idea of almost minimalization can be used more widely. Given an MMP run $f\:(X,D) \to (\ov X, \ov D)$ on a log surface defined over an algebraically closed field of arbitrary characteristic we define its \emph{almost minimalization} $f^\#$ as the unique $K_X$-MMP over $\ov X$; we call $(f^\#X,f^\#_*D)$ an \emph{almost minimal model} of $(X,D)$. By construction, when measured in terms of log discrepancies, $f^\#X$ is not more singular than $X$. In particular, if $X$ is smooth then $f^\#X$ is smooth. However, describing log singularities of an almost minimal model requires a detailed analysis of the almost minimalization morphism. This amounts to the analysis of reordering of contractions of log exceptional curves. For this we introduce general notions of peeling, squeezing, redundant and almost log exceptional curves, see Section \ref{ssec:peeling_squeezing_alm_min}. The non-almost-minimality of a log surface is witnessed by the existence of an almost log exceptional curve, necessarily not supported in the boundary, see Corollary \ref{cor:nef_gives_peeling}. Almost log exceptional curves keep some geometric properties of log exceptional curves, which makes them well-behaved and important. In particular, their intersection with the boundary is well controlled.

We extend the theory to the case of MMP runs of the second kind, in which curves intersecting the log canonical divisor trivially may be contracted, too; see Section \ref{ssec:2nd_kind}. This extension requires additional care, as contractions of log exceptional curves of the second kind may destroy $\Q$-factoriality and may lead to non-algebraic surfaces.

As an application of these general ideas we work out several characterizations of redundant boundary components and of almost log exceptional curves for log surfaces with uniform boundaries, that is, the ones of type $rD$, where $D$ is reduced and $r\in [0,1]\cap \Q$, see Section \ref{ssec:Redundant_and_ALE_for_uniform}. It is known that if $(X,rD)$ is $(1-r)$-divisorially log terminal ($(1-r)$-dlt) or $(1-r)$-log canonical ($(1-r)$-lc), see Definition \ref{dfn:eps-dlt}, then so is its image under the contraction of every log exceptional curve, hence its minimal model, too. In general these properties of log singularities are not respected by the process of almost minimalization, see Example \ref{ex:(1-r)-log_terminality}. Nevertheless, we prove that they are inherited by almost minimal models in case $X$ is smooth or $r\in [0,\frac{1}{2}]\cup \{1\}$. For $X$ singular and $r\in (\frac{1}{2},1)$ we construct counterexamples; see Examples \ref{ex:aMM_not_(1-r)dlt} and \ref{ex:aMM_not_(1-r)dlt_2}. For the notion of \emph{intermediate} models see Definition \ref{dfn:MMP1}. 

\begin{thm}\label{thm:aMM_respects_(1-r)-dlt}
Let $(X,D)$ be a log surface with a reduced boundary and let $r\in [0,1]\cap \Q$. Assume that $(X,rD)$ is $(1-r)$-lc and that $X$ is smooth or $r\leq \frac{1}{2}$. Then the following hold.
\begin{enumerate}[(1)]
\item Every almost minimal model of $(X,rD)$ of the second kind is $(1-r)$-lc.
\item If $(X,rD)$ is $(1-r)$-dlt  then every almost minimal model of $(X,rD)$ is $(1-r)$-dlt.
\item If $X$ is smooth and $\frac{1}{r}\in \N\cup \{\8\}$ then every almost minimalization of $(X,rD)$ of the second kind can be decomposed into elementary contractions so that every intermediate model is $(1-r)$-lc.
\end{enumerate}
\end{thm}

Part (2) of the theorem  for $r=1$ implies the following result by Miyanishi, see \cite[p.\ 105]{Miyan-OpenSurf}. 

\begin{cor}\label{cor:aMM_for_log_smooth}
An almost minimal model of a log smooth surface with a reduced boundary is log smooth.
\end{cor}

\begin{proof} Let $(X,D)$ be a log smooth surface with a reduced boundary. Since $(X,0)$ is terminal, by Lemma \ref{lem:negativity} and by the definition of almost minimalization the underlying surface of an almost minimal model of $(X,D)$ is terminal, hence smooth. Since $(X,D)$ is dlt, Theorem \ref{thm:aMM_respects_(1-r)-dlt}(2) implies that an almost minimal model is dlt, hence log smooth.
\end{proof}

The new parameter $r$ gives additional flexibility to the theory of almost minimal models. Previously we treated the case $r=\frac{1}{2}$ \cite{Palka-minimal_models}. This instance of the construction turned out to be especially useful for the difficult class of rational surfaces of log general type, which share many properties with projective surfaces of general type but at the same time have rich birational geometry. The analysis of the process of almost minimalization for $r=\frac{1}{2}$ and affine $X\setminus D$ was a key tool in the recent proof of the Coolidge-Nagata conjecture \cite{Palka-Coolidge_Nagata1}, \cite{KoPa-CooligeNagata2} and in obtaining classification results for rational cuspidal curves \cite{PaPe_Cstst-fibrations_singularities}, \cite{PaPe_delPezzo}, \cite{KoPa-4cusps}. Recently, it allowed to obtain strong classification results for $\Q$-acyclic surfaces \cite{Pelka_Thesis}, an class of plane-like surfaces studied for a long time, see \cite[\S 3.4]{Miyan-OpenSurf}. 

\smallskip
The content of the article is as follows.  We discuss necessary properties of the MMP in the class of generalized log canonical surfaces introduced by Fujino, which contains $\Q$-factorial and log canonical surfaces. We discuss the geometry of log exceptional curves of the first and second kind; see Definitions \ref{dfn:log_exceptional}, \ref{dfn:log_exc2}. In Section \ref{ssec:relative_MMP_reordering} we discuss the process of reordering of contractions of an MMP run of the first and second kind, including a characterization of MMP runs for surfaces as birational morphisms increasing discrepancies of the contracted curves, see Corollary \ref{cor:improving_ld_gives_MMP}. We discuss the uniqueness of almost minimalization and its properties for compositions in Sections \ref{ssec:relative_minimalization} and \ref{ssec:almost_minimalization}. To effectively describe the process we generalize Miyanishi's 'theory of peeling', defining the peeling of a boundary as a composition of a maximal sequence of contractions of log exceptional curves supported in the boundary and its images, see Definition \ref{def:peel_squeeze}. We then describe how to construct an almost minimal model in steps, successively contracting redundant components of $D$ and almost log exceptional curves, which are proper transforms of log exceptional curves, see Definition \ref{dfn:redundant_and_almost-log-exc} and Corollary \ref{cor:aMMP_algorithm}.  In Section \ref{ssec:2nd_kind} we discuss analogous properties for contractions of log exceptional curves of the second kind, that is, the ones intersecting the log canonical divisor trivially.

In Section \ref{ssec:reduced_boundaries} we review some properties of log canonical surface singularities and in Section \ref{ssec:Reduced boundary} we work out characterizations of peeling, redundant and almost log exceptional curves of the first and second kind for a reduced boundary in a generality which later allows to use it in the analysis of uniform boundaries. Since almost log exceptional and redundant curves are not necessarily log exceptional themselves, log singularities can get worse in the process of almost minimalization. For uniform boundaries we are able to give a complete description of such situations, see Section \ref{ssec:Redundant_and_ALE_for_uniform}, and hence we are able to control the behavior of log singularities under almost minimalization.

Theorem \ref{thm:aMM_respects_(1-r)-dlt} is proved in Section \ref{ssec:proof_of_Thm}, where we also discuss the assumptions in detail constructing several interesting examples. Finally, in Section \ref{ssec:half} we conclude with a complete description of the process of almost minimalization in case $r\leq \frac{1}{2}$.

In the Appendix (Section \ref{sec:d_and_ld}) we reprove some and slightly improve results by Alexeev concerning log canonical singularities and their log discrepancies.

\medskip
\textit{Acknowledgments.} We thank Tomasz Pe{\l}ka for his help in preparing pictures and for a careful reading of the manuscript.


\tableofcontents

\section{Preliminaries}\label{sec:prelim}
We work over an algebraically closed field of arbitrary characteristic. 

\subsection{Log Minimal Model Program for surfaces}\label{ssec:logMMP}

Given a normal surface, a \emph{boundary} is a Weil $\Q$-divisor whose coefficients of irreducible components are between $0$ and $1$. A \emph{log surface} $(X,D)$ consists of a normal projective surface $X$ and a boundary $D$ such that $K_X+D$ is $\Q$-Cartier, where $K_X$ denotes a canonical divisor. We use the Minimal Model Program, which works in the class of $\Q$-factorial log surfaces and in the class of log canonical log surfaces, see \cite{Fujino-logMMP_dim2}, \cite{Tanaka-MMP_char>0} and \cite{Fujino-logMMP2_char>0}, cf.\ \cite{KollarKovacs-2DlogMMP}. Recently it has been generalized to the class of GMRLC log surfaces \cite[Theorem 1.5]{Fujino-MMP_for_algebraic_log_surfaces}, which contains both of them; see Definition \ref{def:GMRLC}.

Given a log surface $(X,D)$ and a proper birational morphism from a normal surface $f\:Y\to X$, the \emph{log pullback} of $D$ is defined as the unique Weil $\Q$-divisor $D_Y$ on $Y$ such that
\begin{equation}\label{dfn:log_pullback}
K_Y+D_Y\sim_{\Q}f^*(K_X+D) \text{\ \  and\ \ } f_*D_Y=D.
\end{equation}
The exceptional divisor of $f$, denoted by $\Exc f$, is the sum of curves contracted by $f$. By a \emph{curve} we mean an irreducible and reduced variety of dimension $1$. 

\begin{dfn}[{\cite[1.4]{Fujino-MMP_for_algebraic_log_surfaces}}]\label{def:GMRLC}
Let $(X,D)$ be a log surface. If there exists a normal $\Q$-factorial algebraic surface $Y$ and a proper birational morphism $Y\to X$ such that $D_Y$ is a boundary then we say that $(X,D)$ is a \emph{generalized MR log canonical (GMRLC) surface}.
\end{dfn}

A log surface is \emph{log smooth} if $X$ is smooth and $D$ has simple normal crossings. A birational morphism $f\:Y\to X$ is a \emph{log resolution} of $(X,D)$ if $(Y,f_*^{-1}D+\Exc f)$ is log smooth. Let $f\:X\to S$ be a projective morphism from a normal surface $X$ onto an algebraic variety $S$ and let $D$ be a $\Q$-Cartier divisor on $X$. We say that $D$ is \emph{$f$-nef} if $D\cdot E\geq 0$ for every curve $E$ contracted by $f$ and we say that it is \emph{$f$-ample} if $D\cdot E>0$ for every $E\in \ov{\NE(X/Y)}\setminus \{0\}$, where $\NE(X/Y)$ is the cone of effective $1$-cycles contracted by $f$, cf.\ \cite[Theorem 1.44]{KollarMori-bir_geom}. We note that if $\dim S=1,2$ then $D$ is $f$-ample if and only if $D\cdot E>0$ for every curve $E$ contracted by $f$. We say that $D$ is \emph{$f$-semi-ample} if there exists a morphism onto an algebraic variety $g\:X\to Y$ over $S$ such that $D\sim g^*A$ for some $f\circ g^{-1}$-ample $\Q$-Cartier divisor $A$ on $Y$ (cf.\ \cite[Lemmas 4.13, 4.14]{Fujino_Fundamentals_logMMP}). We will use the following facts.

\begin{lem}[{\cite[4.3, 4.4]{Fujino-MMP_for_algebraic_log_surfaces}}] \label{lem:GMRLC} Let $(X,D)$ be a GMRLC surface.
\begin{enumerate}[(1)]
\item If $f\:X\to Z$ is a proper birational morphism onto a normal surface such that  $-(K_X + D)$ is $f$-nef then $(Z, f_*D)$ is a GMRLC log surface.
\item If $D'$ is a boundary such that $D'\leq D$ then $(X,D')$ is a GMRLC log surface. In particular, $K_X$ and all components of $D$ are $\Q$-Cartier.
\end{enumerate}
\end{lem}

By a \emph{contraction} we mean a morphism between normal varieties with connected fibers. Writing a birational contraction between log surfaces as $f\:(X,D)\to (X',D')$ we always assume that $D'=f_*D$. 

\begin{dfn}[A log exceptional curve]\label{dfn:log_exceptional} Let $(X,D)$ be a log surface. A curve $\ll\subseteq X$ is \emph{log exceptional} if
\begin{equation}\label{eq:log_exc}
\ll\cdot (K_{X}+D)<0 \quad\mbox{and}\quad \ll^{2}<0.
\end{equation}
\end{dfn}

Note that although $\ll$ is only a Weil divisor, the self-intersection number $\ll^2$ is well defined, see Remark \ref{rem:singularities}. By the logarithmic contraction theorem \cite[Theorem 5.5]{Fujino-MMP_for_algebraic_log_surfaces} there exists a birational contraction of $\ll$, which we denote by $\ctr_\ll$. It is known that $\ll$ is a rational curve and $\ll\cdot (K_{X}+D)\geq -2$ \cite[Theorem 5.6]{Fujino-MMP_for_algebraic_log_surfaces}. We emphasize that the cone theorem and the contraction theorem hold for arbitrary projective log surfaces, also in the relative form. In this generality a log exceptional curve does not have to be $\Q$-Cartier, hence the direct image of a log canonical divisor can be non-$\Q$-Cartier, see Example \ref{ex:log_exc_non_Q-Cartier}.

\begin{rem}[Direct images of $\Q$-Cartier divisors] \label{rem:Q-Cartier_under_contraction} \ 
Let $\ll$ be a log exceptional curve on a log surface $(X,D)$ and let $f$ be its contraction. The following hold.
\begin{enumerate}[(1)]
\item If $\ll$ is $\Q$-Cartier then $f_*$ and $f^{-1}_*$ map $\Q$-Cartier divisors to $\Q$-Cartier divisors.
\item $\ll$ is $\Q$-Cartier if and only if $K_{f(X)}+f_*D$  is $\Q$-Cartier.
\item If $(X,D)$ is GMRLC then $\ll$ and $K_X+\ll$ are $\Q$-Cartier.
\end{enumerate}
\end{rem}

\begin{proof}Put $X'=f(X)$ and $D'=f_*D$. 

(1) If $G$ is $\Q$-Cartier on $X'$ then $f^*(G)$, and hence $f^{-1}_*G$ is $\Q$-Cartier on $X$. Assume $G$ is a $\Q$-Cartier divisor on $X$. Then $G+a\ll$ for some $a\in\Q$ is $\Q$-Cartier and intersects $\ll$ trivially, hence by the contraction theorem \cite[Theorem 5.5(iii)]{Fujino-MMP_for_algebraic_log_surfaces} $G+a\ll=f^*C$ for some $\Q$-Cartier divisor $C$ on $X'$. Then $f_*G=C$ is $\Q$-Cartier.

(2) Since $K_X+D$ is $\Q$-Cartier, the linear equivalence $K_X+D\sim f^*(K_{X'}+D')+a\ll$, where $a>0$, implies that $\ll$ is $\Q$-Cartier if and only if $f^*(K_{X'}+D')$ is $\Q$-Cartier, which by (1) holds if and only if $K_{X'}+D'$ is $\Q$-Cartier.

(3) By Lemma \ref{lem:GMRLC} the divisor $K_{X'}+D'$ is $\Q$-Cartier, hence (2) implies that $\ll$ is $\Q$-Cartier. By Lemma  \ref{lem:GMRLC}(2) $K_X$ is $\Q$-Cartier, hence $K_X+\ll$ is $\Q$-Cartier.
%
\end{proof}

If $\ll$ is a log exceptional curve on a GMRLC log surface $(X,D)$ then, since the coefficient of $\ll$ in $D$ is at most $1$ and $D\geq 0$, we get $\ll\cdot (K_X+\ll)<0$. In case the log surface is log canonical or $\Q$-factorial we have $\ll\cong\P^1$, see \cite[Theorem 3.19]{Tanaka-MMP_char>0}, cf.\ \cite[Lemma 2.3.5]{KollarKovacs-2DlogMMP}. 

\begin{dfn}[A run of a Minimal Model Program]\label{dfn:MMP1}
A \emph{partial MMP run on a log surface $(X,D)$} (a \emph{partial $(K_X+D)$} -MMP) is a birational morphism $\psi\:(X,D)\to (\ov X,\ov D)$ which can be decomposed as a sequence of birational contractions between log surfaces
\begin{equation}\label{eq:MMP}
(X,D)=(X_1,D_1)\xrightarrow{\phi_1}\ldots \xrightarrow{\phi_n}(X_{n+1},D_{n+1})=(\ov X,\ov D)
\end{equation}
such that each $\phi_i$ is a contraction of a log exceptional curve on $(X_i,D_i)$. An MMP run is a maximal partial MMP run. We call $\varphi_i$ the \emph{elementary contractions} and $(X_i,D_i)$ the \emph{intermediate models} of (this decomposition of) $\psi$.
\end{dfn}

By Remark \ref{rem:Q-Cartier_under_contraction}(2) if $(X_i,D_i)$ is a log surface then $(X_{i+1},D_{i+1})$ is a log surface if and only if $\Exc \varphi_i$ is $\Q$-Cartier. Hence every curve contracted by any partial MMP run is $\Q$-Cartier. An MMP run is \emph{complete} if the resulting log surface is \emph{minimal}, that is, it contains no log exceptional curve.

\medskip
Assume that $(X,D)$ is GMRLC. Then every log exceptional curve is automatically $\Q$-Cartier by Remark \ref{rem:Q-Cartier_under_contraction}(3). Its contraction leads to a GMRLC log surface by Lemma \ref{lem:GMRLC}(1). It follows that every MMP run on a GMRLC log surface is complete. Moreover, if  $(X,D)$ is $\Q$-factorial (respectively, log canonical) then each $(X_i,D_i)$ is $\Q$-factorial (respectively log canonical). By \cite[Theorem 1.5]{Fujino-MMP_for_algebraic_log_surfaces} on a minimal GMRLC log surface either the log canonical divisor is semi-ample or its negative is $f$-ample for some contraction of positive relative dimension and relative Picard rank $1$.

\begin{rem}[Projectivity]\label{rem:projectivity}
We note that a $\Q$-factorial algebraic surface is quasi-projective \cite[Lemma 2.2]{Fujino-logMMP_dim2} and that a complete normal surface whose singularities are contained in one affine open subset is projective \cite[Corollary 4]{Kleiman-ampleness}. Also, every log terminal surface is $\Q$-factorial, see Remark \ref{rem:singularities}. 
\end{rem}

\medskip
\subsection{Coefficients and discrepancies}\label{ssec:discrepancies}
Singularities of log surfaces and their changes under a run of an MMP are conveniently measured in terms of log discrepancies, see \cite[\S 2.3]{KollarMori-bir_geom}.

Given a Weil divisor $D$ and its irreducible component $E$ we denote by $\coeff_E(D)$ the coefficient of $E$ in the irreducible decomposition of $D$. Given a log surface $(X,D)$ and a proper birational morphism from a normal surface $f\:Y\to X$, for a prime divisor $E$ on $Y$ we define the \emph{coefficient} of $E$ over $(X,D)$ and the \emph{log discrepancy} of $E$ over $(X,D)$, as, respectively (see \eqref{dfn:log_pullback})
\begin{equation}\label{eq:coefficient}
\cf(E;X,D)=\coeff_E(D_Y)\quadtext{and} \ld(E;X,D)=1- \coeff_E(D_Y).
\end{equation}
They depend only on the valuation of the field of rational functions on $X$ associated with $E$, not on $f$. Let $\cE(f)$ denote the set of curves contracted by $f$. Then $\Exc f=\sum_{E\in\cE(f)} E$. We have a linear equivalence over $\Q$:
\begin{equation}\label{eq:discrepancies}
K_Y+f_*^{-1}D+\sum_{E\in\cE(f)} \cf(E;X,D)E\sim f^*(K_X+D).
\end{equation}
We write $\cf(E)$ instead of $\cf(E;X,D)$ if $X$ and $D$ are clear from the context. We put
\begin{equation}\label{eq:ld_Y_and_cf_Y}
\cf_Y(X,D)=\sum_{E\in \cE(f)} \cf(E;X,D)E \text{\quad and \quad} \ld_Y(X,D)=\sum_{E\in \cE(f)} \ld(E;X,D)E
\end{equation}
and we call $\cf_Y(X,D)$ the \emph{coefficient divisor}. A \emph{divisor over $X$} is a divisor on some $Y$ as above. It is \emph{exceptional} if its image on $X$ has codimension bigger than $1$. We define the coefficient and the log discrepancy of $(X,D)$ as
\begin{equation}\label{eq:ld}
\cf(X,D)=\sup\{\cf(E;X,D):E\text{\ is exceptional over\ } X\} \text{\quad and \quad} \ld(X,D)=1-\cf(X,D).
\end{equation}
The \emph{total coefficient $(X,D)$} is
\begin{equation}\label{eq:tld}
\tcf(X,D)=\sup\{\cf(E;X,D):E\text{\ is a divisor over\ } X\}.
\end{equation}
Thus $\tcf(X,D)=\max(\cf(X,D), \{\cf(E)\: E\text{\ is a component of\ } D \})$. We put $\cf(X)=\cf(X,0)$ and $\ld(X)=\ld(X,0)$. For a set $\cS$ of divisors over $X$ we put $\cf(\cS;X,D)=\sup \{\cf(E;X,D)\: E\in \cS\}$.

\begin{dfn}[$\epsilon$-lc surfaces]\label{dfn:eps-dlt}
Let $\epsilon\in [0,1]\cap \Q$. A log surface $(X,D)$ is \emph{$\epsilon$-log canonical} ($\epsilon$-lc) if $\tcf(X,D)\leq 1-\epsilon$. It is \emph{$\epsilon$-divisorially log terminal} ($\epsilon$-dlt) if it is $\epsilon$-log canonical and $\cf(\cE(f))<1-\epsilon$ for some log resolution $f$.
\end{dfn}

In particular, if $(X,D)$ is $\epsilon$-lc then the coefficients of components of $D$ do not exceed $1-\epsilon$. However, even if $(X,D)$ is $\epsilon$-dlt, the boundary $D$ can have components with coefficient equal to $1-\epsilon$. Note also that for $\epsilon\neq0$ an $\epsilon$-lc log surface is log terminal, hence $\Q$-factorial, see Remark \ref{rem:singularities}.

\begin{rem}\label{rem:dlt_resolution_dependence}
When we blow up the point of intersection of two components of a boundary of a log smooth surface with log discrepancies $u_1$ and $u_2$ then the new exceptional curve has log discrepancy $u_1+u_2$. It follows that the supremum in \eqref{eq:tld} can be computed on the set consisting of components of $D$ and exceptional curves of any chosen log resolution. Hence $\epsilon$-divisorial log terminality for $\epsilon\neq 0$ and $\epsilon$-log canonicity can be verified using any log resolution. For $\epsilon=0$ this is not the case, as the exceptional curve of a blowup of a log smooth surface with reduced boundary at a point belonging to two boundary components has $\ld=0$.
\end{rem}

\begin{samepage}
\begin{rem}[Numerically lc surfaces and their singularities]\label{rem:singularities}\nopagebreak  \ 
\begin{enumerate}[(1)]
\item Let $X$ be a normal surface. Since the intersection matrix of the exceptional divisor of resolution of singularities of $X$ is negative definite \cite{Mumford-surface_singularities}, every Weil $\Q$-divisor has a uniquely determined pullback intersecting trivially all exceptional curves. This determines uniquely a $\Q$-valued intersection product of Weil $\Q$-divisors on $X$ consistent with the projection formula. Then the formula \eqref{eq:discrepancies} defining coefficients extends to pairs $(X,D)$, where $D$ is a boundary, for which $K_X+D$ is just a Weil $\Q$-divisor. In particular, it allows to define \emph{numerically dlt (numerically lc)} surfaces as the ones for which coefficients of some log resolution are smaller (smaller or equal) than $1$, see \cite[\S 4.1]{KollarMori-bir_geom}. However, it turns out that in both cases $K_X+D$ is in fact $\Q$-Cartier, see \cite[4.11]{KollarMori-bir_geom}, \cite[6.3]{Fujino-logMMP2_char>0}. The proof uses classification of
graphs of minimal resolutions of numerically lc surfaces - see \cite[\S 3]{Flips_and_abundance} and \cite[\S 4.1]{KollarMori-bir_geom}.

\item Dlt surfaces have rational singularities and are $\Q$-factorial, see \cite[4.11, 4.12]{KollarMori-bir_geom}, \cite[6.3, 6.4]{Fujino-logMMP2_char>0}. Knowing the resolution graphs, rationality can be in fact verified directly \cite[Theorem 3]{Artin-Rational_sing} and then the $\Q$-factoriality follows by \cite[17.1]{Lipman-rational_sing_2d}, too.
\end{enumerate}
\end{rem}
\end{samepage}

We need the following results. Note that for surfaces intersections of Weil divisors with curves are well defined by Remark \ref{rem:singularities}, so the numerical definitions of $f$-nef and $f$-ample divisors given in Subsection \ref{ssec:logMMP} extend to Weil $\Q$-divisors. We remark that the two lemmas below work in any dimension under the assumption that  $K_X+D$, $K_{X'}+f_*D$  and $A$ are $\Q$-Cartier. For surfaces these assumptions can be dropped.

\begin{lem}[Negativity lemma {\cite[3.39]{KollarMori-bir_geom}}] \label{lem:negativity}
Let $f\:X\to X'$ be a proper birational morphism between normal surfaces and let $A$ be a $\Q$-divisor on $X$. If $f_*A$ is effective and $-A$ is $f$-nef then $A$ is effective and $\Supp A=f^{-1}(\Supp f(A))$.
\end{lem}

We note that given a proper birational morphism between normal surfaces $f\:X\to X'$ and a $\Q$-divisor $D$, putting $D'=f_*D$ we can rewrite \eqref{eq:discrepancies} as
\begin{equation}\label{eq:ramification}
(K_X+D)-f^*(K_{X'}+D')=\sum_{E\in\cE(f)} a_E E \text{\ \ where } a_E=\cf(E;X,D)-\cf(E;X',D').
\end{equation}

Applying the above lemma to $A=\sum_{E\in\cE(f)} a_E E$ one obtains the following result.

\begin{lem}[{\cite[3.38]{KollarMori-bir_geom}}]\label{lem:ld_increasaes}
Let $f\:X\to X'$ be a proper birational morphism between normal surfaces and let $D$ be a $\Q$-divisor on $X$.
If $-(K_X+D)$ is $f$-nef then for every exceptional divisor $E$ over $X'$ we have $$\cf(E;X',f_*D)\leq \cf(E;X,D).$$ If additionally $-(K_X+D)$ is $f$-ample and $f$ is not an isomorphism over a general point of the center of $E$ on $X'$ then the inequality is strict.
\end{lem}

\begin{cor}[$f$-nef canonical divisor]\label{cor:pure_peeling}
Let $f\:X\to Y$ be a birational morphism between normal surfaces such that $K_X$ is $f$-nef. Then $f^*K_Y\sim K_X+ E$ for some effective divisor $E$ such that $\Supp E\subseteq \Supp \Exc f$. Moreover, if $\Exc f$ is connected then $\Supp E=\Supp \Exc f$, unless $E=0$.
\end{cor}

\begin{proof}Apply Lemma \ref{lem:negativity} to $A=f^*K_Y-K_X$.
\end{proof}

Lemmas \ref{lem:ld_increasaes} and \ref{lem:GMRLC}(1) give the following corollary.

\begin{cor}\label{lem:eps-lc_is_respected}
Let $(X,D)$ be an $\epsilon$-lc log surface for some $\epsilon\geq 0$ and $f\:X\to X'$ a birational morphism onto a normal surface such that $-(K_X+D)$ is $f$-nef. Then $(X,f_*D)$ is an $\epsilon$-lc log surface. If $(X,D)$ is $\epsilon$-dlt and $-(K_X+D)$ is $f$-ample then $(X',f_*D)$ is $\epsilon$-dlt. 
\end{cor}

It follows from Corollary \ref{lem:eps-lc_is_respected} that the outcome of a partial MMP run on an $\epsilon$-dlt log surface is $\epsilon$-dlt. Example \ref{ex:discrepancies_drop} reminds that in the above corollary it is important to have an upper bound on boundary coefficients, as a contraction of a boundary component may increase the coefficient (but not the total coefficient) of the log surface.

\begin{ex}[The coefficient of the log surface may increase under MMP] \label{ex:discrepancies_drop} 
Let $X=\F_n:=\P(\cO_{\P^1}\oplus \cO_{\P^1}(n))$ for $n\geq 3$. Let $X\to X'$ be the contraction of the negative section $D$. We have $(K_X+D)\cdot D<0$, $f_*D=0$ and $\cf(D;X',f_*D)=1-\frac{2}{n}<\coeff_D(D)$, in agreement with Lemma \ref{lem:ld_increasaes}. At the same time we have $\cf(X',f_*D)=1-\frac{2}{n}>0=\cf(X,D)$. 
\end{ex}

\begin{samepage}
\begin{ex}[Non-GMRLC log surfaces] \label{ex:non_GMRLC} \
\begin{enumerate}[(1)]
\item Let $\ov X$ be a projective cone over a smooth normally embedded curve $\bar E\subseteq \P^n$ of genus $g\geq 1$. Then $\ov X$ is GMRLC if an only if $g=1$. To see this let $\pi\:X \to \ov X$ be the minimal resolution and let $E$ be the exceptional curve. It follows from \cite[\S 2, Remark]{Samuel_anneaux_factoriels} that the local ring of the vertex of the cone is $\Q$-factorial if and only if the restriction homomorphism $\Cl(\P^n)\otimes \Q\to \Cl(\ov E)\otimes \Q$ is surjective. The latter fails for curves of positive genus, as then ${\dim (\Cl(\ov E)\otimes \Q)=1+\dim_{\Q} (J(\bar{E})\otimes \Q)}$, where $J(\bar{E})$ is the Jacobian variety of $\ov E$ (having dimension equal to the genus of $\ov E$), viewed as an abelian group. The latter group is divisible \cite[p.\ 62]{Mumford-abelian_var}, in particular it has nonzero rank; cf.\ \cite[0BA0]{stacks-project}. Thus $\pi$ is a minimal $\Q$-factorialization of $\ov X$. Writing $\pi^*K_{\ov X}=K_X+(1+u)E$ and using adjunction we compute $u=\frac{2}{-E^2}(g-1)$, which is positive for $g>1$ and zero for $g=1$.

\item Let $\pi\:X \to \ov X$ be the minimal resolution of a projective cone over an elliptic curve and let $E$ be the exceptional curve. Let $\ov D$ be a nonzero sum of lines through the vertex. Choosing the lines over torsion points of the elliptic curve we may assure that the components of $\ov D$ are $\Q$-Cartier. Since $\pi^*K_{\ov X}=K_X+E$, the cone is GMRLC. Put $\ov D'=\pi_*^{-1}\ov D$. We have $\pi^*(K_{\ov X}+\ov D)=K_X+E+\ov D'+eE$, where $e=E\cdot \ov D'/(-E^2)>0$, so $E$ appears in $\pi^*(K_{\ov X}+\ov D)$ with coefficient bigger than $1$. It follows that $(\ov X,\ov D)$ is not GMRLC.
\end{enumerate}
\end{ex}
\end{samepage}

\medskip
\subsection{Curves of the second kind}\label{ssec:second_kind}

\begin{dfn}\label{dfn:log_exc2}
A curve $\ll\subseteq X$ is \emph{log exceptional of the second kind} on $(X,D)$ if 
\begin{equation}\label{eq:log_exc2}
\ll\cdot (K_{X}+D)=0\quad\mbox{and}\quad \ll^{2}<0.
\end{equation}
\end{dfn}
Such curves appear in the construction of a log canonical model.  By \cite[Theorem 3.19]{Tanaka-MMP_char>0}, on log canonical and on $\Q$-factorial surfaces they are either isomorphic to $\P^1$ or
satisfy $\cO(m(K_X+\ll))|_{\ll}\cong \cO_{\ll}$ for every positive $m$ for which $m(K_X+\ll)$ is Cartier. As an example, \cite[2.4.4-12]{Miyan-OpenSurf} gives a description of log exceptional curves of the second kind on a log smooth surface with reduced snc-minimal boundary.

\begin{rem}\label{rem:2nd_type_Q-Cartier}
A log exceptional curve $\ll$ of the second kind on a GMRLC log surface $(X,D)$ is $\Q$-Cartier, unless $\ll\cap D=\emptyset$ and $\ll\cdot K_X=0$.
\end{rem}

\begin{proof}
By Lemma \ref{lem:GMRLC}(2) $X$ is GMRLC and we may assume that $\ll$ is not a component of $D$. We have $\ll\cdot K_X=-\ll\cdot D\leq 0$. In particular, if $\ll\cdot K_X=0$ then $\ll$ is disjoint from $D$. We may therefore assume that $\ll\cdot K_X<0$, hence $\ll$ is log exceptional on $(X,0)$. By Remark \ref{rem:Q-Cartier_under_contraction}(3) $\ll$ is $\Q$-Cartier.
\end{proof}

\begin{ex}[Non-$\Q$-Cartier log exceptional curve of the second kind]\label{ex:elliptic cone}

Let $X\to \bar X$ be the blowup of a smooth point $x\in \bar X$ of a projective cone over an elliptic curve $\bar E\subseteq \P^2$. Denote the line through $x$ on $\ov X$ by $\ov\ll$. Let $\pi\:Y\to X$ be the minimal resolution of singularities and let $E$ denote the unique exceptional curve. By \cite[V.2.11.4]{Hartshorne_AG} and \cite[Prop V.2.9]{Hartshorne_AG} we have $E^2=-\deg \ov E=-3$. We compute $\pi^*K_{X}=K_Y+E$ and $\pi^*\ll=\ll'+\frac{1}{3}E$, where $\ll'$ is the proper transform of $\ll$, hence $\ll\cdot K_X=0$ and $\ll^2=(\ll')^2+\frac{1}{3}=-\frac{2}{3}<0$. By adjunction the canonical divisor of the cone is $\Q$-Cartier. The curve $\ll$ is log exceptional on $(X,0)$ of the second kind. Let $\eta\:\bar X\setminus\{q\}\to \bar E$, where $q\in \bar X$ is the vertex, be the projection. If $\ov x:=\eta(x)\in \bar E$ is non-torsion then $\ov \ll$, and hence $\ll$, is not $\Q$-Cartier. 

Let us recall an argument for the latter claim. Assume the line $\ov \ll\subset \ov X$ over $\ov x\in \ov E$ is $\Q$-Cartier. If $n\ov \ll$ is Cartier for some $n>0$ then it is trivial in $\Cl(\Spec \cO_{\bar X,q})$, hence is in the kernel of the natural homomorphism $\Cl(\bar E)\to \Cl(\Spec \cO_{\bar X,q})$ induced by the projection. By \cite[Exercise II.6.3(b),(d)]{Hartshorne_AG} the latter homomorphism is surjective and the kernel is generated by the class of $\bar E\cdot L$, where $\bar E$ is identified with the hyperplane section at infinity of $\ov X\subseteq \P^3$ and $L\subseteq \P^2$ is a line. Choosing $L$ to be tangent to a flex point $o\in \bar E$ we have $\bar E\cdot L=3o$, so $n\ov x$ is linearly equivalent to a multiple of $3o$. Then $n\ov x=0$ in the group law of $\bar E$, so $\ov x$ is torsion.
\end{ex}

\begin{ex}[Non-$\Q$-Cartier log exceptional curve]\label{ex:log_exc_non_Q-Cartier}
Let the situation be as in Example \ref{ex:elliptic cone} but choose $x\in \ov X$ so that $\ov x\in \ov E$ is non-torsion. Then $\ov \ll$ is not $\Q$-Cartier. Let $\ov y=-\ov x$ in the group law of $\ov E$ and let $\ov \mm$ be the line over $\ov y$. Put $\ov D=\ov \ll+\ov \mm$. Since $\ov x+\ov y\sim 2o$, we infer that $3\ov D\sim 6\ll_o=2(3\ll_o)$, where $\ll_o$ is the line over $o$. Since $3o$ is principal, $3\ll_0$ is Cartier and hence $\ov D$ is $\Q$-Cartier. Thus $(\ov X,\ov D)$ is a log surface. Let $\ll$, $\mm$, $D$ be the proper transforms of $\ov \ll$, $\ov \mm$ and $\ov D$ on $X$. Consider the log surface $(X,D)$. As above, we have $\ll^2=-\frac{2}{3}<0$. We have $\pi^*\mm=\mm'+\frac{1}{3}E$, where $\mm'$ is the proper transform of $\mm$, hence $\ll\cdot\mm=\frac{1}{3}$ and $\ll\cdot (K_{X}+D)=-\frac{1}{3}$. Thus $\ll$ is a non-$\Q$-Cartier log exceptional curve on the log surface $(X,D)$. Despite that, by the contraction theorem there exists a birational morphism $\ctr_\ll\:X\to Z$ onto an algebraic variety contracting exactly $\ll$.
\end{ex}

We say that a curve $\ll$ on a normal projective surface $X$ is \emph{of elliptic type} if $\ll\cdot (K_X+\ll)=0$, where the intersection is computed for Weil divisors. Given a boundary divisor $D$ on $X$ we say that $\ll$ is an \emph{isolated reduced component of $D$} if $\ll$ is a connected component of $\Supp D$ such that $\coeff_\ll(D)=1$. If $(X,D)$ is GMRLC then $\ll$, being a component of $D$, is necessarily $\Q$-Cartier by Remark \ref{rem:2nd_type_Q-Cartier}.

\begin{lem}[Contractibility of log exceptional curves of the second kind]\label{lem:contractibility_2nd}
Let $\ll$ be a log exceptional curve of the first or second kind on a log surface $(X,D)$ such that $K_X$ is $\Q$-Cartier. Then one of the following holds.
\begin{enumerate}[(1)]
\item $\ll$ is $\Q$-Cartier, log exceptional on $(X,\ll)$, hence there exists a contraction $\ctr_\ll\:(X,D)\to (Y,B)$ and $(Y,B)$ is a log surface with $\Q$-Cartier $K_Y$.
\item $\ll$ is a $\Q$-Cartier isolated reduced component of $D$ of elliptic type (hence is of the second kind), 
\item $\ll$ is non-$\Q$-Cartier. If $(X,D)$ is GMRLC then $\ll$ is $K_X$-trivial and disjoint from $D$.
\end{enumerate} 
Moreover, if the contraction of $\ll$ exists then $\kappa(K_X+D)=\kappa(K_Y+B)$. If $(X,D)$ is GMRLC then $(Y,B)$ is a GMRLC log surface.
\end{lem}
\begin{proof}
Assume that $\ll$ is $\Q$-Cartier. Then $K_X+\ll$ is $\Q$-Cartier. If $\ll\cdot (K_X+\ll)<0$ then the existence of the contraction in case (1) follows from the logarithmic contraction theorem, see \cite[Theorem 5.5]{Fujino-MMP_for_algebraic_log_surfaces}. Since $\ll$ is log exceptional, its contraction maps $\Q$-Cartier divisors to $\Q$-Cartier divisors, hence $K_Y+B$ and $K_Y$ are $\Q$-Cartier. Hence we may assume that $\ll\cdot (K_X+\ll)\geq 0$. Let $c=\coeff_\ll(D)$. Then $$0\geq \ll\cdot (K_X+D)=\ll\cdot(K_X+\ll)+(1-c)(-\ll^2)+\ll\cdot (D-c\ll),$$ so since each term is non-negative, $\ll$ is an isolated reduced component of $D$ of elliptic type, which gives (2). 
Assume that $\ll$ is not $\Q$-Cartier and $(X,D)$ is GMRLC. Then by Remark \ref{rem:2nd_type_Q-Cartier} we obtain (3). 

Let $n$ be an integer such that $n(K_Y+B)$ is Cartier. Since $\pi_*(K_X+D)=K_Y+B$, we get $h^0(n(K_X+D))\leq h^0(n(K_Y+B))$. Write $\pi^*(K_Y+B)=K_X+D-a\ll$ for some $a\in \Q$. Since $\ll$ is log exceptional of the first or second kind, we have $a\geq 0$, so $h^0(n(K_Y+B))\leq h^0(n(K_X+D)-na\ll)\leq h^0(n(K_X+D))$. Thus $\kappa(K_X+D)=\kappa(K_Y+B)$.

If $(X,D)$ is GMRLC then $K_X$ is automatically $\Q$-Cartier and the log surface $(Y,B)$ is GMRLC by Lemma \ref{lem:GMRLC}(1). 
\end{proof}

An example of a curve of type (3) on a GMRLC log surface is given in Example \ref{ex:elliptic cone}. In Example \ref{ex:elliptic_contraction} we show that log exceptional curves of the second kind as in (2) and (3) in Lemma \ref{lem:contractibility_2nd} may indeed be non-contractible.

\smallskip

\begin{dfn}[A run of a Minimal Model Program of the second kind]\label{dfn:MMP2} A \emph{partial MMP run of the second kind on a log surface $(X,D)$} is a birational morphism $\psi\:(X,D)\to (\ov X,\ov D)$ which can be decomposed as a sequence of birational contractions between log surfaces \eqref{eq:MMP} such that each $\phi_i$ is a contraction of a log exceptional curve of the first or second kind on $(X_i,D_i)$ which is $\Q$-Cartier on $X_i$. An MMP run of the second kind is a maximal partial MMP run of the second kind. 
\end{dfn}

If $\psi$ is a partial MMP run of the second kind then it follows by induction with respect to the number of components of $\Exc \psi$ that components of $\Exc \psi$ are all $\Q$-Cartier. Note that even if a minimal log surface $(X,D)$ admits no nontrivial MMP run of the second kind, there can exist log exceptional curves of the second kind on $(X,D)$ which are not $\Q$-Cartier or which are $\Q$-Cartier but are either non-contractible in the category of algebraic varieties or are contractible but the direct image of a log canonical divisor is not $\Q$-Cartier, cf.\ Example \ref{ex:elliptic_contraction}.

To make a clear distinction between an MMP of the second kind and an MMP in the sense of Section \ref{ssec:logMMP}, we sometimes refer to the latter as an MMP \emph{of the first kind}. Lemma \ref{lem:ld_increasaes} gives the following corollary.

\begin{cor}[Coefficients decrease under MMP]\label{cor:cf_decrease_under_MMP}
Let $f\:(X,D)\to (X',D')$ be a partial MMP run of the first or second kind. Then
\begin{equation}\label{eq:ramification_2}
f^*(K_{X'}+D')=(K_X+D)-A, \text{\ \ where } A\geq 0 \text{\ \ and\ \ }f_*A=0.
\end{equation}
In particular, for every curve $L\subseteq X'$ we have
\begin{equation}
L\cdot (K_{X'}+D')\leq f_*^{-1}L\cdot (K_X+D).
\end{equation}
\end{cor}

If $(X,D)$ is GMRLC then by Lemma \ref{lem:GMRLC}(1) a contraction of a log exceptional curve of the second kind leads to a GMRLC log surface, hence each $(X_i,D_i)$ is GMRLC (note that the assumption that the curve is $\Q$-Cartier is redundant here).  Since a contraction of a log exceptional curve of the second kind as in Lemma \ref{lem:contractibility_2nd}(2) does not always exist, the MMP of the second kind on a GMRLC log surface may stop  even though a log exceptional curve of the second kind exists (and is $\Q$-Cartier). Since log exceptional curves of the second kind have trivial coefficients, it follows that if the initial log surface is log canonical then so is the final one. On the other hand, the process does not respect $\Q$-factoriality, which is why it is better to work with GMRLC surfaces when contracting log exceptional curves of the second kind.

\smallskip
We will now show that the exceptions in Lemma \ref{lem:contractibility_2nd}(2) may occur. We begin with the following example; part (1) is due to Hironaka, see \cite[Example 5.7.3]{Hartshorne_AG}.

\begin{ex}[Contractibility of some curves of genus $1$]\label{ex:elliptic_setup} \ 
\begin{enumerate}[(1)]
\item  Let $p_1,\ldots, p_n$ for $n\geq 10$ be distinct points on an elliptic curve $\bar E\subseteq \P^2$ and let $\sigma\: X\to \P^2$ be a blowup at all those points. Denote by $E\subseteq X$ the proper transform of the elliptic curve $\bar E$. We have $K_X+E\sim 0$ and $E^2=9-n<0$.

Assume there exists a contraction $\pi\:X\to Y$ of $E$ (onto an algebraic surface). The surface $Y$ is projective by Remark \ref{rem:projectivity}. Let $H$ be a very ample divisor on $Y$ not passing through $\bar E$ and let $\bar H=\sigma_*\pi^*H$. In the sense of intersection theory we have $\bar H\cap\bar E=\sum_{i=1}^n m_i p_i$ for some positive integers $m_1,\ldots, m_n$, which gives $\bigoplus_{i=1}^n m_i p_i=0$ in the group law of $\bar E$. Indeed, since $\bar H\sim d L$, where $d=\deg \bar H$ and $L$ is a line on $\P^2$ tangent to some flex point $o\in \bar E$, we get $\sum_{i=1}^n m_i p_i\sim d(L\cap \bar E)=3do$, hence $\sum_{i=1}^n m_i (p_i-o)\sim 0$. Therefore, the existence of the contraction implies that the points $p_1,\ldots, p_n$ are $\Z$-linearly dependent.

\item Choosing the points in (1) to be  $\Z$-linearly independent we get $E$ which is not contractible. Such a choice is possible as long as the rank of the elliptic curve, $\dim_\Q(\Pic^0(\sigma(E))\otimes \Q)$, is not smaller than $n$, for instance infinite. For algebraically closed fields the latter holds if and only if the field is not an algebraic closure of a finite field, see \cite[Theorem 10.1]{Frey-Jarden_rank_abelian_var} and \cite[Fact 2.3]{Tanaka-MMP_char>0}.

\item Let $C_d$ be curve of degree $d\geq 4$ meeting an elliptic curve $\bar E\subseteq \P^2$ normally and let $p_1,\ldots, p_n$ be the intersection points. Let $\sigma$ be as in (1). Then the linear system of  $C_d':=\sigma_*^{-1}C_d$ contains $E+(d-3)\sigma_*^{-1}L$, where $L\subseteq \P^2$ is a line not passing through the points of $C_d\cap \bar E$, hence it contains no curve in the base locus. It follows that some $|kC_d'|$, $k>0$ has no base points. Since for any curve $U\neq E$ we have $U\cdot C_d'=U\cdot E+(d-3)\sigma(U)\cdot L>0$, we see that $|kC_d'|$ induces the contraction of $E$.
\end{enumerate}
\end{ex}

\begin{ex}[Problems with contractions of curves of the second kind]\label{ex:elliptic_contraction}\ 
\begin{enumerate}[(1)]
\item Let $(X,E)$ and $p_1,\ldots, p_n$ be as in Example \ref{ex:elliptic_setup}(2). Then $(X,E)$ is log smooth, $E$ is $\Q$-Cartier and log exceptional of the second kind, but is not contractible (in the category of algebraic varieties); cf.\ Lemma \ref{lem:contractibility_2nd}(2).

\item A contraction of a log exceptional curve of the second kind may destroy $\Q$-factoriality. Let $X\to Y$ be the minimal resolution of a projective cone over an elliptic curve $\bar E\subseteq \P^2$ and let $E$ be the exceptional divisor. Being smooth, $X$ is $\Q$ factorial. Clearly, $E\cong \bar E$ and $E\cdot (K_X+E)=0$ by adjunction, so $E$ is log exceptional of the second kind (and $\Q$-Cartier). Moreover, $Y$ is not $\Q$-factorial, because the line over any non-torsion point of $\bar E$ is not $\Q$-Cartier. Note that despite $Y$ being non $\Q$-factorial, $K_Y$ is $\Q$-Cartier by adjunction, \cite[Theorem 5.50]{KollarMori-bir_geom}.

\item A log exceptional curve of the second kind on a GMRLC surface with no boundary may be non-contractible. Let $d\geq 1$ and let $p_1,\ldots,p_{3d-1}\in \bar E\subseteq \P^2$ be $\Z$-linearly independent points on an elliptic curve. We claim that may choose and irreducible curve $C_d$ of degree $d$ meeting $\bar E$ normally and passing through $p_i$, $i=1,\ldots, 3d-1$.

To see this we note that the dimension of the space of planar curves $C_d$ of degree $d$ passing through all $p_i$, $i<3d$ is at least $\binom{d+3-1}{3-1}-1-(3d-1)=\frac{1}{2}(d-1)(d-2)$. Suppose that $C_d$ is reducible and $\bar E$ is not one of its components. The intersection cycle $C_d\cap \bar E$ is not supported in $\{p_1,\ldots, p_{3d-1}\}$, as the points would be $\Z$-linearly dependent by the claim in Example \ref{ex:elliptic_setup}(1). Write $C_d\cap \bar E=\sum_{i=1}^{3d}m_i p_i$, where $p_{3d}\in \bar E$ and $m_i$ are positive integers. Since $\deg (C_d\cap \bar E) =3d$, we get $m_i=1$ for each $i$. Since $m_{3d}=1$, one of the components of $C_d$ does not pass through $p_{3d}$, which gives a dependence relation for $\{p_1,\ldots, p_{3d-1}\}$; a contradiction. Therefore, if $C_d$ is reducible then $\bar E$ is one of its components, hence the space of reducible $C_d$ has dimension $\binom{d-3+2}{2}-1=\frac{1}{2}(d-3)d$, hence has positive codimension.

We have $C_d\cap \bar E=\sum_{i=1}^{3d}p_i$, where all $p_i\in \bar E$ are distinct. Assume that $d\geq 4$. Let $\sigma\: X\to \P^2$ and $E$ be as in Example \ref{ex:elliptic_setup}(1) and let $\pi\:X\to Y$ be the contraction of $E$; it exists by Example \ref{ex:elliptic_setup}(3). Let $\ll'\subseteq X$ be the exceptional curve over $p_{3d}$ and let $\ll=\pi(\ll')$. Since $-(K_X+E)$ is $\pi$-nef, $(Y,0)$ is GMRLC by Lemma \ref{lem:GMRLC}(1), so $K_Y$ is $\Q$-Cartier. We have $\pi^*K_Y=K_X+E$ and $\pi^*\ll=\ll'+\frac{1}{e}E$, where $e=\frac{1}{-E^2}=\frac{1}{3d-9}>0$. It follows that $\ll\cdot K_Y=0$ and $\ll^2=\frac{1}{e}-1<0$, so $\ll$ is a log exceptional curve of the second kind on $(Y,0)$. It is easy to see that $\ll$ cannot be contracted. Indeed, if there exists a contraction $\eta\:Y\to Y'$ of $\ll$ onto a complete (algebraic) surface then by Remark \ref{rem:projectivity} $Y'$ has a very ample irreducible divisor $A$ not passing through the image of $\ll$ and then the intersection of $\sigma_*\pi^*\eta^*A$ and $\bar E$ is supported on $p_1,\ldots,p_{3d-1}$. But this gives a non-trivial dependence relation between the latter points in the group law of the cubic, contrary to the assumption. We conclude also that $\ll$ is not $\Q$-Cartier, as otherwise $\ll$, being a log exceptional curve of the first kind on $(X,\ll)$, could be contracted according to the contraction theorem. 
\end{enumerate}
\end{ex}


\section{Almost minimalization}

\subsection{Relative MMP and reordering contractions}\label{ssec:relative_MMP_reordering}

Although a run of an MMP on a log surface $(X,D)$ improves log singularities (Lemma \ref{lem:ld_increasaes}), singularities of the underlying surface $X$ may easily get worse for contractions of exceptional curves contained in the boundary. For instance, even if $X$ is smooth, its image may be singular, see Example \ref{ex:peeling}. Passing to almost minimal models allows to delay such contractions. Before going into details, we need some preparations.

\begin{lem}[Sub-contractions]\label{lem:sub_contractions}
Let $f_i\:X\to X_i$ for $i=1,2$ be birational contractions between normal surfaces such that $\Exc f_1\leq \Exc f_2$. Then $f_2\circ(f_1)^{-1}\:X_1\map X_2$ is a birational contraction.
\end{lem}

\begin{proof}
We may assume that $\Exc f_2$ is connected, so its image is a point. Hence we may assume that $X_2$ is affine. The rational map $f_2\circ(f_1)^{-1}$ is defined off the image of the exceptional divisor of $f_1$, which has codimension $2$, hence is regular by \cite[Corollary 11.4]{Eisenbud_comm-alg}. 
\end{proof}

Recall that given a log variety $(X,D)$ and a projective morphism of normal varieties $f\:X\to Y$ one defines the \emph{relative MMP over $Y$}. For surfaces we simply require additionally that in the sequence \eqref{eq:MMP} the contracted curves are contained in the fibers of $f$. Then each $(X_i,D_i)$ has an induced projective morphism $f_i\:X_i\to Y$. The following result will be used frequently, cf.\ \cite[1.35]{Kollar_singularities_of_MMP} and \cite{Fujino-semi_stable_MMP_K=0}.
Recall that $\cE(f)$ is the set of curves contracted by $f$.

\begin{lem}[The minimal model of a birational morphism]\label{lem:relativeMMP_unique}
Let $(X,D)$ be a log surface and $f\:X\to Y$ be a birational morphism onto a normal surface.  Then there is a unique $(K_X+D)$-MMP run over $Y$, which we denote by $f^\#_D$. Moreover, if all curves in $$\cE^+_D(f):=\{E\in \cE(f)\: \cf(E;Y,f_*D)<\cf(E;X,D)\}$$ are $\Q$-Cartier, then $f^\#_D$ factors through a partial $(K_X+D)$-MMP which contracts $\cE^+_D(f)$. The latter condition holds for instance if $(X,D)$ is GMRLC. 
\end{lem}

\begin{proof}
Assume first that all curves in $\cE^+_D(f)$ are $\Q$-Cartier. We show that there exists a partial $(K_X+D)$-MMP run over $Y$ which contracts exactly $\cE^+_D(f)$. Put $B=f_*D$ (it does not have to be $\Q$-Cartier). By \eqref{eq:ramification}
\begin{equation}
K_X+D-f^*(K_{Y}+B)=\sum_{E\in \cE(f)}a_EE,
\end{equation}
where $a_E=\cf(E;X,D)-\cf(E;Y,B)$. We may assume that $\cE^+_D(f)$ is nonempty. Since the intersection matrix of the divisor $A:=\sum_{E\in \cE^+_D(f)} a_E E$ is negative definite, we get $0>A^2=\sum_{E\in \cE^+_D(f)}a_EE\cdot A$, hence $a_EE\cdot A<0$ for some $E\in \cE^+_D(f)$. Then $E\cdot (K_X+D)\leq E\cdot A<0$, so $E$ is log exceptional on $(X,D)$. By assumption the curve $E$ is $\Q$-Cartier. Then by Remark \ref{rem:Q-Cartier_under_contraction}(1) the direct image of $K_X+D$ under the contraction of $E$ is $\Q$-Cartier, too. By induction we obtain a partial $(K_X+D)$-MMP over $Y$ contracting exactly $\cE^+_D(f)$.

Let $\ll$, $\mm$ be two distinct log exceptional curves on $(X,D)$ which are $\Q$-Cartier and are contracted by $f$. Let $g\:(X,D) \to (X',D')$ be the contraction of $\mm$. By Remark \ref{rem:Q-Cartier_under_contraction} $g(\ll)$ is $\Q$-Cartier and $(X',D')$ is a log surface. Since $\Exc f$ is negative definite, $g(\ll)^2<0$. We have $g^*(K_{X'}+D')=K_X+D-u\mm$, where $u=(\mm\cdot (K_X+D))/\mm^2>0$, hence $$g(\ll)\cdot (K_{X'}+D')=\ll\cdot g^*(K_{X'}+D')\leq \ll\cdot (K_X+D)<0.$$ Thus $g(\ll)$ is $\Q$-Cartier and log exceptional on $(X',D')$. 

Let $f_i\:X\to X_i$, $i=1,2$ be two $(K_X+D)$-MMP runs over $Y$. Write $f_2=f_2'\circ \sigma$, where $\sigma$ is a contraction of a log exceptional curve and $f_2'$ is an MMP run over $Y$. The above argument shows that $\Exc \sigma\leq \Exc f_1$. Put $f_1':=f_1\circ \sigma^{-1}$. We have $\cE^+_{\sigma(D)}(f_1')=\cE(f_1')$ by Lemma \ref{lem:ld_increasaes}. By the definition of an MMP run the successive contractions are allowed between log surfaces only, which by Remark \ref{rem:Q-Cartier_under_contraction} implies that the log exceptional curves and their proper transforms are $\Q$-Cartier. Thus all curves in $\cE^+_{\sigma(D)}(f_1')$ are $\Q$-Cartier, hence $f_1'$ is an MMP run over $Y$. By induction on the number of curves contracted by $f$ we have $f_1'=f_2'$, hence $f_1=f_2$. This proves that $f_D^\#$ is unique and factors through the contraction of $\cE^+_D(f)$.

Finally, assume that $(X,D)$ is GMRLC. We argue that curves in $\cE^+_D(f)$ are $\Q$-Cartier. We may assume that $\cE^+_D(f)$  is non-empty. Then the log exceptional curve $E\in \cE^+_D(f)$ found by the argument in the first paragraph of the proof is automatically $\Q$-Cartier by Lemma \ref{rem:Q-Cartier_under_contraction} and hence the image of $(X,D)$ under its contraction is GMRLC by Lemma \ref{lem:GMRLC}(1). Write $f=f'\circ \ctr_E$. Then $\cE^+_D(f)\setminus\{E\}$ is the proper transform of $\cE^+_D(f')$, so we proceed by induction using Remark \ref{rem:Q-Cartier_under_contraction}(1).
\end{proof}

A birational morphism of log surfaces $f\:(X,D)\to (Y,B)$ is called \emph{log crepant} (respectively, crepant) if $f^*(K_Y+B)=K_X+D$ (respectively, $f^*K_Y=K_X$). For the definition of an MMP run see Definition \ref{dfn:MMP2}.  

\begin{cor}[A characterization of MMP runs]\label{cor:improving_ld_gives_MMP}
Let $(X,D)$ be a log surface and let $f\:X\to Y$ be a birational morphism onto a normal surface. Put $B=f_*D$. Then the following hold.
\begin{enumerate}[(1)]
\item $f$ is a partial MMP run on $(X,D)$ if and only if every component $E$ of $\Exc f$ is $\Q$-Cartier and satisfies $\cf(E;Y,B)<\cf(E;X,D)$.
\item Assume that $K_X$ and $K_Y+B$ are $\Q$-Cartier or $(X,D)$ is GMRLC. Then $f$ is a partial MMP run on $(X,D)$ of the second kind if and only if every component $E$ of $\Exc f$ is $\Q$-Cartier and satisfies $\cf(E;Y,B)\leq \cf(E;X,D)$.
\item If $f$ is a partial MMP run of the second kind and $K_X$ is $\Q$-Cartier then $f=f_2\circ f_1$, where $f_1$ is a partial MMP run (of the first kind), $K_{f_1(X)}$ is $\Q$-Cartier and $f_2$ is a log crepant partial MMP run of the second kind.
\end{enumerate}
\end{cor}

\begin{proof}
(1) If $f$ is a partial MMP run then the above properties follow from Lemma \ref{lem:ld_increasaes} and Remark \ref{rem:Q-Cartier_under_contraction}(1),(2). The converse implication follows from Lemma \ref{lem:relativeMMP_unique}.

(2) Assume $f$ is an MMP run of the second kind. By Lemma \ref{lem:ld_increasaes} coefficients do not increase. We argue that $f_*^{-1}$ maps $\Q$-Cartier divisors to $\Q$-Cartier divisors and that components of $\Exc f$ are $\Q$-Cartier. By induction we may assume that $\Exc f$ is an irreducible curve, in which case it is $\Q$-Cartier by Definition \ref{dfn:MMP2}, so the claim is clear.

Conversely, assume that components of $\Exc f$ are $\Q$-Cartier and satisfy $\cf(E;Y,B)\leq \cf(E;X,D)$. By Lemma \ref{lem:GMRLC}(2) $K_X$ is $\Q$-Cartier. By Lemmas \ref{lem:relativeMMP_unique}, \ref{rem:Q-Cartier_under_contraction} and \ref{lem:GMRLC}(1) we may assume that $f$ is log crepant. Contracting log exceptional curves of the second kind which are of type (1) in Lemma \ref{lem:contractibility_2nd} we may assume that each log exceptional curve of the second kind on $(X,D)$ contracted by $f$ is as in part (2) or (3) of that lemma. Since components of $\Exc f$ are $\Q$-Cartier, they are of type (2), i.e.\ they are connected components of $D$ of elliptic type. Let $E$ such a component. By Lemma \ref{lem:GMRLC}(1) we may assume that $K_Y+B$ is $\Q$-Cartier. Gluing $X-E$ with $Y-f(\Exc f-E)$ we see that there exists a contraction of $E$. Since $f$ is log crepant and $K_Y+B$ is $\Q$-Cartier, the contraction is log crepant and its image is a log surface. Thus we may contract components of $\Exc f$ one by one, obtaining an MMP run of the second kind. 

(3) Since components of $\Exc f$ are $\Q$-Cartier, the decomposition is given by Lemma \ref{lem:relativeMMP_unique} and the fact that $f_2$ is a partial MMP run of the second kind follows from (2).
\end{proof}

With the same arguments, using Lemma \ref{lem:contractibility_2nd} we prove the following version of Corollary \ref{cor:improving_ld_gives_MMP}(2). 

\begin{rem}
Let $f\:(X,D)\to (Y,B)$ be a birational morphism between log surfaces such that $\cf(E;Y,B)\leq \cf(E;X,D)$ for all $E\in \cE(f)$. If $(X,D)$ is GMRLC, then $f=\sigma\circ f'$, where $f'$ is an MMP run of the second kind and $\sigma\:(X',f'_*D)\to (Y,B)$ is a crepant morphism whose exceptional curves are non-$\Q$-Cartier and disjoint from $f'_*D$. 
\end{rem}

\begin{cor}[Reordering MMP contractions]\label{cor:reordering_MMP}
Let $f_i\:(X,D)\to (X_i, D_i)$, $i=1,2$ be birational morphisms between log surfaces such that $\Exc f_1\leq \Exc f_2$. Put $\alpha:=f_2\circ(f_1)^{-1}\:(X_1,D_1)\to (X_2, D_2)$. Then 
\begin{equation}\label{eq:kappa_is_invariant}
\kappa(K_X+D)=\kappa(K_{X_1}+D_1)=\kappa(K_{X_2}+D_2).
\end{equation}
Moreover, if $(f_1)_*$ maps $\Q$-Cartier divisors to $\Q$-Cartier divisors (for instance it is an MMP run on $(X,D')$ for some boundary divisor $D'$) then the following hold.
\begin{enumerate}[(1)]
\item If $f_2$ is a partial MMP run then $\alpha$ is a partial MMP run.
\item If $f_2$ is a partial MMP run of the second kind and $K_X$ is $\Q$-Cartier then $\alpha$ is a partial MMP run of the second kind.
\end{enumerate}
\end{cor}

\begin{proof}
Given a birational morphism $\sigma\:X_1\to X_2$ between normal surfaces and an effective Cartier divisor $B$ on $X_2$ we have an induced injection $H^0(X_1,\sigma^{-1}_*B)\mono H^0(X_1,\sigma^*B)=H^0(X_2,B)$. The rational map $\alpha\:(X_1,D_1)\to (X_2, D_2)$ is a morphism by Lemma \ref{lem:sub_contractions}, so we get $\kappa(K_X+D)\leq \kappa(K_{X_1}+D_1)\leq \kappa(K_{X_2}+D_2)$. Now since $f_2$ is a composition of contractions of log exceptional curves of the first kind and of log crepant morphisms, we have $\kappa(K_X+D)=\kappa(K_{X_2}+D_2)$ by Lemma \ref{lem:contractibility_2nd}, hence all three Kodaira dimensions are equal. 

(1) Let $E\in \cE(\alpha)$. We have $\cf(E;X_2,D_2)<\cf(E;X,D)$ by Lemma \ref{lem:ld_increasaes}. Since $D_1=(f_1)_*D$, we have $\cf(E;X_1,D_1)=\coeff_{E}(D_1)=\coeff_{E}(D)=\cf(E;X,D)$, so $\cf(E;X_2,D_2)<\cf(E;X_1,D_1)$. The claim follows now from Corollary \ref{cor:improving_ld_gives_MMP}(1).

(2) We have now $\cf(E;X_2,D_2)\leq \cf(E;X_1,D_1)$ for all $E\in \cE(\alpha)$. Moreover, $\Exc \alpha=(f_1)_*\Exc f_1$ and $K_{X_1}=(f_1)_*K_X$, so components of $\Exc \alpha$ and $K_{X_1}$ are $\Q$-Cartier. Then the claim is a consequence of Corollary \ref{cor:improving_ld_gives_MMP}(2).
\end{proof}

The following example reminds that in the situation of Corollary \ref{cor:reordering_MMP} in general we cannot expect $f_1$ to be an MMP run on $(X,D)$. 

\begin{ex}[Reordering contractions] \label{ex:reordering1}
Let $\ll_1$ be a $(-1)$-curve on a smooth projective surface $X$ and let $\ll_2$ be a $(-2)$-curve meeting $\ll_1$ once. The contraction of $\ll_1+\ll_2$ is a composition of the contraction of $\ll_1$ followed by the contraction of the image of $\ll_2$, which is a $(-1)$-curve, hence both resulting surfaces are smooth. But it is also a composition of the contraction of $\ll_2$ onto a singular (canonical) surface $Y$ followed by the contraction of $\ll_1'$, the image of $\ll_1$. Here $\ll_2$ is not log exceptional on $X$, but $\ll_1'$ is log exceptional on $Y$.
\end{ex}

\medskip
\subsection{Relative minimalization}\label{ssec:relative_minimalization}

Given a normal surface $X$ with $\Q$-Cartier $K_X$ (and no boundary) and a birational morphism onto a normal surface $f\:X\to Y$ we denote the unique relative MMP run given by Lemma \ref{lem:relativeMMP_unique} by $f^\#=f^\#_0\:X\to X_f$ and the resulting minimal model of $f$ by $f_{\min}:=f_{0,\min}\:X_f\to Y$. This gives a commutative diagram:
\begin{equation}\label{diag:f_hash}
\begin{gathered}
\xymatrix{	X \ar[rr]^-{f^\#} \ar[rd]_{f}&{} & X_f\ar[ld]^{f_{\min}} \\
	{}& {Y} &{} }
\end{gathered}
\end{equation}
Given a divisor $D$ on $X$ we put $D_f:=f^\#_*D$. If all components of $\Exc f$  are $\Q$-Cartier or if $X$ is GMRLC then $X_f$ is relatively minimal over $Y$, that is, $K_{X_f}$ is $f_{\min}$-nef, and hence $f_{\min}^*K_Y\geq K_{X_f}$ by Corollary \ref{cor:pure_peeling}.

\begin{lem}[Relative minimalization and composition]\label{lem:aMMP_properties}\ 
Let $f\:X\to Y$ and $g\:Y\to Z$ be birational morphisms between normal projective surfaces such that $K_X$ and $K_Y$ are $\Q$-Cartier.  Then
\begin{equation}\label{eq:composing_almost_min}
(g\circ f)^\#=(g\circ f_{\min})^\#\circ f^\#. 
\end{equation} and $(g\circ f_{\min})^\#$ factors through $(g^\#\circ f_{\min})^\#$. If all components of $\Exc f$ are $\Q$-Cartier or $X$ is GMRLC then $(g\circ f)^\#$ contracts the proper transform of  $\Exc g^\#$. 
\end{lem}

\begin{proof}
By definition $(g\circ f_{\min})^\#\circ f^\#$ is a $K_X$-MMP run over $Z$, hence by uniqueness, see Lemma \ref{lem:relativeMMP_unique}, we have $(g\circ f)^\#=(g\circ f_{\min})^\#\circ f^\#$. Moreover $(g^\#\circ f_{\min})^\#\circ f^\#$ is a $K_X$-MMP run over $Y_g$, hence a partial $K_X$-MMP run over $Z$, so $(g\circ f)^\#$ factors through it by uniqueness. This gives a commutative diagram
\begin{equation}
\begin{gathered}
\xymatrix@C=3em{X \ar[r]^{f^\#}\ar[rd]_{f} & X_f\ar[r]^-{(g^\#\circ f_{\min})^\#}\ar[d]^{f_{\min}}\ar@/^2.5pc/[rr]^{(g\circ f_{\min})^\#} & {X_{g^\#\circ f}} \ar[d]\ar[r]& {X_{g\circ f}} \ar[ddl]\\
{}& Y\ar[r]^{g^\#}\ar[rd]_{g} & Y_g\ar[d] & {}\\
{}&{}&Z & {}}
\end{gathered}
\end{equation}

By the definition of an MMP and by Remark \ref{rem:Q-Cartier_under_contraction} components of $\Exc f^\#$ and $\Exc g^\#$ are $\Q$-Cartier. Assume that $X$ is GMRLC or that components of $\Exc f$ are $\Q$-Cartier. In the first case by Lemma \ref{lem:GMRLC} $X_f$ is GMRLC. In the second case components of $\Exc f_{\min}$ are $\Q$-Cartier. As noted above, in both cases $K_{X_f}$ is $f_{\min}$-nef. 

To show that $(g^\#\circ f_{\min})^\#$ contracts the proper transform of $G:=\Exc g^\#$, it is sufficient to construct a partial $K_{X_f}$-MMP run over $Y_g$ which contracts the proper transform of $G$. Put $\alpha=f_{\min}$. We may assume that $\Exc g^\#\neq 0$. Take a component $\mm$ of $G$ intersecting $K_Y$ negatively (it is $\Q$-Cartier by the definition of $g^\#$) and let $\sigma\:Y\to Y'$ be its contraction. By Corollary \ref{cor:pure_peeling} we have $K_{X_f}=\alpha^*K_Y-E$ for some effective $\alpha$-exceptional divisor $E$. Put $\mm'=\alpha_*^{-1}\mm$. Then $\mm'\cdot K_{X_f}\leq \mm'\cdot \alpha^*K_Y=\mm\cdot K_Y<0$, so $\mm'$ is log exceptional on $X_f$. If $X_f$ is GMRLC then it follows that $\mm'$ is $\Q$-Cartier. Similarly, if all components of $\Exc (f_{\min})$ are $\Q$-Cartier then the total transform of $\mm$, and hence $\mm'$ is $\Q$-Cartier. By Remark \ref{rem:Q-Cartier_under_contraction} it follows that there exists a $K_{X_f}$-MMP over $Y'$ contracting $\mm'$ whose outcome is a birational morphism $\alpha'\:X'\to Y'$ for which $K_{X'}$ is $\alpha'$-nef. Thus we may replace $X_f$, $Y$, $\alpha\:X_f\to Y$, $g^\#\:Y\to Y_g$ with $X'$, $Y'$, $\alpha'\:X'\to Y'$ and $g^\#\circ \sigma^{-1}\:Y'\to Y_g$ and we continue the construction until the whole proper transform of $G$ is contracted.
\end{proof}

\begin{lem}[Relative minimality and composition]\label{lem:pure_composition}
Let $f\:X\to Y$ and $g\:Y\to Z$ be birational morphisms between normal projective surfaces such that $K_X$ and $K_Y$ are $\Q$-Cartier. If $X$ is minimal over $Z$ then $X$ is minimal over $Y$ and $Y$ is minimal over $Z$.
\end{lem}

\begin{proof}The first part of the statement is obvious. By Lemma \ref{cor:pure_peeling} $f^*K_Y=K_X+E$ for some effective $f$-exceptional divisor $E$, hence for any curve $L$ on $Y$ we have $L\cdot K_Y=f_*^{-1}L\cdot (K_X+E)\geq f_*^{-1}L\cdot K_X$, which gives the second part of the statement.
\end{proof}

\begin{rem}[Relative minimality and composition]\label{rem:composing_almost_minimalizations}
Let the notation be as in Lemma \ref{lem:aMMP_properties}. We note that $(g\circ f)^\#$ may contract more than $(g^\#\circ f_{\min})^\#\circ f^\#$. For instance, if $X$ is minimal over $Y$ and $Y$ is minimal over $Z$  then the latter morphism is the identity. But it can happen that at the same time $X$ is not minimal over $Z$, so $(g\circ f)^\#$ is not an isomorphism; see Example \ref{ex:composition_of_nef}.
\end{rem}

By $[a_1,\ldots,a_n]$ we denote an snc divisor (on a smooth complete surface) which is a chain of smooth rational curves with subsequent self-intersection numbers equal to $-a_1, -a_2,\ldots, -a_n$, respectively; cf.\ Section \ref{ssec:twigs_forks_barks}.

\begin{ex}[Relative minimality and composition] \label{ex:composition_of_nef}
Consider a smooth projective surface $X$ with a divisor $[n_1,1,n_2]$ on it, where $n_1,n_2$ are positive integers such that $\frac{1}{n_1}+\frac{1}{n_2}\leq \frac{1}{2}$. Let $A$ be the $(-1)$-curve and let $D_i$, $i=1,2$ be the component with self-intersection $-n_i$. Let $h\:X\to Z$ be the contraction of $A+D_1+D_2$ and let $f\:X\to Y$ be the contraction of $D_1+D_2$. Then $g=h\circ f^{-1}\:Y\to Z$ is the contraction of $f_*A$. We have $f^*K_Y=K_X+(1-\frac{2}{n_1})D_1+(1-\frac{2}{n_2})D_2,$ hence $f_*A\cdot K_Y=1-\frac{2}{n_1}-\frac{2}{n_2}\geq 0.$ It follows that $K_X$ is $f$-nef and $K_Y$ is $g$-nef, but $K_X$ is not $h=g\circ f$-nef. Still, $h^*K_Z-K_X=f^*(g^*K_Z-K_Y)+(f^*K_Y-K_X)$ is effective by Corollary \ref{cor:pure_peeling}.
\end{ex}

\begin{cor}\label{cor:nef_gives_peeling} Let $f\:(X,D)\to (\ov X,\ov D)$ be a birational morphism of log surfaces. Assume that $K_X$ is $\Q$-Cartier.
\begin{enumerate}[(1)]
\item If $f$ is a partial MMP run then $f_{\min}$ is a partial MMP run, $X_f$ is minimal over $\ov X$ and $\Exc f_{\min}$ consists of some components of $D_f$.

\item Assume that $f$ is a partial MMP run of the second kind. Then  $f_{\min}$ is a partial MMP run the second kind on $(X_f,D_f)$, $X_f$ is minimal over $\ov X$ and  $\Exc f_{\min}$ consists of some components of $D_f$ and $K_X$-trivial $\Q$-Cartier components disjoint from $D_f$. 
\end{enumerate}
\end{cor}

\begin{proof} 
Since almost minimalization is a partial MMP run on $(X,0)$, by Remark \ref{rem:Q-Cartier_under_contraction} it maps $\Q$-Cartier divisors to $\Q$-Cartier divisors. By Corollary \ref{cor:reordering_MMP} it follows that $f_{\min}$ is a partial MMP run of the first kind in (1) and of the second kind in (2). We may thus assume that $f=f_{\min}$ and $X=X_f$. Since components of $\Exc f$ are $\Q$-Cartier, $X$ is now minimal over $\ov X$. 

(1) Let $E$ be a log exceptional curve on $(X,D)$. We have $E\cdot D<-E\cdot K_{X}\leq 0$, so $E$ is a component of $D$. Since it is $\Q$-Cartier, its contraction, call it $\tau$, maps $\Q$-Cartier divisors to $\Q$-Cartier divisors. We have $\tau^*(\tau_*K_{X})=K_{X}+uE$ for some $u\geq 0$, hence $\tau_*K_{X}$ is $f\circ \tau^{-1}$-nef. We obtain the claim by induction. 

(2) Write $f=f_2\circ f_1$, where $f_1$ is the unique $(K_{X}+D)$-MMP over $\ov X$ given by Lemma \ref{lem:relativeMMP_unique}. It contracts $\cE^+_{D}(f)$, so $f_2$ is log crepant. By part (1) we have $\Exc f_1\leq D$. Put $(Y,B)=(f_1(X),(f_1)_*D)$. Then $K_Y$ and the components of $\Exc f_2$ are $\Q$-Cartier, and $K_Y$ is $f_2$-nef. Let $E$ be a component of $\Exc f_2$. We have $E\cdot B=-E\cdot K_Y\leq 0$. In particular, if $E$ is not a component of $B$ then it is $K_Y$-trivial and disjoint from $B$. 
\end{proof}

\medskip
\subsection{The definition of an almost minimalization}\label{ssec:almost_minimalization}
An inductive construction of an almost minimal model and a characterization of almost log exceptional curves were given by Miyanishi in the special case of log smooth surfaces with reduced boundary, see \cite[p.\ 107]{Miyan-OpenSurf}. In \cite{Palka-minimal_models} we generalized this construction in an analogous way to smooth completions of smooth affine surfaces with half-integral boundaries. Here we give a simple general definition. 

\begin{dfn}[Almost minimalization]\label{dfn:almost_minimalization_new}
Let $(X,D)$ be a log surface with $K_X$ being $\Q$-Cartier and let $f\:(X,D)\to (\ov X,\ov D)$ be a partial MMP run. Then $f^\#\:X\to X_f$, that is, the unique $K_X$-MMP over $\ov X$, is called an \emph{almost minimalization of $f$}. If $f$ is an MMP run (of the first kind) then the log surface $(X_f,D_f)$ is called an \emph{almost minimal model} of $(X,D)$. 
\end{dfn}

\begin{rem}\label{rem:(Xf,Df)}
Remark \ref{rem:Q-Cartier_under_contraction}(1) implies that $f^\#$ maps $\Q$-Cartier divisors to $\Q$-Cartier divisors, hence $K_{X_f}$ is $\Q$-Cartier and $(X_f,D_f)$ is a log surface. The morphism $f_{\min}$ is a partial MMP run by Corollary \ref{cor:nef_gives_peeling}(1). If $X$ is GMRLC then by Lemma \ref{lem:GMRLC} the surface $X_f$ is GMRLC. Similarly, if $X$ is $\Q$-factorial (respectively, log canonical) then $X_f$ is $\Q$-factorial (respectively, log canonical).
\end{rem}

The geometry of $D_f$ is to be studied. It follows from the definition that an almost minimalization is a composition of contractions of some of the proper transforms of log exceptional curves on $(X,D)$ and its images under the contractions constituting $f$. The analysis of the resulting boundary $D_f$ comes down to the analysis of reordering MMP contractions for $(X,D)$ constituting $f$. It is well-known that in general changing the order of contractions may lead to worse singularities of the intermediate surfaces, see Example \ref{ex:reordering1}.

\medskip
\subsection{Peeling and squeezing}\label{ssec:peeling_squeezing_alm_min}

Let $(X,D)$ be a log surface with $K_X$ being $\Q$-Cartier. Given a partial MMP run $\psi\:(X,D)\to (\ov X,\ov D)$, by Corollary \ref{cor:nef_gives_peeling} its almost minimalization $\psi_\am:=\psi^\#\:(X,D)\to (X_\psi,D_\psi)$ induces a partial MMP run $\psi_{\min}=\psi\circ \psi_\am^{-1}\:(X_\psi,D_\psi)\to (\ov X, \ov D)$ such that $K_{X_\psi}$ is $\psi_{\min}$-nef and $\Exc \psi_{\min}\leq D_\psi$:
\begin{equation}\label{eq:am-min_decomposition}
\begin{gathered}
\xymatrix{	(X,D) \ar[r]^-{\psi_\am} \ar[rd]_{\psi}& (X_\psi,D_\psi)\ar[d]^{\psi_{\min}} \\
	& (\ov X,\ov D)}
\end{gathered}
\end{equation}
Conversely, using Lemma \ref{lem:aMMP_properties} we can construct an almost minimal model of $(X,D)$ in steps in parallel to a construction of an MMP run $\psi$. We now discuss how to conveniently group these steps. The distinction between log exceptional curves contained  and not contained in the boundary leads to the definition of a peeling, cf.\ \cite[2.3.3.6]{Miyan-OpenSurf}. 

\begin{samepage}
\begin{dfn}[Peeling, almost minimality] \label{def:peel_squeeze} Let $(X,D)$ be a log surface.
\begin{enumerate}[(1)]
\item A \emph{partial peeling} of $(X,D)$ (called also a partial peeling of $D$) is a partial MMP run $\alpha$ on $(X,D)$ for which $\Exc \alpha\subseteq \Supp D$. A \emph{peeling} is a maximal partial peeling, i.e.\ a partial peeling which cannot be extended to a partial peeling with a strictly bigger number of contracted curves.
\end{enumerate}
Let $\alpha$ be a partial peeling of $(X,D)$ with $K_X$ being $\Q$-Cartier. 
Then:
\begin{enumerate}[(1)]
\stepcounter{enumi}
\item $\alpha$ is called \emph{pure} if $K_X$ is $\alpha$-nef. In this case we say that $(X,D)$ is \emph{$\alpha$-squeezed}. 
 
\item If $(X,D)$ is $\alpha$-squeezed and $(\alpha(X),\alpha_*D)$ is minimal then $(X,D)$ is called \emph{$\alpha$-almost minimal}. (In this case $\alpha$ is necessarily a peeling.)
\item $\alpha_\am$ is called an \emph{$\alpha$-squeezing} of $(X,D)$; see \eqref{eq:am-min_decomposition}.
\end{enumerate}
We say that $(X,D)$ is \emph{squeezed} if it is $\alpha$-squeezed for some peeling $\alpha$. It is \emph{almost minimal} if it is $\alpha$-almost minimal for some (pure) peeling $\alpha$.
\end{dfn}
\end{samepage}

\begin{dfn}[Redundant and almost log exceptional curves]\label{dfn:redundant_and_almost-log-exc} 
Let $\alpha$ be a pure partial peeling of a log surface $(X,D)$ and let $\ll\subseteq X$ be a curve. 
\begin{enumerate}[(1)] 
\item We say that $\ll$ is \emph{$\alpha$-almost log exceptional} if $\ll\nleq D$ and $\alpha(\ll)$ is log exceptional on $(\alpha(X),\alpha_*D)$,
\item Assume $K_X$ is $\Q$-Cartier. We say that $\ll$ is \emph{$\alpha$-redundant} if $\ll\leq D$, $\ll\cdot K_X<0$ and $\alpha(\ll)$ is log exceptional on $(\alpha(X),\alpha_*D)$.
\end{enumerate}
A curve is \emph{almost log exceptional} if it is $\alpha$-almost log exceptional for some pure peeling $\alpha$. A curve is \emph{redundant} if it is $\alpha$-redundant for some pure partial peeling $\alpha$. 
\end{dfn}

We use terms like '$\alpha$-redundant' and 'redundant with respect to $\alpha$' interchangeably. A peeling restricts to an isomorphism of $X\setminus D$ onto its image. Since a redundant curve intersects $K_X$ negatively, for smooth $X$ it is in particular a $(-1)$-curve. A redundant component of $D$ or an almost log exceptional curve on $(X,D)$ is in general not log exceptional itself, so even if it can be contracted, the effect of this contraction on (coefficients of) log singularities requires a more careful analysis. Squeezing should be thought of as a useful preparation for running an almost MMP by contracting some components of $D$ without making the singularities of the underlying surface $X$ worse.  

\begin{cor}\label{cor:witnessing_squeezing_or_alm_min}\ Let $(X,D)$ be a log surface for which $K_X$ is $\Q$-Cartier.
\begin{enumerate}[(1)]
\item Let $\alpha$ be a peeling of $(X,D)$. Then $(X,D)$ is $\alpha$-squeezed (i.e.\ $\alpha$ is pure) if $D$ contains no $\alpha'$-redundant component for any pure partial peeling $\alpha'$ with $\Exc \alpha'\leq \Exc \alpha$. An $\alpha$-squeezed $(X,D)$ is $\alpha$-almost minimal if there are no $\alpha$-almost log exceptional curves on $(X,D)$.
\item $(X,D)$ is almost minimal if and only if there exists an MMP run $\psi$ on $(X,D)$ for which $\psi_\am$ is an isomorphism.
\end{enumerate} 
\end{cor}

\begin{proof}
(1) follows from definitions. (2) If $(X,D)$ is $\alpha$-almost minimal then we take $\psi=\alpha$. Now $\psi_\am$ is an isomorphism, because $K_X$ is $\psi$-nef. Conversely, if $\psi_\am$ is an isomorphism then $K_X$ is $\psi$-nef, hence $\Exc \psi\subseteq \Supp D$ by Corollary \ref{cor:nef_gives_peeling}, so $\psi$ is a pure peeling, hence $(X,D)$ is $\psi$-almost minimal.
\end{proof}

\begin{ex}[Peeling and almost minimality] \label{ex:peeling}
Let $\F_n=\P(\cO_{\P^1}\oplus \cO(n))$ for $n\geq 2$ be a Hirzebruch surface. Let $F_i$, $i=1,\ldots, N$, $N\geq 1$ be distinct fibers of a $\P^1$-fibration of $\F_n$ and let $C$ be the negative section. Put $D=cC+\sum_{i=1}^N w_i F_i$, where $c,w_1,\ldots,w_N\in \Q\cap [0,1]$. The peeling morphism is either the identity or it contracts $C$. In each case the log surface $(\F_n,D)$ is almost minimal. The peeling morphism for $(\F_n,D)$ does contract $C$ if and only if $C\cdot (K+D)<0$, that is, if $n(1-c)+\sum_{i=1}^N w_i<2$. For instance, if all coefficients of $D$ are equal to $r$ then the condition reads as $n(1-r)+Nr<2$, so peeling contracts $C$ if and only if $N=1$ and $r\in [1-\frac{1}{n-1},1]$.
\end{ex}

\begin{ex}[An almost minimal surface with an $\alpha$-almost log exceptional curve] \label{ex:partially_almost_minimal}
Blowing up three times on $\P^1\times \P^1$ we construct a smooth $\P^1$-fibered surface $X$ with a unique degenerate fiber $T_1+T_2+T_3+L$ where $T_i$, $i=1,2,3$ are $(-2)$-curves and $L$ is a $(-1)$-curve meeting $T_2$. Let $F_1, F_2$ be two smooth fibers and $T_0$ be a section meeting $T_3$ and having self-intersection number equal to $0$. Put $D=T_1+T_2+T_3+T_0+F_1+F_2$, see Figure \ref{ex:a-min_surf_with_a-min_curve}.
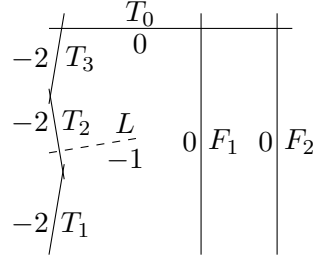
\begin{figure}[H] 
\begin{tikzpicture}
	\draw (0,3) -- (3.2,3);
	\node at (1.2,3.2) {$T_0$};
	\node at (1.2,2.8) {$0$};
	\draw (2,3.2) -- (2,0);
	\node at (2.3,1.5) {$F_1$};
	\node at (1.85,1.5) {$0$};
	\draw (3,3.2) -- (3,0);
	\node at (3.3,1.5) {$F_2$};
	\node at (2.85,1.5) {$0$};
	\draw (0.2,3.2)-- (0,2);
	\node at (-0.25,2.6) {$-2$};
	\node at (0.35,2.6) {\ $T_3$};
	\draw (0,2.2) -- (0.2,1);
	\node at (-0.25,1.75) {$-2$};
	\node at (0.3,1.75) {\ $T_2$};
	\draw (0.2,1.2) -- (0,0);
	\node at (-0.25,0.4) {$-2$};
	\node at (0.35,0.4) {$T_1$};
	\draw[dashed] (0,1.35) -- (1.2,1.55);
	\node at (1, 1.75) {$L$};
	\node at (1, 1.25) {$-1$};
\end{tikzpicture}
\caption{The boundary $D$ in Example \ref{ex:partially_almost_minimal}.}\label{ex:a-min_surf_with_a-min_curve}
\end{figure}
The curve $L$ is a unique $(-1)$-curve on $X$ intersecting $D$ once. Let $\alpha$ be the contraction of $T_1+T_2$. It is a pure partial peeling. The curves
$\alpha(T_3)$ and $\alpha(L)$ are log exceptional on $(\alpha(X),\alpha_*D)$. Let $\psi$ be the contraction of $T_1+T_2+T_3$. It is an MMP run with singular image $\psi(X)$. (Another MMP run $\psi'$ has $\Exc \psi'=T_1+T_2+L$, in which case $\psi'(X)$ is smooth). We infer that $(X,D)$ is a log smooth almost minimal log surface. Since $L$ is $\alpha$-almost log exceptional, $(X,D)$ is not $\alpha$-almost minimal. Thus, on an almost minimal log surface there may exist curves which are almost log exceptional with respect to some pure partial peeling. 
\end{ex}

Almost log exceptional (respectively, $\alpha$-redundant) curves will be also called \emph{almost log exceptional} (respectively, \emph{$\alpha$-redundant}) \emph{curves of the first kind}. Their numerical characterization will be discussed in the next sections. For their analogues of the second kind see Definition \ref{dfn:almost_and_superlfuous_2nd}. For convenience the following result is formulated jointly for both kinds.
 
\begin{lem}[Almost log exceptional vs redundant] \label{lem:a.l.e._is_redundant}
Let $\alpha$ be a pure partial peeling of a log surface $(X,D)$ with $K_X$ being $\Q$-Cartier. If $A\subseteq X$ is a $\Q$-Cartier $\alpha$-almost log exceptional curve of the first or second kind then\begin{enumerate}[(1)]
\item  $A\cdot K_X<0$, 
\item for every $s\in (0,1]\cap \Q$ the morphism $\ctr_{\alpha(A)}\circ\alpha$ is a peeling of $(X,D+sA)$,
\item for every $s\in (0,1]\cap \Q$ the curve $A$ is $\alpha$-redundant on $(X,D+sA)$ of the first kind.
\end{enumerate}
\end{lem}

\begin{proof}Let $(\ov X,\ov D)$ be the image of $\alpha$. The curve $\alpha(A)$ is $\Q$-Cartier. Let $\varphi\:\ov X\to Y$ be its contraction. By Corollary \ref{cor:pure_peeling} we have $A\cdot K_X\leq A\cdot \alpha^*(K_{\ov X})=\alpha(A)\cdot K_{\ov X}$. Since $A\nleq D$, we have $\alpha(A)\cdot K_{\ov X}\leq \alpha(A)\cdot (K_{\ov X}+\ov D)\leq 0$. For a curve of the first kind the latter inequality is strict and for a curve of the second kind $A\cdot K_X\neq 0$ by definition, which gives (1).  By Lemma \ref{lem:ld_increasaes} for every prime divisor $E$ contracted by $\alpha$ we have $$\cf(E;Y,\varphi_*\ov D)\leq \cf(E;\ov X,\ov D)<\cf(E;X,D)=\cf(E;X,D+sA).$$ On the other hand, $$\cf(A;Y,\varphi_*\ov D)\leq \cf(A;\ov X,\ov D)=0<s=\cf(A,X,D+sA).$$ Since $\alpha$ is a partial MMP run, at each step the contracted curve is $\Q$-Cartier, and hence components of $\Exc \alpha$ are $\Q$-Cartier; see Remark \ref{rem:Q-Cartier_under_contraction}(2),(1). By Corollary \ref{cor:improving_ld_gives_MMP}(1) the morphism $\varphi\circ\alpha\:(X,D+sA)\to (Y,\varphi_*D)$ is a peeling, which gives (2). Since $\alpha$ is pure, $A$ is the only curve in $\Exc \alpha+A$ intersecting $K_X$ negatively, hence there exists $\alpha'$ - a pure partial peeling of $(X,D+sA)$ with $\Exc \alpha'\leq \Exc \alpha$ such that $A$ is $\alpha'$-redundant of the first kind on $(X,D+sA)$. But then it is $\alpha$-redundant of the first kind by Corollary \ref{cor:cf_decrease_under_MMP}, hence (3).
\end{proof}

\medskip
\subsection{Effective almost minimalization}

The following lemma shows how to conveniently construct an almost minimal model of a log surface $(X,D)$ in parallel to a minimal model. It shows also that the construction generalizes Miyanishi's construction for log smooth surfaces (with the squeezing of the boundary being an analogue of an snc-minimalization), see \cite[2.3.11]{Miyan-OpenSurf}, cf.\ Remark \ref{rem:relation_to_Miyanishi}

\begin{lem}[Effective almost minimalization] \label{lem:effective_almost_minimalization} 
Assume that $\ov \psi\:(X,D)\to (\ov X, \ov D)$ is a partial MMP run on a log surface with $K_X$ being $\Q$-Cartier. Let $n$ denote the number of $\ov \psi$-exceptional curves not contained in $D$. Choose a decomposition $\ov \psi=\ov \psi_n\circ \ldots\circ \ov \psi_0$, where $\ov \psi_0\:(X,D)\to (\ov X_1, \ov D_1)$ is a partial peeling and $\ov \psi_i\:(\ov X_i,\ov D_i)\to (\ov X_{i+1},\ov D_{i+1})$ for $i\geq 1$, where $(\ov X_{n+1},\ov D_{n+1})=(\ov X,\ov D)$, is a composition of a contraction of a single log exceptional curve $\ov A_i\nleq \ov D_i$ followed by a partial peeling. Then there exists a commutative diagram
\begin{equation}\label{eq:MMP_diagram}
\begin{gathered}
\xymatrix{	(X,D) \ar[r]^-{\psi_0} \ar[rd]_-{\ov \psi_0}& (X_1,D_1) \ar[d]^-{\alpha_1} \ar[r]^-{\psi_1} & (X_2,D_2) \ar[d]^-{\alpha_2} \ar[r]^-{\psi_2} & \dots \ar[r]^-{\psi_n} & (X_{n+1},D_{n+1})\ar[d]^-{\alpha_{n+1}} \\ 
	& (\ov X_1,\ov D_1) \ar[r]^-{\bar{\psi}_1} & (\ov X_2,\ov D_2) \ar[r]^-{\bar{\psi}_2} & \dots \ar[r]^-{\bar{\psi}_n} & (\ov X_{n+1},\ov D_{n+1})} 
\end{gathered}
\end{equation}
where
\begin{enumerate}[(1)]
\item Each $(X_i,D_i)$, $i\geq 1$ is a log surface with $K_{X_i}$ being $\Q$-Cartier.
\item $\alpha_i$ is a pure partial peeling for $i\geq 1$,
\item $\psi_0=(\ov \psi_0)_\am$ and $\psi_i=(\ov \psi_i\circ \alpha_i)_\am$ for $i\geq 1$, hence $\psi_i$ is a composition of a contraction of an $\alpha_i$-almost log exceptional curve $A_i:=(\alpha_i)_*^{-1}\ov A_i\subseteq X_i$ followed by a $\tau_i$-squeezing,  where $\ov \psi_i\circ \alpha_i=\tau_i\circ \ctr_{A_i}$.
\item $\psi:=\psi_n\circ \ldots\circ \psi_0=\ov \psi_\am$ and $(X_{n+1},D_{n+1})$ is $\alpha_{n+1}$-almost minimal. In particular, if $(\ov X_{n+1}, \ov D_{n+1})$ is minimal then $(X_{n+1},D_{n+1})$ is an almost minimal model of $(X,D)$.
\end{enumerate}
Moreover, for every such diagram we have:
\begin{enumerate}[(1)]
\setcounter{enumi}{4}
\item $\cf(X_{i+1})\leq \cf(X_i)$ for $i\geq 0$, where $X_0=X$,
\item $\kappa(K_{X_i}+D_i)=\kappa(K_X+D)$ for $i\geq 1$,
\end{enumerate}
\end{lem}

\begin{proof} 
We define inductively $\psi_i$ to be an almost minimalization as in (3) and hence by Corollary \ref{cor:nef_gives_peeling} and Remark \ref{rem:Q-Cartier_under_contraction}(1) we get that $K_{X_{i+1}}$ is $\Q$-Cartier and the induced $\alpha_{i+1}$ is as in (2). Note that $\psi_i$ and $\alpha_i$ map $\Q$-Cartier divisors to $\Q$-Cartier divisors by Remark \ref{rem:Q-Cartier_under_contraction}(1). By Lemma \ref{lem:a.l.e._is_redundant} $A_i\cdot K_{X_i}<0$, so $\psi_i$ is a composition of a contraction of $A_i$ with a $\tau_i$-squeezing. Part (4) follows from equation \eqref{eq:composing_almost_min}. Part (5) is a consequence of the definition of almost minimalization and of Lemma \ref{lem:ld_increasaes}. Part (6) follows from Corollary \ref{cor:reordering_MMP}.
\end{proof}

\begin{cor}[Summary] \label{cor:aMMP_algorithm} Let $(X,D)$ be a log surface with $\Q$-Cartier $K_X$, which admits a complete MMP run, for instance $(X,D)$ is GMRLC. The above lemma shows that an almost minimal model of $(X,D)$ can be constructed inductively as follows. All contractions below are assumed to be over the chosen minimal model, so that the intermediate surfaces admit complete MMP runs, too.
\begin{enumerate}[(1)]
\item Choose a maximal pure partial peeling morphism $\alpha$ of $(X,D)$ and consider the triple $(X,D,\alpha)$.
\item If there exists a curve $\ll\subseteq X$ such that $\alpha(\ll)$ is log exceptional then put $D':=D$ in case $\ll\leq D$ and $D':=D+\ll$ otherwise. Let $\sigma$ be an $\ctr_{\alpha(\ll)}\circ \alpha$ -squeezing of $(X,D')$ and let $\alpha'$ be a maximal pure partial peeling extending $\ctr_{\alpha(\ll)}\circ\alpha\circ \sigma^{-1}$. Replace $(X,D,\alpha)$ with $(\sigma(X),\sigma_*D', \alpha')$ and repeat.
\item 
If there is no $\ll$ as in (2) then the resulting $(X,D)$ is an almost minimal model and $\alpha$ is a pure peeling morphism onto a minimal model.
\end{enumerate}
\end{cor}

\begin{rem}
Note that in part (2) of Corollary \ref{cor:aMMP_algorithm} $\ll$ is necessarily a redundant component of $D$ or an $\alpha$-almost log exceptional curve, hence a redundant component of $D'$; cf. Lemma \ref{lem:a.l.e._is_redundant}. Also, $\sigma$ is the composition of successive contractions of redundant components of $D'$ and its respective images which are contained in $\ll+\Exc \alpha$ and its respective images.
\end{rem}

\begin{rem}\label{rem:MMP_with_D_priority}
In Lemma \ref{lem:effective_almost_minimalization} assume that $\ov \psi$ is a complete MMP run and choose $\ov \psi_0$ to be a peeling and $\ov \psi_i$ for $i\geq 1$ to be a composition of a contraction of a single log exceptional curve $\ov A_i\nleq \ov D_i$ and a peeling. Then in Lemma \ref{lem:effective_almost_minimalization} 
$\alpha_i$ are (pure) peeling morphisms and $(X_i,D_i)$ are squeezed.
\end{rem}

\begin{rem}[Comparison with Miyanishi's construction] \label{rem:relation_to_Miyanishi}
Let $(X,D)$ be a log smooth surface with a reduced boundary. In this case there is already a definition of an almost minimal model by Miyanishi \cite[2.3.11]{Miyan-OpenSurf} which is very close the one above. The difference is that in our construction $\psi$ (and $\ov \psi$) contracts slightly fewer components of the boundary. To see this difference let $C\leq D$ be a superfluous component, that is, a $(-1)$-curve meeting at most two other components of $D$, each at most once in the sense of intersection theory. In particular $\beta_D(C)\leq 2$. If $\beta_D(C)\leq 1$ or if $\beta_D(C)=2$ and $C$ meets some maximal admissible (rational) twig or fork of $D$ (see Section \ref{ssec:reduced_boundaries}) then it is almost log exceptional and hence is contracted by almost minimalization. However, if $\beta_D(C)=2$ and $C$ meets no maximal admissible (rational) twig or fork of $D$ then in Miyanishi's construction we do contract $C$ (which is reasonable, as this respects log smoothness), but we do not do that according to our definition. Indeed, in the latter case $C$ is not almost log exceptional in our sense. In fact, it is not possible to assure in general that the contraction of $C$ of the latter type is a part of some MMP run. Still, if one wants to contract such curves too, then in our approach this can be done after an almost minimal model in the sense of our definition is reached. Indeed, note that we have $C\cdot (K_X+D)=0$, so $C$ is log exceptional of the second kind and is disjoint form the exceptional locus of a peeling of $(X,D)$, see Definition \ref{eq:log_exc2}. Then it is easy to see that after such additional contraction of $C$ the image of $(X,D)$ remains almost minimal; cf.\ Lemma \ref{lem:reordering_1st_2nd}.
\end{rem}

\medskip
\subsection{Analogues for curves of the second kind}\label{ssec:2nd_kind}

We now discuss analogues of the above definitions for curves of the second kind. The following lemma implies that contractions of log exceptional curves of the second kind (if they exist, see Example \ref{ex:elliptic_contraction}) can be delayed until a minimal model is reached and then they respect minimality.

\begin{lem}[Delaying contractions of the second kind] \label{lem:reordering_1st_2nd}   
Let $\gamma\:(X,D)\to (\ov X,\ov D)$ be a log crepant MMP run of the second kind and let $\ov \psi\:(\ov X,\ov D)\to (\ov Y,\ov B)$ be a partial MMP run (of the first kind). Assume that $K_X$ is $\Q$-Cartier. Then there exists $\psi\:(X,D)\to (Y,B)$, a partial MMP run of the first kind with $\gamma_*^{-1}(\Exc \ov \psi)\leq \Exc \psi$, such that $K_Y$ is $\Q$-Cartier and the induced birational morphism $\ov \gamma\:(Y,B)\to (\ov Y,\ov B)$ is a log crepant MMP run of the second kind. In particular, if $(X,D)$ is minimal then $(\ov X, \ov D)$ is minimal.
\end{lem}

\begin{proof} The morphism $f=\ov \psi\circ \gamma$ is an MMP run of the second kind. Since components of $\Exc \ov \psi$ are $\Q$-Cartier, so are the components of $\Exc f$. By Lemma \ref{lem:relativeMMP_unique} $f$ factors through a unique MMP run on $(X,D)$ over $\ov Y$, call it $\psi$, and $\psi$ contracts $\cE^+_D(f)$.  This gives a commutative diagram
\begin{equation}\label{eq:reordering2}
\begin{gathered}
\xymatrix{
	(X,D) \ar[d]_{\gamma} \ar[r]^{\psi} & (Y,B) \ar[d]^{\ov \gamma}\\ 
	(\ov X,\ov D) \ar[r]^{\ov \psi} & (\ov Y,\ov B)
}
\end{gathered}
\end{equation}
where $\ov \gamma$ is log crepant and components of $\Exc \ov \gamma$ are $\Q$-Cartier. Moreover, $K_Y=\psi_*K_X$ is $\Q$-Cartier. By Corollary \ref{cor:improving_ld_gives_MMP} $\ov \gamma$ is an MMP run of the second kind.
\end{proof}

The following example shows that for $\ov \psi$ of the second kind it can happen that $\gamma_*^{-1}(\Exc \ov \psi)$ is non-contractible (in the category of algebraic varieties).

\begin{ex}[Non-permutable order of contractions of log exceptional curves of the second kind] \label{ex:problems_with_reordering_2nd_kind}
Let $\sigma\:\ov X\to \P^2$ be a blowup as in Example \ref{ex:elliptic_setup}(3) and let $\ov \ll$ be the proper transform of the cubic. The centers $p_1,\ldots,p_n$ of $\sigma$ are chosen so that there exists a contraction of $\ov \ll$; denote it by $\ov \psi\:\ov X\to \ov Y$. Now pick a point $p\in \ov \ll$ not lying over any $p_i$, $i=1,\ldots, n$. Denote by $\gamma\:X\to \ov X$ the blowup at $p$ and put $\ll=\gamma_*^{-1}\ov \ll$. Let $E$ be the exceptional curve of $\gamma$. Clearly, $\ll$ and $E$ are log exceptional curves of the second kind on $(X,\ll)$ and by construction we have a log crepant contraction $\gamma$ of $E$ followed by the contraction of $\gamma(\ll)$. If there exists a contraction of $\ll$ then the induced morphism contracts the image of $E$, that is, the order of contractions can be changed and we get a commutative diagram \eqref{eq:reordering2}. But the existence of the contraction of $\ll$ implies that $\sigma(p), p_1,\ldots, p_n$ are $\Z$-linearly dependent in the group law of $\sigma(\ll)$. Assuming that the field is not an algebraic closure of a finite field we can always pick $p$ so that this is impossible; cf.\ Example \ref{ex:elliptic_setup}(2).
\end{ex}

\begin{rem}[Non-contractible components in the exceptional locus]\label{rem:non-contractible_component}
Note that in Example \ref{ex:problems_with_reordering_2nd_kind} we obtain a birational contraction from a smooth surface whose exceptional locus consists of two smooth irreducible curves (which are $\Q$-Cartier, as the surface is smooth), but one of them cannot be contracted onto an algebraic surface.
\end{rem}

\begin{lem}[Advancing contractions of the second kind]\label{lem:reduntant_2nd_kind} 
Let $\alpha\:(X,D)\to (\ov X,\ov D)$ be a pure partial peeling morphism on a log surface and let $\bar \gamma\:(\ov X,\ov D)\to (\ov Y, \ov B)$ be a partial MMP run of the second kind. Assume that $K_X$ is $\Q$-Cartier. Then there exists a partial MMP run of the second kind $\gamma\:(X,D)\to (Y,B)$ with $\alpha_*^{-1}\Exc \ov \gamma\leq \Exc \gamma$ onto a log surface with $K_Y$- $\Q$-Cartier and a pure partial peeling morphism $\ov \alpha$ such that the following diagram commutes
\begin{equation}
\begin{gathered}
\xymatrix{(X,D)\ar[d]^{\alpha} \ar[r]^-{\gamma} & (Y,B) \ar[d]^-{\ov \alpha} \\
 (\ov X,\ov D)\ar[r]^-{\bar{\gamma}} &  (\ov Y,\ov B)}
\end{gathered}
\end{equation}
If $\alpha$ is a peeling then $\ov \alpha$ is a peeling.
\end{lem}

\begin{proof} By induction we may assume that $\Exc \ov \gamma$ consists of a single log exceptional curve $\ov \ll$ of the first or second kind which is $\Q$-Cartier. Since $\alpha$ is a partial MMP run, each component of $E=\Exc \alpha$ is $\Q$-Cartier. Let $\ov D'$ be the proper transform of $\ov D$ on $X$. 

We first prove that there exists $\gamma'$, the contraction of $\ll:=\alpha_*^{-1}\ov \ll$, and that it maps $\Q$-Cartier divisors to $\Q$-Cartier divisors. Since $\ov \ll$ is $\Q$-Cartier, $\ll$ is $\Q$-Cartier. We may assume that $E\neq 0$ and that each connected component of $E$ meets $\ll$. Put $\tilde E=\cf_X(\ov X, \ov D)$. We have $$\tilde E=\alpha^*(K_{\ov X}+\ov D)-(K_X+\ov D'),$$ so for every component $E_0$ of $E$ we get $E_0\cdot \tilde E=-E_0\cdot (K_X+\ov D')\leq -E_0\cdot \ov D'\leq 0$, because $\alpha$ is pure. By Lemma \ref{lem:negativity} we obtain $\tilde E\geq 0$. We have $\cf(E_0;\ov X,\ov D)<\cf(E_0;X,D)=\coeff_{E_0}(D)\leq 1$, so $\cf(\ov X, \ov D)\leq 1$. We deduce that $0\leq \tilde E\leq E$ and $\ll$ is a log exceptional curve of the first or second kind on $(X,\ov D'+\tilde E)$. The latter is a log surface, as $K_X+D$ is $\Q$-Cartier and each component of $E$ is $\Q$-Cartier. 

If $\ll\cdot \tilde E\neq 0$ then $\ll$ is not a connected component of $\ov D'+\tilde E$, hence we are done by Lemma \ref{lem:contractibility_2nd}. We may thus assume that $\ll\cdot \tilde E=0$. Since $\ll$ meets every connected component of $\Supp E$, Lemma \ref{lem:negativity} implies that $\tilde E=0$. We have now $(\alpha^*K_{\ov X}-K_X)+(\alpha^*\ov D-\ov D')=0$, so by Corollary \ref{cor:pure_peeling} we get $\alpha^*K_{\ov X}=K_X$ and $\alpha^*\ov D=\ov D'$. It follows that $\ov D'\cdot E=0$. Since $\ll\cdot E>0$, the curve $\ll$ is not a component of $\ov D'$. Again, we are done by Lemma \ref{lem:contractibility_2nd}.

Let $\alpha'=\ov\gamma\circ \alpha\circ (\gamma')^{-1}$ be the induced rational map. It is regular due to the normality of $\gamma'(X)$. Since $\gamma'$ maps $\Q$-Cartier divisors to $\Q$-Cartier divisors and since for every prime exceptional component $E_0$ of $E$ we have $\cf(E_0;\alpha'(X),\alpha'_*D)<\cf(E_0;X,D)$, by Corollary \ref{cor:reordering_MMP}(1) the morphism $\alpha'$ is a partial peeling. Since $K_{\alpha'(X)}$ is $\Q$-Cartier, we have a decomposition $\alpha'=\ov \alpha\circ \alpha''$, where $\alpha''$ is an almost minimalization of $\alpha'$. Then $\gamma=\alpha''\circ\gamma'$ and $\ov \alpha$ are as required.

We have $\ov \gamma^{*}(K_{\ov Y}+\ov B)=K_{\ov X}+\ov D$, which implies that a proper transform of a log exceptional curve on $(\ov Y, \ov B)$ is log exceptional on $(\ov X,\ov D)$. Therefore, if $\alpha$ is a peeling then $\ov \alpha$ is a peeling.
\end{proof}

In an analogy to Definition \ref{dfn:redundant_and_almost-log-exc} we introduce the following definition.

\begin{dfn}[Redundant and almost log exceptional curves of the second kind]\label{dfn:almost_and_superlfuous_2nd}
Let $\alpha$ be a pure partial peeling of a log surface $(X,D)$ with $K_X$ being $\Q$-Cartier and let $\ll\subseteq X$ be a curve.
\begin{enumerate}[(1)]
\item We say that $\ll$ is \emph{$\alpha$-almost log exceptional of the second kind} if $\ll\nleq D$, $\ll\cdot K_X\neq 0$ and $\alpha(\ll)$ is log exceptional of the second kind on $(\alpha(X),\alpha_*D)$. 
\item We say that $\ll$ is \emph{$\alpha$-redundant of the second kind} if $\ll\leq D$, $\ll\cdot K_X<0$ and $\alpha(\ll)$ is log exceptional of the second kind  on $(\alpha(X),\alpha_*D)$,
\end{enumerate}
A curve is \emph{almost log exceptional of the second kind} if it is $\alpha$-almost log exceptional of the second kind for some pure partial peeling $\alpha$. A curve is \emph{redundant of the second kind} if it is $\alpha$-redundant of the second kind for some pure partial peeling $\alpha$.
\end{dfn}

Recall that if $\ll$ is an almost log exceptional curve of the first or second kind on $(X,D)$ then by Lemma \ref{lem:a.l.e._is_redundant} we have $\ll\cdot K_X<0$. In particular, almost log exceptional curves of the second kind and redundant curves of the second kind can be contracted.

The following example shows in particular that in the situation of Lemma \ref{lem:reduntant_2nd_kind} the morphism $\gamma$ may contract more than just $\alpha_*^{-1}\Exc \ov \gamma$.

\begin{ex}[Non-purity of the induced peeling]\label{ex:a.l.e._2nd_type}
Let $X$ be a smooth surface containing a chain $[1,2]$, that is, two smooth rational curves $\ll$ and $D$, such that $\ll^2=-1$, $D^2=-2$ and $\ll\cdot D=1$. The peeling morphism $\alpha\:(X,D)\to (\ov X, \ov D)$ is the contraction of $D$ (here $\ov D=0$), so it is pure and $\alpha^*(K_{\ov X}+\ov D)\sim K_X$. It follows that $\ll$ is almost log exceptional of the first kind. After the contraction of $\ll$ the image of $D$ is a $(-1)$-curve, so the induced peeling contracting the image of $D$ is not pure.

Similarly, consider a smooth surface $X$ containing a chain $[2,1,0]$. Let $D_1=[2]$, $\ll=[1]$ and $D_2=[0]$ be its components. Put $D=D_1+D_2$. The peeling morphism $\alpha\:(X,D)\to (\ov X, \ov D)$ is the contraction of $D_1$, so it is pure and $\alpha^*(K_{\ov X}+\ov D)\sim K_X+D_2$. It follows that now $\ll$ is almost log exceptional of the second kind. Again, after the contraction of $\ll$ the image of $D$ is $[1,-1]$, so the induced peeling contracts the $(-1)$-curve, hence is not pure.
\end{ex}

\begin{samepage}
\begin{dfn}[Almost minimalization of the second kind]\label{dfn:almost_minimalization_2}
Let $\psi\:(X,D)\to (\ov X,\ov D)$ be a partial MMP run of the second kind with $\Q$-Cartier $K_X$. 
\begin{enumerate}[(1)]
\item We  call $\psi_\am$ a \emph{partial almost minimalization of the second kind of $(X,D)$}.
\item If $\Exc \psi\subseteq \Supp D$ then $\psi$ is a \emph{partial peeling of $(X,D)$ of the second kind} and $\psi_\am$ a \emph{partial $\psi$-squeezing of the second kind}. 
\item If $\psi$ is a maximal partial peeling of the second kind then it is called \emph{peeling of the second kind} and $\psi_\am$ is called a \emph{squeezing of the second kind}.
\item If $\psi$ is maximal then $\psi$ (respectively, $\psi_\am$) is called a \emph{minimalization (respectively, almost minimalization)} of the second kind of $(X,D)$.
\item In the above case $(\psi_\am(X),(\psi_\am)_*D)$ is called an \emph{almost minimal model of $(X,D)$ of the second kind}.
\end{enumerate}
\end{dfn}
\end{samepage}

\begin{rem}\label{rem:aMM_2nd_kind}
By Corollary \ref{cor:improving_ld_gives_MMP} if $K_X$ is $\Q$-Cartier and $\psi$ is a partial MMP run of the second kind then $\psi=\psi_2\circ \psi_1$, where $\psi_1$ is a partial MMP run (of the first kind) and $\psi_2$ is log crepant partial MMP run of the second kind. By \eqref{eq:composing_almost_min} it follows that an almost minimalization of the second kind is an extension of an almost minimalization of the first kind. 
\end{rem}

\medskip
\section{Almost minimalization for reduced boundaries}\label{ssec:reduced_boundaries}

We want to obtain a more explicit description of the process of almost minimalization for log surfaces whose boundary is \emph{uniform}, that is, all coefficients of prime components are equal to some fixed number $r\in [0,1]$. Before we do this we need to review the well-known case $r=1$ in detail.

\medskip
\subsection{Discriminants and log canonical subdivisors}\label{ssec:twigs_forks_barks}

Let $X$ be a smooth projective surface and $D$ a reduced divisor on $X$. We introduce or recall (see \cite{Fujita-noncomplete_surfaces}, \cite{Miyan-OpenSurf}) some notions and notation related to the geometry of divisors and specific subdivisors, which will be used later to discuss coefficient divisors and peeling morphisms. We do not assume that $D$ is a simple normal crossing divisor. We denote the number of components of $D$ by $\#D$.

By $p_a(D)$ we denote the arithmetic genus of $D$, that is, $p_a(D):=\frac{1}{2}D\cdot(K_X+D)+1$. We have $$p_a(D_1+D_2)=p_a(D_1)+p_a(D_2)+D_1\cdot D_2-1$$ for any divisors $D_1$, $D_2$ on $X$. Given a reduced subdivisor $T\leq D$ we call 
\begin{equation}\label{eq:br_number}
\beta_D(T)=T\cdot (D-T)
\end{equation}
the \emph{branching number of $T$ in $D$}. A \emph{tip of $D$} is a component with $\beta_D\leq 1$ and a \emph{branching component of $D$} is a component with $\beta_D\geq 3$. We say that $D$ is \emph{rational} if all its components are rational. A component of $D$ is \emph{admissible} if it is smooth rational and its self-intersection number is at most $(-2)$. 

Let $D_1,\ldots, D_n$ be the components of $D$. We put $Q(D)=[D_i\cdot D_j]_{1\leq i,j\leq n}$ and we call $$d(D):=\det(-Q(D)),$$ where $d(0):=1$, the \emph{discriminant of $D$}; it does not depend on the chosen order of components.

\begin{lem}[Splitting formula for discriminants]\label{lem:d(D_1+D_2)}
Assume $D_1$, $D_2$ are two reduced divisors on a smooth projective surface which have no common component and which intersect in unique components $T_1\leq D_1$ and $T_2\leq D_2$, respectively. Then
\begin{equation}\label{eq:d(D1+D2)}
d(D_1+D_2)=d(D_1)d(D_2)-(T_1\cdot T_2) d(D_1-T_1)d(D_2-T_2).
\end{equation}
\end{lem}

\begin{proof}The proof follows from the additivity of the determinant function with respect to column addition and its behavior on block-triangular matrices. 
\end{proof}

In particular, if $D$ is a reduced divisor with a tip $D_1$ and this tip meets a component $D_2\leq D$ then
\begin{equation}\label{lema:d(D-tip)}
d(D)=(-D_1^2)d(D-D_1)-d(D-D_1-D_2).
\end{equation}

\bigskip
If $\Supp D$ is connected then we call $D$ a \emph{rational tree} if $p_a(D)=0$ and a we call $D$ a \emph{rational cycle} if $p_a(D)=1$, $D$ has no branching component and is not a smooth elliptic curve. A rational tree is simply a rational snc-divisor with a connected support and with no rational cycle as a subdivisor. Each component of a reducible rational cycle is rational and has $\beta_D=2$. A rational cycle is \emph{degenerate} if it is not snc. In this case it is a rational curve with a unique singular point - a node or a simple cusp, a sum of two smooth rational curves intersecting at a unique point with multiplicity two or a triple of smooth rational curves passing through a common point, each two meeting normally, see Figure \ref{fig:degenerate_cycle}.

\medskip 
\begin{figure}[H]
\begin{tikzpicture}[scale=0.8]
	\draw (-1,1) to[out=-60,in=180] (0.5,-0.5) to[out=0,in=-90] (1,0) to[out=90,in=0] (0.5,0.5) to[out=180,in=60] (-1,-1);
\end{tikzpicture}
\hspace{1cm} 
\begin{tikzpicture}[scale=0.8]
	\draw (-1,1) to[out=-60,in=180] (0.5,0) to[out=180,in=60] (-1,-1);
\end{tikzpicture}
\hspace{1cm} 
\begin{tikzpicture}[scale=0.8]
	\draw (-1,1) to[out=0,in=90] (0,0) to[out=-90,in=0] (-1,-1);
	\draw (1,1) to[out=180,in=90] (0,0) to[out=-90,in=180] (1,-1);
\end{tikzpicture}
\hspace{1cm} 
\begin{tikzpicture}[scale=0.8]
	\draw (-1,0) -- (1,0);
	\draw (-1,1) -- (1,-1);
	\draw (-1,-1) -- (1,1);
\end{tikzpicture}
\caption{Degenerate rational cycles.} \label{fig:degenerate_cycle}
\end{figure}
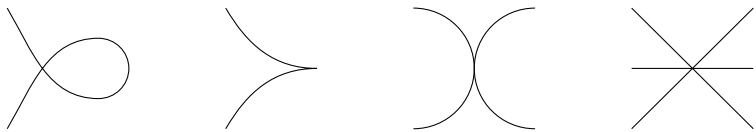

 A rational tree with no branching component is a \emph{rational chain}. If $D$ is a rational chain or a rational cycle, and the components are ordered so that $D_i$ meets $D_{i+1}$ for $i\in\{1,\ldots, n-1\}$, then we write $D=[-D_1^2,\ldots, -D_n^2]$ in the first case and $D=((-D_1^2,\ldots, -D_n^2))$ in the second case. A sequence consisting of an integer $r$ repeated $k$ times will be abbreviated by $(r)_{k}$. A rational chain and a rational cycle are \emph{admissible} if they have admissible components and are negative definite. For chains the second condition is in fact redundant and for reducible rational cycles it is equivalent to one of the inequalities $-D_i^2\geq 2$ being strict. 

A component $T$ of $D$ is called \emph{superfluous} if it is a $(-1)$-curve meeting at most two other components of $D$, each at most once. Equivalently, after the contraction of $T$ the image of $T$ is a simple normal crossing point of the image of $D$. Note that a log smooth completion of a smooth quasi-projective surface is \emph{minimal} (does not dominate non-trivially some other log smooth completion) if and only if the boundary contains no superfluous component which is not its connected component. 

\smallskip
A rational chain $T$ is \emph{ordered} if it has a fixed  ordering on its components $T\cp{1},\dots, T\cp{m}$ such that $T\cp{i}\cdot T\cp{i+1}=1$ for $i\in \{1,\dots, m-1\}$. By $T\trp$ we denote the same chain with the opposite order of components. The \emph{tip} (a \emph{first tip}, if there is any chance of confusion) of a nonzero ordered chain is by definition $\tip{T}=T\cp{1}$. For $i\geq 1$ we put 
\begin{equation}\label{eq:di(T)}
d\cp{i}(T)=d(T\cp{i+1}+\dots+T\cp{m})\quad \text{and}\quad d_{(i)}(T)=d(T\cp{1}+\dots+T\cp{i-1})
\end{equation}
with $d\cp{i}(0)=d_{(i)}(0)=0$. We put also $d'(T)=d\cp{1}(T)=d(T-\tip{T})$. 
Let $T$ be an admissible ordered chain. By Lemma \ref{lem:d(D_1+D_2)}
\begin{equation}\label{eq:d(chain)}
d(T)=(-\tip{T})^2d'(T)-d'(T-\tip{T}).
\end{equation} We have $0\leq d'(T)<d(T)$ and $\gcd(d(T),d'(T))=1$, see \cite[3.5]{Fujita-noncomplete_surfaces}. 
We put 
\begin{equation}
\delta(T)=\frac{1}{d(T)}\in \Q\cap (0,1]\quadtext{and} \ind(T)=\frac{d'(T)}{d(T)}\in \Q\cap [0,1)
\end{equation} 
and we call $\ind(T)$ the \emph{inductance of $T$}. 

\medskip
Given a reduced divisor $D$ we now define some specific subdivisors $T\leq D$ and depict their \emph{extended dual graphs} (see \cite[Definition 4.6]{KollarMori-bir_geom}). White graph vertices correspond to components of $T$ and their weights (if displayed) are negatives of their self-intersection numbers. Black vertices represent $D-T$ (they are not necessarily distinct). The number of edges between two vertices is the intersection number of the represented components, see Fig.\ \ref{fig:dual_graphs}.

A nonzero rational chain $T\leq D$ whose all components are non-branching in $D$, that is $\beta_D(T)\leq 2$, is called a (rational) \emph{twig of $D$} if some component of $T$ is a tip of $D$ ($\beta_D\leq 1$) and a \emph{segment of $D$} otherwise. A segment is \emph{non-degenerate} if $T$ meets $D$ normally (this holds for instance if $T$ has at least three components), for degenerate segments see Figure \ref{fig:degenerate_segment}.
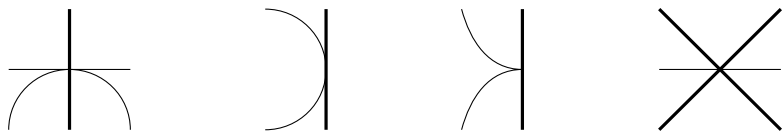
\begin{figure}[h]
\begin{tikzpicture}[scale=0.8]
	\draw (-1,0) -- (1,0);
	\draw (-1,-1) to[out=90,in=180] (0,0) to[out=0,in=90] (1,-1);
	\draw[very thick] (0,-1) -- (0,1);
\end{tikzpicture}
\hspace{1.5cm} 
\begin{tikzpicture}[scale=0.8]
	\draw (-1,1) to[out=0,in=90] (0,0) to[out=-90,in=0] (-1,-1);
	\draw[very thick] (0,-1) -- (0,1);
\end{tikzpicture}
\hspace{1.5cm} 
\begin{tikzpicture}[scale=0.8]
	\draw (-1,1) to[out=-75,in=180] (0,0) to[out=180,in=75] (-1,-1);
	\draw[very thick] (0,-1) -- (0,1);
\end{tikzpicture}
\hspace{1.5cm} 
\begin{tikzpicture}[scale=0.8]
	\draw (-1,0) -- (1,0);
	\draw[very thick] (-1,1) -- (1,-1);
	\draw[very thick] (-1,-1) -- (1,1);
\end{tikzpicture}
\caption{Degenerate segments. Thick lines indicate their components.} \label{fig:degenerate_segment}
\end{figure}
A twig is a \emph{maximal twig} of $D$ if it is not properly contained in another twig of $D$ in the sense of supports. A twig whose support is a connected component of $D$ is called a \emph{rod} of $D$. A twig which is not a rod comes (and will be considered) with a unique order in which $\tip{T}$ is a tip of $D$. 

Let $F$ be a rational tree with a unique branching component $B$ and three maximal twigs $T_1$, $T_2$, $T_3$. Then we call $F$ a \emph{(rational) fork}, we write $F=\langle B;T_1,T_2,T_3\rangle$ and we say that $F$ is of type $(-B^2;d(T_1),d(T_2),d(T_3))$. We put \begin{equation}\label{eq:delta_def}
\delta(F):=\delta(T_1)+\delta(T_2)+\delta(T_3).
\end{equation} By a \emph{$(-2)$-fork (respectively, a $(-2)$-chain)} we mean an fork (respectively, a chain) consisting of $(-2)$-curves. By a \emph{fork of $D$} we mean a fork which is a connected components of $D$. By Lemma \ref{lem:d(D_1+D_2)} 
\begin{equation}\label{eq:fork_d}
d(F)=d(T_1)d(T_2)d(T_3)(-B^2-\ind(T_1\trp)-\ind(T_2\trp)-\ind(T_3\trp)).
\end{equation}
A fork $F$ with admissible components is called an \emph{admissible fork} if $\delta(F)>1$ and is called a \emph{log canonical fork} if $\delta(F)=1$ and not all components of $F$ are $(-2)$-curves. 

A rational tree $T\leq D$ is called a \emph{bench of $D$} if $T$ is a connected component of $D$ which contains a chain (called a \emph{central chain}) $C=C_1+\ldots+C_n$, $n\geq 1$ with tips $C_1$, $C_n$ such that $T-C=T_1+T_2+T_3+T_4$, $T_i=[2]$ for $i=1,2,3,4$, $T_i\cdot C_1=1$ for $i=1,2$ and $T_i\cdot C_n=1$ for $i=3,4$. A bench is \emph{log canonical} if the central chain is admissible and does not consist of $(-2)$-curves only. 

A rational tree $T\leq D$ is called a \emph{half-bench of $D$} if $T$ contains a chain (again called a \emph{central chain}) $C=C_1+\ldots+C_n$, $n\geq 1$ with tips $C_1$, $C_n$ such that $T-C=T_1+T_2$, $T_i=[2]$, $T_i\cdot C_1=1$ for $i=1,2$ and  $T\cdot (D-T)=C_n\cdot (D-T)=1$. In particular, a half-bench of $D$ is a fork of type $(b;2,2,t)$ or a chain $[2,b,2]$ for some $b,t\geq 2$. We call it a \emph{log canonical} half-bench if $C$ is admissible.
\medskip
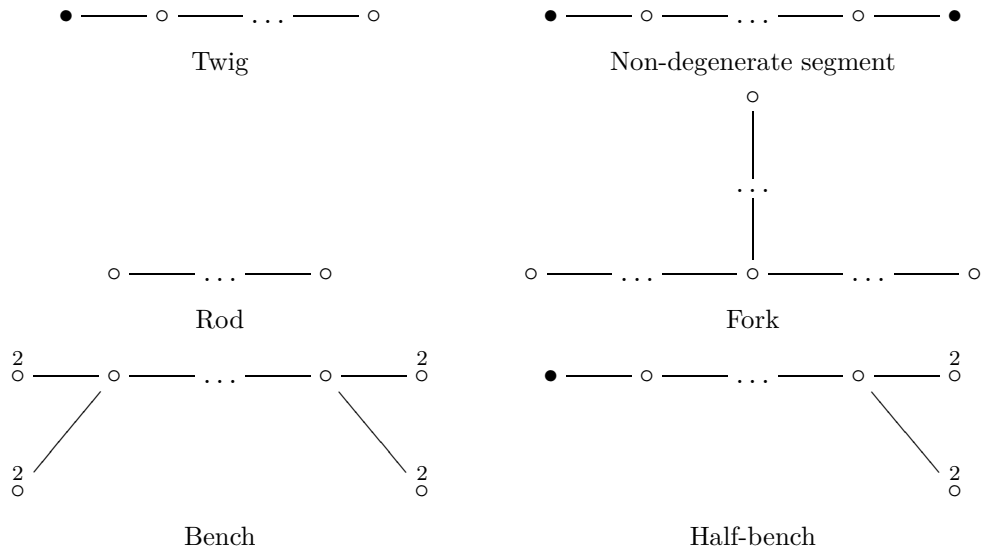
\begin{figure}[H]
\centering
\subcaptionbox*{Twig}[.4\textwidth]
{$\xymatrix{{\bullet}\ar@{-}[r] & {\circ}\ar@{-}[r] & \ldots\ar@{-}[r] & {\circ}}$}
\subcaptionbox*{Non-degenerate segment}[.4\textwidth]
  {$\xymatrix{{\bullet}\ar@{-}[r] & {\circ}\ar@{-}[r] & \ldots\ar@{-}[r] & {\circ}\ar@{-}[r] & {\bullet}}$}
\subcaptionbox*{Rod}[.4\textwidth]
{$\xymatrix{ {\circ}\ar@{-}[r] & \ldots\ar@{-}[r] & {\circ}}$}
\subcaptionbox*{Fork}[.4\textwidth]
{$\xymatrix{
{} & {} & {\circ}\ar@{-}[d]  & {} & {}\\
 {} & {} & {\ldots}\ar@{-}[d]  & {} & {}\\
 {\circ}\ar@{-}[r] & {\ldots}\ar@{-}[r] & {\circ}\ar@{-}[r]& \ldots\ar@{-}[r] & {\circ}}$}
\subcaptionbox*{Bench}[.4\textwidth]
{$\xymatrix{
{\overset{2}{\circ}}\ar@{-}[r] &{\circ}\ar@{-}[r]\ar@{-}[dl] & {\ldots}\ar@{-}[r] & {\circ}\ar@{-}[r]\ar@{-}[dr] &{\overset{2}{\circ}}\\
{\overset{2}{\circ}} & {} & {} &  {} & {\overset{2}{\circ}}}$}
\subcaptionbox*{Half-bench}[.4\textwidth]
{$\xymatrix{
 {\bullet}\ar@{-}[r] &{\circ}\ar@{-}[r]& {\ldots}\ar@{-}[r] & {\circ}\ar@{-}[r]\ar@{-}[dr] &{\overset{2}{\circ}}\\
  {} & {} & {} &  {} & {\overset{2}{\circ}}}$}
\caption{Log canonical subdivisors.}\label{fig:dual_graphs}
\end{figure}

\begin{rem}\label{rem:log_can_graps_neg-def}\ 
\begin{enumerate}[(1)]
\item For every admissible chain $T$ we have $\delta(T)+\ind(T)\leq 1$ with the equality for a $(-2)$-chain only.  Hence for a fork $F$ with admissible twigs and $\delta(F)>1$ we infer from \eqref{eq:fork_d} and from Sylvester's criterion that $F$ is negative definite if and only if $F$ is admissible (equivalently $-B^2\geq 2$).
\item 
If a fork with $\delta=1$ or a bench consists of $(-2)$-curves then its discriminant vanishes. By \eqref{eq:fork_d} and by elementary properties of determinants it follows that a fork with admissible twigs and $\delta=1$ is log canonical (equivalently, is not a $(-2)$-fork) if and only if it is negative definite. Similarly, a bench with an admissible central chain is admissible (equivalently, its central chain is not a $(-2)$-chain) if and only if it is negative definite.
\end{enumerate}
\end{rem}

\begin{lem}[Log terminal and log canonical subdivisors] \label{lem:lc_singularities}
Let $p\in \ov X$ be a germ of a normal singular surface and let $\ov D$ be a reduced divisor on $\ov X$. Let $E$ be the exceptional divisor of a minimal resolution $\pi\:X \to \ov X$. Put $D=\pi_*^{-1}\ov D+E$. We have $\cf(\cE(\pi);\ov X, \ov D)<1$ (see Section \ref{ssec:discrepancies}) if and only if one of the following holds:
\begin{enumerate}[(1)]
\item $\ov D=0$ and $E$ is either an admissible fork or an admissible rod of $D$,
\item $\ov D\neq 0$ and $E$ is an admissible twig of $D$.
\end{enumerate}
We have $\cf(\cE(\pi);\ov X, \ov D)=1$ if and only if one of the following holds:
\begin{enumerate}[(1)]
\setcounter{enumi}{2}
\item $\ov D=0$ and $E$ is one of the following:
\begin{enumerate}[(a)]
\item a smooth elliptic curve with negative self-intersection number,
\item a (possibly degenerate) admissible rational cycle,
\item a log canonical fork,
\item a log canonical bench.
\end{enumerate}
\item $\ov D\neq 0$, $E\cdot (D-E)=1$ and $E$ is a log canonical half-bench of $D$.
\item $\ov D\neq 0$, $E\cdot (D-E)=2$ and $E$ is a segment of $D$ (possibly degenerate). 
\end{enumerate}
Conversely, each nonzero divisor $E$ as above has a negative definite intersection matrix and in cases other than (3a) and (3b) it contracts to a rational singularity.
\end{lem}

\begin{proof}A direct arithmetic proof independent of the characteristic of the base field is given in \cite[3.2.7, 3.4.1]{Flips_and_abundance}; cf.\ \cite[3.39, 3.40]{Kollar_singularities_of_MMP}. Since we do not assume that $p\in (\ov X, \ov D)$ is log canonical (we consider a minimal resolution, not a minimal log resolution), in (3b) we allow degenerate cycles; necessary minor corrections of the arguments in \cite[(3.1.4), (3.1.5)]{Flips_and_abundance} can be done easily. Concerning the rationality of log canonical singularities see also \cite[2.3]{Artin-Contractibility} and \cite[2.28]{Kollar_singularities_of_MMP}.
\end{proof}

Let $\kk$ be the base field. In case $\cha \kk=0$ log terminal singularities are locally analytically isomorphic to quotient singularities, that is, the ones obtained as quotients of $\A^{2}$ by the actions of finite subgroups of $\operatorname{GL}(2,\kk)$, see \cite[Satz 2.10]{Brieskorn_rat_sing}, cf.\ \cite[\S 7.4]{Ishii_singularities}. For instance, for an admissible chain $T$, the action of the cyclic group $\langle \zeta \rangle<(\kk^{*},1)$ of order $d(T)$ given by $\zeta \cdot (x,y)=(\zeta x, \zeta^{d'(T)} y)$ produces a singularity with $T$ as an exceptional divisor of the minimal resolution, see \cite[III.5]{BHPV_complex_surfaces}. 

\medskip
\subsection{Comparing graphs of log terminal singularities}

Let $\pi\:X\to \ov X$ be a minimal resolution of a germ of a log surface $(\ov X,\ov D)$ at some point $p$. The associated weighted decorated dual graph $\cG(\ov X,\ov D,p)$ is the dual graph of the exceptional divisor $E$ of $\pi$ with weights of its vertices being negatives of the self-intersection numbers of the corresponding components of $E$ and decorations of vertices being equal to the intersections of these components with the proper transform of $\ov D$. Such a graph is called \emph{du Val} if and only if $\ov D=0$ and the germ is canonical. On the set of vertices of this graph we have also the $\Q$-valued function $\cf_\cG$ computing coefficients of the corresponding components of $E$ over $(\ov X,\ov D)$. By a weighted decorated subgraph of $\cG(\ov X,\ov D,p)$ we mean a subgraph od the dual graph of $E$ with weights and decorations not bigger than the ones of the original graph. 

\begin{lem}[Alexeev]\label{lem:Alexeev_dlt} Let $\cF$ and $\cG$ be the weighted decorated dual graphs of two germs of log terminal log surfaces. If $\cF$ is a subgraph of $\cG$ not equal to $\cG$ then $\cf_\cF<\cf_\cG|_{\cF}$, unless $\cG$ is du Val, in which case $\cf_{\cF}=\cf_{\cG}=0$.
\end{lem}

\begin{proof}
See \cite[3.1.3]{Flips_and_abundance} and \cite[3.7(ii)-(iv)]{Alexeev-Fractional_del_Pezzo}, cf.\ Lemma \ref{lem:log_terminal} in the Appendix (Section \ref{sec:d_and_ld}).
\end{proof}

\medskip
\subsection{Barks}
To write down explicit compact formulas for the coefficient divisor in case of uniform boundaries we use \emph{barks} (cf.\ \cite[2.3.5.2]{Miyan-OpenSurf}). Let $T=T\cp{1}+\cdots + T\cp{m}$ be an admissible ordered chain. We define the \emph{bark of $T$} as (see \eqref{eq:di(T)})
\begin{equation}\label{eq:bark}
\Bk' T=\sum_{i=1}^{m}\frac{d\cp{i}(T)}{d(T)}T\cp{i}.
\end{equation}
In particular, the coefficients of $\tip{T}$ and $T\cp{m}$ in $\Bk' T$ are $\ind (T)$ and $\delta(T)$, respectively. We put
\begin{equation}\label{eq:Bk_T}
\Bk\trp T=\Bk' (T\trp).
\end{equation}

Denoting by $\delta_i^j$ the Kronecker delta, from \eqref{eq:d(chain)} we infer that 
\begin{equation}\label{eq:Bk(chain)}
T\cp{i}\cdot \Bk' T=-\delta_i^1. 
\end{equation}

\begin{dfn}[The bark] \label{dfn:Bark} 
Assume that $X$ is a smooth surface and $D$ is a reduced divisor on $X$. 
\begin{enumerate}[(1)]
\item If $E$ is a maximal admissible twig of $D$ but not a rod of $D$, we put $\Bk_DE=\Bk' E$.
\item If $E$ is an admissible rod of $D$ then we pick any order which makes it an ordered twig and we put $\Bk_DE=\Bk' E+\Bk\trp E$.
\item If $E$ is an admissible fork of $D$ we denote its maximal twigs by $T_1$, $T_2$, $T_3$ and its central component by $E_0$ and we put 
\begin{equation*}
\Bk_D E=u(E_0+\sum_{i=1}^3 \Bk\trp T_i)+\sum_{i=1}^3\Bk' T_i, \quadtext{where} u=\frac{\sum_{i=1}^3 \delta(T_i)-1}{-E_0^2-\sum_{i=1}^3 \ind (T_i\trp)}.
\end{equation*}
\end{enumerate}
We extend the definition of $\Bk_D$ additively for disjoint sums of admissible twigs, rods and forks. Finally, we define the \emph{bark of $D$} by $\Bk D:=\Bk_D (\Exc \pi)$, where $\pi$ is the unique pure peeling morphism for $(X,D)$, see Lemma \ref{lem:peel_sq_red_bdry}.
\end{dfn}

\begin{rem}\label{rmk:Bark}
In Definition \ref{dfn:Bark}(3) we have $0<u\leq 1$. To see this note that $\ind(T_i\trp)+\delta(T_i)\leq 1$ for $i=1,2,3$, which gives $-E_0^2-\sum_{i=1}^3 \ind (T_i\trp)\geq - E_0^2+\sum_{i=1}^3 \delta(T_i)-3\geq \sum_{i=1}^3 \delta(T_i)-1>0$.
\end{rem}

\medskip
\subsection{Peeling, redundant and almost log exceptional curves}\label{ssec:Reduced boundary}

We now discuss notions related to almost minimality for reduced boundaries. Many results concerning minimal and almost minimal models of log surfaces with reduced boundary were obtained in the 80s in particular by T. Fujita, M. Miyanishi, S. Tsunoda and F. Sakai, see for instance \cite{Fujita-noncomplete_surfaces}, \cite{Fuj-Zariski_prob}, \cite{MiTs-Platonic_fibration},  \cite{Sakai-normal_surfaces}. The results we present below are close to \cite[2.3.3-3.5]{Miyan-OpenSurf}. We do not claim originality, but the proposed formulation and line of reasoning will be used to obtain analogous results for $(1-r)$-log canonical surfaces with uniform boundary $rD$. For a definition of a peeling of the first and second kind see Definitions \ref{def:peel_squeeze}(1) and  \ref{dfn:almost_minimalization_2}(2). 

\begin{lem}[Peeling and squeezing for a reduced boundary] \label{lem:peel_sq_red_bdry}
Assume that $X$ is a smooth surface and $D$ is a reduced divisor on $X$. 
\begin{enumerate}[(1)]
\item A contraction of some number of admissible forks, admissible rods and admissible twigs of $D$ is a pure partial peeling. Every pure partial peeling is of this type.
\item A contraction of some number of divisors $E$ as in Lemma \ref{lem:lc_singularities} (if exists) is a pure partial peeling of the second kind. Every pure partial peeling of the second kind is of this type. 
\item $L\leq D$ is redundant of the first kind if and only if it is a $(-1)$-curve such that either $\beta_D(L)\leq 1$ or $\beta_D(L)=2$ and $L$ meets some admissible twig of $D$. 
\item $L\leq D$ is redundant of the second kind if and only if it is a $(-1)$-curve with $\beta_D(L)=2$ which meets no admissible twig of $D$. 
\end{enumerate}
In particular, if $D$ contains no superfluous $(-1)$-curve (is snc-minimal) then $(X,D)$ is squeezed and if $(X,D)$ is squeezed then it has a unique peeling.
\end{lem}

\begin{proof}
(1), (2) If $\alpha$ is a pure partial peeling of the second kind then it is also a minimal resolution of $\ov X$, hence $\Exc \alpha$ is as in Lemma \ref{lem:lc_singularities}. Conversely, let $\alpha\:(X,D)\to (\ov X,\ov D)$ denote the contraction as in (1) or (2). We may assume that $\Exc \alpha$ is connected. By Lemma \ref{lem:lc_singularities} for every component $E$ of $\Exc \alpha$ we have $\cf(E;\ov X,\ov D)\leq 1=\cf(E;X,D)$ and the inequality is strict in case (1). Since $X$ is smooth, $\Exc \alpha$ has $\Q$-Cartier components and $(X,D)$ is GMRLC. Then $\alpha$ is a partial MMP run of the second kind by Corollary \ref{cor:improving_ld_gives_MMP} and it is of the first kind in case (1). 

(3), (4) Assume that $L$ is $\alpha$-redundant of the first or second kind for some pure partial peeling $\alpha$ (of the first kind). It is a $(-1)$-curve, as $L\cdot K_X<0$. Denote by $E$ the sum of connected components of $\Exc \alpha$ meeting $L$. Since $X$ is $\Q$-factorial and $\alpha$ is a partial MMP (of the first kind), $\alpha(X)$ is $\Q$-factorial. Let $\sigma\:X\to \ov X$ be the contraction of $E+L$. Put $\ov D=f_*D$ and $q=\sigma(E+L)$. If $E=0$ then we get $0\geq L\cdot (K_X+D)=-2+\beta_D(L)$, which gives (3) or (4). We may thus assume that $E\neq 0$. By Lemma \ref{lem:ld_increasaes} coefficients of components of $E+L$ with respect to $(\ov X,\ov D)$ are smaller than $1$, hence by Corollary \ref{cor:improving_ld_gives_MMP}(1)  $\sigma$ is a partial peeling of $D+L$. By (1) $D+L$ has simple normal crossings in a neighborhood of $E+L$. In case $\ov X$ is smooth we infer that $\ov D$ has simple normal crossings at $q$. Assume that $\ov X$ is singular. Let $\pi\:\tilde X\to \ov X$ be the minimal resolution of singularities and let $\tilde D$ be the reduced total transform of $\ov D$. Then $\pi$ is a pure partial peeling of $\tilde D$ and by (1) $\tilde D$ has simple normal crossings in a neighborhood of $\Exc \pi$. It follows that $L$ contracts to an snc-point of the image of $D$, which means that $L$ is a superfluous $(-1)$-curve. This gives (3).

Finally, suppose that $(X,D)$ is squeezed and has more than one peeling morphism. By Corollary \ref{cor:reordering_MMP} there are connected components $E_1$ and $E_2$ of exceptional divisors of these peelings such that $E_1\nleq E_2$, $E_2\nleq E_1$ and $E_1$ meets $E_2$. In particular, $E_1$, $E_2$ are not connected components of $D$. By (1) $E_1$ and $E_2$ are admissible twigs of $D$, hence $(E_1+E_2)\redd$ is an admissible rod of $D$. It follows that both peelings contract $(E_1+E_2)\redd$; a contradiction.
\end{proof}

\begin{rem}Assume that $D$ as above is connected and squeezed of the second kind (i.e.,\ contains no $(-1)$-curve with $\beta_D\leq 2$). Then $(X,D)$ has a unique peeling of the second kind except in the following cases:
\begin{enumerate}[(1)]
\item $D$ is a bench consisting of $(-2)$-curves,
\item $D$ is a rational cycle (possibly degenerate) which either consists of $(-2)$-curves or is negative definite but non-contractible algebraically.
\end{enumerate}
\end{rem}

\begin{proof}Suppose that $(X,D)$ has at least two distinct peeling morphisms of the second kind. By Corollary \ref{cor:reordering_MMP} there are connected components $E_1$ and $E_2$ of their exceptional divisors such that $E_1\nleq E_2$, $E_2\nleq E_1$ and $E=(E_1+E_2)\redd$ is connected. In particular, $E_1$, $E_2$ are not connected components of $D$.  It follows that $E_1$, $E_2$ are of type (2), (4) or (5) in Lemma \ref{lem:lc_singularities}. Assume first that, say, $E_2$ contains a non-nc point of $D$. Then it is of type (5) with $\#E_2\leq 2$ and the non-nc point belongs to $E_1$, too. Then $D$ is a degenerate rational cycle with two or three components. By the maximality of peelings $E$ is not algebraically contractible, hence $D=E$. Since its components have self-intersection numbers at most $(-2)$, $D$ is semi-negative definite, which is a special case of (2). We may thus assume that $E$ contains only nc points of $D$. Since $E_1$ and $E_2$ are of type (2), (4) or (5) of Lemma \ref{lem:lc_singularities}, $E$ is a chain, a fork of type $(b;2,2,n)$, a bench or a non-degenerate rational cycle, and its components have self-intersection numbers at most $(-2)$. By the maximality of peelings $E$ is not algebraically contractible, so it is a bench or a rational cycle, hence $D=E$. This gives (1) or (2).
\end{proof}

The following lemma is a generalization of \cite[Section 2.3.6-8]{Miyan-OpenSurf}. The original more computational proof is for $D$ which is snc. 

\begin{lem}[Almost log exceptional curves for a reduced boundary] \label{lem:case_r=1}
Assume that $X$ is a smooth surface and $D$ is a reduced divisor on $X$. Let $A$ be an $\alpha$-almost log exceptional curve of the first or second kind, where $\alpha$ is a pure partial peeling with exceptional divisor $E$. Then $A$ is a $(-1)$-curve and one of the following holds.
\begin{enumerate}[(1)]
\item $A\cdot D\leq 1$ and if $A$ is of the first kind then the inequality is strict or the point of intersection belongs to some admissible twig, rod or fork of $D$,
\item $A\cdot D=2$ and there is a rod $E_1$ of $D$ meeting $A$ once, in a tip of $D$. Moreover, if $A$ meets another connected component $E_2$ of $E$ then one of the following holds:
\begin{enumerate}[(a)]
\item $E_2$ is an admissible twig of $D$. It meets $A$ in a tip of $D$ or $E_1$ consists of $(-2)$-curves.
\item $E_2$ is a rod of $D$.
\end{enumerate}
\end{enumerate}
\end{lem}

\begin{proof}  Let $\alpha\:(X,D)\to (\ov X,\ov D)$ be a partial peeling morphism for which $\alpha(A)$ is log exceptional of the first or second kind. Since $\alpha$ is pure, by Corollary \ref{cor:pure_peeling} we have $\alpha(A)\cdot K_{\ov X}\geq A\cdot K_X$. We have also $0\geq \alpha(A)\cdot (K_{\ov X}+\ov D)\geq \alpha(A)\cdot K_{\ov X}$. By Lemma \ref{lem:a.l.e._is_redundant}(1) we have $A\cdot K_X<0$. But $A^2<0$, so in each case $A$ is a $(-1)$-curve. We may assume that $E+A$ is connected. We may also assume that $E\neq 0$, as otherwise (1) holds. If $A\cdot D=1$ and $A$ is of the first kind then we have $\alpha(A)\cdot (K_{\ov X}+\ov D)<0=A\cdot (K_X+D)$, so the point of intersection of $A$ and $D$ belongs to $E$, which gives (1). 

We may thus assume that $A\cdot D\geq 2$. Let $\ov \gamma\:(\ov X, \ov D)\to (\ov Y,\ov B)$ be the contraction of the log exceptional curve $\alpha(A)$ and let $\alpha'\:(Y,B)\to (\ov Y,\ov B)$ be the minimal resolution of singularities. Put $E'=\Exc \alpha'$. Since $X$ is smooth, we have a morphism of log surfaces $\gamma\:(X,D)\to (Y,B)$, such that $\alpha'\circ\gamma=\ov \gamma\circ\alpha$. In particular, $\gamma_*(E+A)=E'$. By Corollary \ref{lem:eps-lc_is_respected}, $(\ov Y,\ov B)$ is log terminal, hence by Lemma \ref{lem:peel_sq_red_bdry}, $E'=0$ or $E'$ is an admissible twig, rod or fork of $B$. In particular, it contains no $(-1)$-curve, hence $\gamma$ is a composition of successive contractions of $(-1)$-curves in $E+A$ and its images. Moreover, since $B$ is a sum of rational trees and is nc in a neighborhood of $E'$, every $(-1)$-curve contracted by a blowdown being a factor of $\gamma$ is superfluous in the respective image of $D$. It follows that $A\cdot D\leq 2$, hence $A\cdot D=2$.

Since $E'$ is a sum of rational trees, so is $E+A$, hence $A$ meets each connected component of $E$ at most once. Since $E'$ is a sum of admissible twigs, rods and forks of $B$, one of the connected components of $E$ met by $A$, say $E_1$, is necessarily a rod of $D$ and it meets $A$ in a tip. Assume that $A$ meets some other connected component of $E$, say $E_2$. If $E_2$ is not a connected component of $D$ then it is a twig of $D$ such that $\gamma_*(E_1+A+E_2)$ is a twig of $B$, which gives (a). We may therefore assume that $E_2$ is an admissible fork or rod of $D$. Since $\gamma_*(E_1+A+E_2)$ is an admissible rod or fork of $B$, we get (a) or (b). 
\end{proof}

\begin{samepage}
\begin{rem}[Almost log exceptional vs log exceptional]\ 
\begin{enumerate}[(1)]
\item Let $A$ be as in Lemma \ref{lem:case_r=1}. If $A\cdot D=0$ then $A$ is log exceptional. If $A\cdot D=1$ then it is log exceptional of the second kind. If $A\cdot D=2$ then we have $A\cdot (K_X+D)>0$.
\item In Lemma \ref{lem:case_r=1}(2) the fact that $A+\Exc \alpha$ is negative definite and $A$ is $\alpha$-almost log exceptional (and hence that after the contraction of $\alpha(A)$ all coefficients decrease) gives additional restrictions on the weights of $E_1$ and $E_2$, see Example \ref{ex:ALE_r=1}.
\end{enumerate} 
\end{rem}
\end{samepage}

\begin{ex}[An almost log exceptional curve]\label{ex:ALE_r=1}
Consider a smooth projective surface $X$ with a rational chain $[n,1,m]$ on it with $n\geq m\geq 2$. Let $A$ be the middle $(-1)$-curve and let $D$ be the sum of the other two components. The peeling morphism $\alpha\:(X,D)\to (\ov X,0)$ is simply the contraction of $D$. We see that $D+A$ is contractible if and only if $n\geq 3$.  The contraction of $A+D$ in each case gives a log terminal singularity and the geometry is as in Lemma \ref{lem:case_r=1}(2b). But for $A$ to be almost log exceptional additional conditions have to be met. Indeed, we have $$\alpha(A)\cdot K_{\ov X}=A\cdot (K_X+D-\Bk D)=1-A\cdot \Bk D=1-\frac{2}{n}-\frac{2}{m},$$ so $A$ is almost log exceptional of the first kind if and only if $m=2$ and $n\geq 3$ or $m=3$ and $n\in\{3,4,5\}$ and it is almost log exceptional of the second kind if and only if $(m,n)\in \{(3,6),(4,4)\}$. 
\end{ex}

\medskip
\section{Almost minimalization for uniform boundaries}\label{ssec:uniform_boundaries}

As follows from computations in Example \ref{ex:peeling}, to develop a reasonable general description of peeling morphisms and almost minimal models for a log surface for $(X,D)$ it is better to add restrictions on the coefficients of $D$. We now study general \emph{uniform} boundaries, that is, boundaries equal to $rD$, where $D$ is a reduced divisor and $r\in\Q\cap [0,1]$. We call $r$ simply the \emph{coefficient} of the boundary. Coefficient $r=0$ means that no boundary is considered, so there is no peeling and almost minimal is the same as minimal. Coefficient $r=1$ was discussed in the previous section in case $X$ is smooth. Below we assume that $r\in\Q\cap (0,1]$. 

\medskip
\subsection{Varying the boundary coefficient}

\begin{lem}[Pure peeling varying with $r$]\  \label{lem:peeling_for_squeezed} 
Assume that $r<1$ and $\alpha\:(X,rD)\to (\ov X,r\ov D)$, where $D$ is reduced and $K_X$ is $\Q$-Cartier, is a pure partial peeling of the first (second) kind. Let $r'>r$. Then the following hold.
\begin{enumerate}[(1)]
\item $\alpha\:(X,r'D)\to (\ov X,r'\ov D)$ is a pure partial peeling of the first (second) kind. 
\item If $D$ is connected and $\cf(\cE(\alpha);\ov X,r\ov D)\leq r$ then $\cf(\cE(\alpha);\ov X,r'\ov D)\leq r'$ and the inequality is strict, unless $\alpha^*K_{\ov X}=K_X$ and $\alpha^*\ov D=D$. 
\item If $X$ is smooth then every connected component of $\Exc \alpha$ is a twig, rod or fork of $D$ or it is a $(-2)$-segment of $D$ (see Lemma \ref{lem:lc_singularities}(5)).
\end{enumerate}
\end{lem}

\begin{proof}Since $K_X$ and $K_X+rD$ is $\Q$-Cartier, $K_X+r'D$ is $\Q$-Cartier. Since $\alpha$ is a peeling of the first or second kind, curves contracted by $\alpha$ are $\Q$-Cartier.

(1) Since $\alpha$ is a partial MMP run of the second kind, we have $K+rD\geq \alpha^*(K_{\ov X}+r\ov D)$ with the difference supported on $\Exc \alpha$. We write this as $$(\tfrac{r'}{r}-1)K_X+K_X+r'D\geq (\tfrac{r'}{r}-1)\alpha^*K_{\ov X}+\alpha^*(K_{\ov X}+r'\ov D),$$ hence $$K_X+r'D\geq (\tfrac{r'}{r}-1)(\alpha^*K_{\ov X}-K_X)+\alpha^*(K_{\ov X}+r'\ov D).$$ Since $\alpha$ is pure, by Corollary \ref{cor:pure_peeling} we have $\alpha^*K_{\ov X}-K_X\geq 0$, so $K_X+r'D\geq \alpha^*(K_{\ov X}+r'\ov D)$. Hence  (1) follows form Corollary \ref{cor:improving_ld_gives_MMP}.

(2) Assume that $\cf(E;\ov X,r'\ov D)=r'$ for some component $E$ of $\Exc(\alpha)$. Then the above inequality shows that $E$ is not contained in the support of $\alpha^*K_{\ov X}- K_X$. Since $D$ is connected, Corollary \ref{cor:pure_peeling} gives $\alpha^*K_{\ov X}=K_X$. Lemma \ref{lem:ld_increasaes} gives $\cf_X(\ov X,r' \ov D)=(1-r')\Exc \alpha$, hence $K_X+r'D=\alpha^*(K_{\ov X}+r'\ov D)$ and we obtain $\alpha^*\ov D=D$.

(3) Assume that $X$ is smooth and $\Exc \alpha$ is connected. Applying part (1) with $r'=1$ we see that the morphism $\alpha$ is a pure partial peeling of $(X,D)$ of the second kind. If $\cf(\cE(\alpha);\ov X,\ov D)<1$ then it is of the first kind and we are done by Lemma \ref{lem:lc_singularities}. We may thus assume that $\cf(\cE(\alpha);\ov X,\ov D)=1$. By (2) $\alpha^*K_{\ov X}=K_X$ and $\alpha^*\ov D=D$. Then $\Exc \alpha$ consists of $(-2)$-curves with $\beta_D=2$. Since $\Exc \alpha$ is negative definite, it is a segment of $D$. 
\end{proof}

\begin{ex}\label{ex:peeling_tip}
Let $(X,D)$ be a log smooth surface with reduced boundary and let $T=[k]$ for some $k\geq 2$ be a tip of $D$ with $\beta_D(T)=1$. By adjunction we have $T\cdot (K_X+rD)=k-2+r(-k+1)$, so the contraction of $T$ is a partial peeling of the first (second) kind of $(X,rD)$ if and only if $r>1-\frac{1}{k-1}$ (respectively, $\geq $). For instance, if $r\leq \frac{1}{2}$ then the contraction of $T$ is a partial peeling of the first kind if and only if $T=[2]$.
\end{ex}

A peeling of the first (second) kind can be decomposed as a squeezing of the first (second) kind followed by a pure peeling, see Corollary \ref{cor:nef_gives_peeling}(1). The part consisting of squeezing will be studied in the next section. As for the part consisting of a pure peeling the following corollary together with Lemma \ref{lem:Bk=ld(pi)} give an explicit description of the coefficient divisor.

\begin{cor}[Uniqueness of peeling for squeezed surfaces]\label{cor:peeling_unique_for_squeezed} Assume that $X$ is a smooth surface, $D$ is reduced and $(X,rD)$ is squeezed. Then $(X,rD)$ has a unique (pure) peeling.
\end{cor}

\begin{proof} Suppose that $(X,rD)$ has more than one peeling morphism. There are connected components $E_1$ and $E_2$ of exceptional divisors of these peelings such that $E_1\nleq E_2$, $E_2\nleq E_1$ and $E_1$ meets $E_2$. By Corollary \ref{cor:reordering_MMP} we see that $E_1$, $E_2$ are not connected components of $D$. By Lemma \ref{lem:peeling_for_squeezed} $E_1$ and $E_2$ are admissible twigs of $D$, hence $(E_1+E_2)\redd$ is an admissible rod of $D$. 

To reach a contradiction we argue that there exists a peeling contracting $(E_1+E_2)\redd$. Let $C$ be some common component of $E_1$ and $E_2$. Write the rod as $T_1+C+T_2$, where $T_1$ and $T_2$ are disjoint admissible twigs meeting $C$, contained in $E_1$ and $E_2$, respectively. We may assume that $D=T_1+C+T_2$, $E_1=C+T_1$ and $E_2=T_2+C$. Let $p_1\:X\to Y_1$ be the contraction of $T_1$ and $p_2\:Y_1\to Y_2$ the contraction of $p_1(T_2)$. Put $C_1=p_1(C)$ and $C_2=p_2(C_1)$. It is sufficient to show that $C_2$ is log exceptional on $(Y_2,rC_2)$. Since $p_2$ is a partial MMP run, we have $p_2^*(K_{Y_2}+rC_2)=K_{Y_1}+r(C_1+p_1(T_2))-A$ for some $A\geq 0$ with $(p_2)_*A=0$. Since $C_1$ is log exceptional, we get $C_2\cdot (K_{Y_2}+rC_2)\leq C_1\cdot (K_{Y_1}+r(C_1+p_1(T_2)))<0$. Since $T_1+C+T_2$ is an admissible rod, it is negative definite, which gives $C_2^2<0$.
\end{proof}

\medskip
\subsection{Formulas for the coefficient divisor}

\begin{lem}[Coefficients for a uniform boundary] \label{lem:Bk=ld(pi)} 
Assume that $X$ is smooth surface and $D$ is a reduced divisor. Let $\alpha\:(X,D)\to (\ov X,\ov D)$ be a contraction of (some) admissible twigs, rods and forks of $D$. Denote by $T$ the sum of connected components of $\Exc \alpha$ which are twigs but not rods of $D$. Then (see Definition \ref{dfn:Bark})
\begin{equation}\label{eq:Bk_rD}
\cf_X(\ov X, r\ov D)=\Exc \alpha-\Bk_D(\Exc \alpha)-(1-r)\Bk\trp T.
\end{equation}
Assume that $\alpha$ is a pure partial peeling of $(X,rD)$. Then the coefficients of the divisor $\cf_X(\ov X, r\ov D)$ belong to $[0,r)$. Moreover, $\cf(U;\ov X, r\ov D)=0$ for a component $U$ or $\Exc \alpha$ if and only if the connected component of $\Exc \alpha$ containing $U$ consists of $(-2)$-curves and is an admissible rod or an admissible fork of $D$.
\end{lem}

\begin{proof} Assume first that $r=1$. The coefficient divisor is uniquely determined by the equations $U\cdot (K+D-E+\cf_X(\ov X,\ov D))=0$ where $U$ runs through components of $E:=\Exc \alpha$. We may assume that $E$ is connected. By \eqref{eq:Bk(chain)} in cases (1) and (2) of Definition \ref{dfn:Bark} the equations hold for $\Bk' E$ and $\Bk' E+\Bk\trp E$, respectively. In case (3) we see that the divisors $E_0+\sum_{i=1}^3 \Bk\trp T_i$ and $\sum_{i=1}^3\Bk' T_i$ intersect trivially with all components of $T_1+T_2+T_3$. Then \eqref{eq:Bk_rD} follows from the equation $E_0\cdot \cf_X(\ov X,\ov D)=-E_0\cdot (K+D-E)=0$. For $r\neq 1$ we need to show that $$K_X+r\alpha_*^{-1}\ov D+E=\alpha^*(K_{\ov X}+r\ov D)+\Bk_D(E)+(1-r)\Bk\trp T.$$ Subtracting the equality for $r=1$ and dividing by $1-r$ we see that it is sufficient to prove that 
\begin{equation}\label{eq:D_pullback}
\alpha^*\ov D=\alpha_*^{-1}\ov D+\Bk\trp T,
\end{equation}
and hence that $\alpha_*^{-1}\ov D+\Bk\trp T$ intersects trivially with every component $U$ of $T$. To see this we may assume that $T$ is a single admissible twig with components $T\cp{i}$, $i=1,\ldots,m$ and $U=T\cp{i}$ for some $i$. Then $U\cdot(\alpha_*^{-1}\ov D+\Bk\trp T)=\delta_i^m+U\cdot \Bk\trp T=0$ by \eqref{eq:Bk(chain)}.

Assume that $\alpha$ is a pure partial peeling for $(X,rD)$. Then $\cf(\ov X, r\ov D)<r$ by Corollary \ref{lem:eps-lc_is_respected}. Put $G=E-\Bk_D(E)+(1-r)\Bk\trp T$. For each component $E_j$ of $E$ we have $$0=E_j\cdot \alpha^*(K_{\ov X}+r\ov D)=E_j\cdot (K_X+r\alpha_*^{-1}\ov D+G)\geq E_j\cdot G,$$ because $\alpha$ is pure. By Lemma \ref{lem:negativity} we see that $G=0$ or $G\geq 0$ and $\Supp G=\Supp E$. The case $G=0$ happens when $E_j\cdot K_X=E_j\cdot \alpha_*^{-1}\ov D=0$ for each $j$, that is, when $G$ is an admissible rod or fork of $D$ consisting of $(-2)$-curves only.
\end{proof}

\begin{cor}[Characterization of almost log exceptional curves]\label{cor:almost_log_exceptional} Assume that $X$ is a smooth surface, $D$ is a reduced divisor and $\alpha\:(X,rD)\to (\ov X,r\ov D)$ is a pure partial peeling. Then $A\nleq D$ is $\alpha$-almost log exceptional of the first (second) kind if and only if $A$ is a $(-1)$-curve such that the following hold:
\begin{enumerate}[(1)]
\item $A+\Exc \alpha$ is negative definite,
\item $A\cdot (r\alpha_*^{-1}\ov D+\cf_X(\ov X,r\ov D))<1$ ($=1$).
\end{enumerate}
\end{cor}

\begin{proof}
Put $\ov A=\alpha(A)$. Clearly, (1) is equivalent to $\ov A^2<0$. We may assume that (1) holds. We have $\ov A\cdot (K_{\ov X}+r\ov D)=A\cdot (K_X+\tilde D)$, where $\tilde D=r\alpha_*^{-1}\ov D+\cf_X(\ov X,r\ov D)$, so we only need to show that if $A$ is almost log exceptional of the first or second kind then it is a $(-1)$-curve. By Lemma \ref{lem:a.l.e._is_redundant}(1) $A\cdot K_X<0$. Since $X$ is smooth and $A^2<0$, $A$ is a $(-1)$-curve. 
\end{proof}

\medskip
\subsection{Redundant and almost log exceptional curves for uniform boundaries}\label{ssec:Redundant_and_ALE_for_uniform}

The following characterization is used in the proof of Theorem \ref{thm:aMM_respects_(1-r)-dlt}.

\begin{lem}[Redundant curves]\label{lem:redundant_curve_numerics}
Let $(X,rD)$ be a log surface with $X$ smooth, $D$ reduced, $r\in [0,1]\cap \Q$. Let $\alpha\:(X,rD)\to (\ov X,r\ov D)$ be a contraction some admissible twigs $T_1, \ldots, T_k$ of $D$ meeting some $(-1)$-curve $\ll\leq D$. Put $\delta=\sum_{i=1}^k\frac{1}{d(T_i)}$, $\ind\trp=\sum_{i=1}^k \ind(T_i\trp)$ and $E=\sum_{i=1}^kT_i$. Then $\alpha(\ll)$ is log exceptional of the first or second kind if and only if the following inequalities hold 
\begin{equation}\label{eq:Ups-curve_main2}
\ind\trp<1,
\end{equation}
\begin{equation}\label{eq:Ups-curve_main1}
r\beta_{D-E}(\ll)\leq 2-k+\delta-(1-r)(1-\ind\trp),
\end{equation}
where the equality in \eqref{eq:Ups-curve_main1} holds exactly when $\alpha(\ll)$ is of the second kind.
\end{lem}

\begin{proof}
Put $\ov \ll=\alpha_*\ll$, $R=\alpha_*^{-1}\ov D-\ll$, $\delta_i=\delta(T_i)=1/d(T_i)$ and $\ind\trp_i=\ind(T_i\trp)$. By \eqref{eq:Bk(chain)} we have $\alpha^*\ov \ll=\ll+\sum_i \Bk\trp T_i$. We compute $\ov \ll^2=\ll\cdot (\ll+\sum_i \Bk\trp T_i)=-1+\ind\trp,$ and $$\ov \ll\cdot (K_{\ov X}+r\ov D)=\ll\cdot \alpha^*(K_{\ov X}+\ov D)+(r-1)\ll\cdot \alpha^*\ov D=\ll\cdot (K_X+R+\ll+\sum_{i}(T_i-\Bk T_i))+$$ $$+(r-1)\ll\cdot (R+\ll+\sum_i \Bk\trp T_i)=-2+\ll\cdot R+\sum_i(1-\delta_i)+(r-1)(\ll\cdot R-1+\sum_i \ind\trp_i)=$$ $$=k-\delta-1+r(\ll\cdot R-1)+(r-1)\ind\trp.$$ It follows that $\ov \ll$ is log exceptional of the first or second kind if and only if $\ind\trp < 1$ and  $r\ll\cdot R\leq 2-k+\delta-(1-r)(1-\ind\trp)$, with the equality for the second kind only.
\end{proof}

Recall that a divisor on a smooth surface \emph{contracts to a smooth point} if its support is the support of the exceptional divisor of a birational contraction onto a smooth surface. Since a blowup does not change the discriminant, such a divisor is necessarily a rational tree whose discriminant is equal to $1$.

\begin{prop}[Description of redundant components]\label{prop:ful_description_redundant} 
Let $(X,rD)$ be a log surface with $X$ smooth, $D$ reduced, $r\in [0,1]\cap \Q$. Let $\alpha$ be a pure partial peeling of $(X,rD)$ of the second kind and $\ll\leq D$ a $(-1)$-curve such that $\alpha(\ll)$ is log exceptional of the first (second) kind. Denote by $E$ the sum of connected components of $\Exc \alpha$ meeting $\ll$ and put $R=D-\ll-E$. Then one of the following holds (see Figure \ref{Fig:RED_1}).
\begin{enumerate}[(1)]
\item $r=1$ and $\beta_D(\ll)\leq 2$, so $\ll$ is log exceptional of the first or second kind. If the intersection of $E$ and $D-E$ is not normal then $E$ is a degenerate segment of $D$, so $\#E\leq 2$. 
\item $r\neq 1$, $\ll$ is log exceptional of the first (second) kind and one of the following holds:
\begin{enumerate}[(a)]
\item $\ll+E$ is a rational chain which contracts to a smooth point. If $E\cdot R\neq 0$ then  $\alpha$ is log crepant, either $E$ is a $(-2)$-segment of $D$ or $r=\frac{1}{2}$, $\ll$ is of the second kind and $\ll+E=[2,1,3]$.
\item $\ll+E$ is a rod or a twig of $D$, and it does not contract to a smooth point. 
\end{enumerate}
\item $r=\frac{1}{2}$, $\beta_D(\ll)=3$, $E\cdot R=0$  and $\ll+E$ is one of $[3,1,3]$, $[1,(2)_{m-1},3]$, $m\geq 1$, $[3,1,2,3]$. In particular, all components of $\ll+E$ are log exceptional of the second kind.
\item  $r=\frac{2}{3}$, $\beta_D(\ll)=3$, $E\cdot R=0$ and $\ll+E$ is one of $[2,1,4]$, $[2,1,3,2]$. In particular, $\alpha$ is not of the first kind and $\alpha(\ll)$ is log exceptional of the second kind.
\item $r\neq 1$ and $E$ is a $(-2)$-twig such that $0\leq \ll\cdot R-\frac{1}{r}\leq \frac{1}{\#E+1}$.
\item $E$ is a sum of two twigs of $D$, $[2]$ and $[3]$, $\ll\cdot R=1$ and $\frac{1}{2}\leq r \leq \frac{4}{5}$.
\end{enumerate}
Moreover, $\ll\cdot E\leq 2$, and the divisor $\ll+E$ is snc, unless $r=1$ and $\ll+E$ is a degenerate cycle.

In cases (5) and (6) the inequalities on the right (respectively, left) become equalities if and only if $\alpha(\ll)$ (respectively, $\ll$) is log exceptional of the second kind. 
\end{prop}

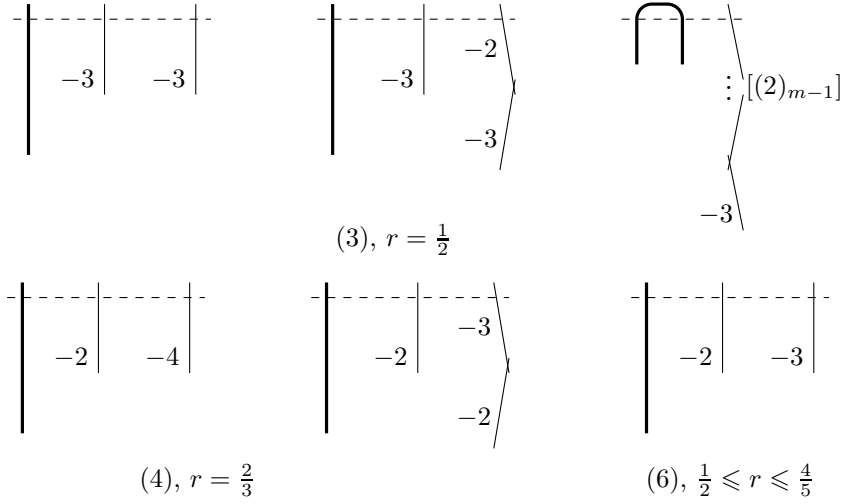
\begin{figure}[h]
	\subcaptionbox*{(3), $r=\frac{1}{2}$}[.8\textwidth]{
		\begin{tikzpicture}
		\path[use as bounding box] (0,-1.5) rectangle (10,1.2);	
		\begin{scope}	
			\draw[dashed] (0,1) -- (2.6,1);
			\draw[thick] (0.2,1.2) -- (0.2,-0.8);
			\draw (1.2,1.2) -- (1.2,0);
			\node[left] at (1.2,0.2) {\small{$-3$}};
			\draw (2.4,1.2) -- (2.4,0);
			\node[left] at (2.4,0.2) {\small{$-3$}};
		\end{scope}
		\begin{scope}[shift={(4,0)}]
			\draw[dashed] (0,1) -- (2.6,1);
			\draw[thick] (0.2,1.2) -- (0.2,-0.8);
			\draw (1.4,1.2) -- (1.4,0);
			\node[left] at (1.4,0.2) {\small{$-3$}};
			\draw (2.4,1.2) -- (2.6,0);
			\node[left] at (2.5,0.6) {\small{$-2$}};
			\draw (2.6,0.2) -- (2.4,-1);
			\node[left] at (2.5,-0.6) {\small{$-3$}};
		\end{scope}
		\begin{scope}[shift={(8,-2)}]
		\draw[dashed] (0,3) -- (1.6,3);
		\draw[thick] (0.2,2.4) -- (0.2,3) to[out=90,in=180] (0.4,3.2) -- (0.6,3.2) to[out=0,in=90] (0.8,3) -- (0.8,2.4);
		\draw(1.4,3.2) -- (1.6,2.2);
		\node[left] at (1.6,2.2) {$\vdots$};
		\node[right] at (1.5,2.1) {\small{$[(2)_{m-1}]$}};
		\draw (1.6,2) -- (1.4,1);
		\draw (1.4,1.2) -- (1.6,0.2);
		\node[left] at (1.6,0.4) {\small{$-3$}};
		\end{scope}
	\end{tikzpicture}
}
\bigskip

\subcaptionbox*{(4), $r=\frac{2}{3}$}[.5\textwidth]{
	\begin{tikzpicture}
		\path[use as bounding box] (0,-1) rectangle (5,1.2);
	\begin{scope}
		\draw[dashed] (0,1) -- (2.6,1);
		\draw[thick] (0.2,1.2) -- (0.2,-0.8);
		\draw (1.2,1.2) -- (1.2,0);
		\node[left] at (1.2,0.2) {\small{$-2$}};
		\draw (2.4,1.2) -- (2.4,0);
		\node[left] at (2.4,0.2) {\small{$-4$}};
	\end{scope}
	\begin{scope}[shift={(4,0)}]
		\draw[dashed] (0,1) -- (2.6,1);
		\draw[thick] (0.2,1.2) -- (0.2,-0.8);
		\draw (1.4,1.2) -- (1.4,0);
		\node[left] at (1.4,0.2) {\small{$-2$}};
		\draw (2.4,1.2) -- (2.6,0);
		\node[left] at (2.5,0.6) {\small{$-3$}};
		\draw (2.6,0.2) -- (2.4,-1);
		\node[left] at (2.5,-0.6) {\small{$-2$}};
	\end{scope}
	\end{tikzpicture}
}	
\subcaptionbox*{(6), $\frac{1}{2}\leq r\leq \frac{4}{5}$}[.3\textwidth]{
	\begin{tikzpicture}
		\path[use as bounding box] (0,-1) rectangle (2.6,1.2);
		\draw[dashed] (0,1) -- (2.6,1);
		\draw[thick] (0.2,1.2) -- (0.2,-0.8);
		\draw (1.2,1.2) -- (1.2,0);
		\node[left] at (1.2,0.2) {\small{$-2$}};
		\draw (2.4,1.2) -- (2.4,0);
		\node[left] at (2.4,0.2) {\small{$-3$}};
	\end{tikzpicture}
}
\caption{Proposition \ref{prop:ful_description_redundant}, cases (3), (4), (6). Thick line indicates $R=D-E-\ll$, dashed line is $\ll$. }\label{Fig:RED_1}
\end{figure}

\begin{proof}
If $E=0$ then $\ll$ is log exceptional of the first (second) kind, so we have $r\beta_D(\ll)\leq r+1$, which is a subcase of (1) or (2a). We may therefore assume that $E\neq 0$, hence $r\neq 0$. To avoid confusion below we refer to the cases of Lemma \ref{lem:lc_singularities} as (L1)-(L5). By Lemma \ref{lem:peeling_for_squeezed} $\alpha$ is a pure partial peeling of the second kind of $D$, and $E$ is as in (L2), (L4) or (L5). Moreover, since $E+\ll$ is negative definite, in case (L4) $E$ does not consist of $(-2)$-curves and in case (L5) it is not a degenerate $(-2)$-segment, so $r=1$. 

Consider the case $r=1$. By Lemma \ref{lem:reordering_1st_2nd} we have $\alpha=\alpha_2\circ\alpha_1$, where $\alpha_1$ is a peeling of the first kind and $\alpha_2$ is log crepant. It follows that $\ll$ is $\alpha_1$-redundant, so $\beta_{D}(\ll)\leq 2$ by Lemma \ref{lem:peel_sq_red_bdry}. If the intersection of $E$ and $D-E$ is not normal then $E$ is a degenerate segment of $D$, which gives (1).

We may further assume that $r<1$. Denote the contraction of $\alpha(\ll)$ by $\sigma\:\alpha(X)\to \ov X$. Put $\ov D=\sigma_*\alpha_*D$ and $q=\sigma(\alpha(\ll))\in \ov X$.

Assume that $q\in \ov X$ is singular. Let $\pi\:\tilde X\to \ov X$ be the minimal resolution of singularities. Since $X$ is smooth, we get a factorization $\sigma\circ \alpha=\pi\circ \varphi$. Put $\tilde D=\varphi_*D$ and $\tilde E=\varphi_*E$. 
\begin{equation}\label{eq:diagram-red_components}
\begin{gathered}
\xymatrix{
	(X,rD) \ar[d]_{\alpha} \ar[r]^{\phi} & (\tilde X,r\tilde D) \ar[d]^{\pi}\\ 
	(\alpha(X),r\alpha_*D) \ar[r]^-{\sigma} & (\ov X,r\ov D)
}
\end{gathered}
\end{equation}
Since $\alpha$ is a partial peeling of $rD$ of the second kind, Corollary \ref{cor:reordering_MMP}(2) implies that $\pi$ is a partial peeling of $r\tilde D$ of the second kind. Since $\pi$ is minimal, it is pure, hence $\tilde E$ is as in Lemma \ref{lem:lc_singularities}. By Lemma \ref{lem:peeling_for_squeezed} $\tilde E$ is a sum of admissible twigs, rods, forks and $(-2)$-segments of $\tilde D$. If $\beta_D(\ll)\leq 2$ then, since $E\neq 0$, $\ll$ is superfluous in $D$ (hence log exceptional of the first kind on $(X,rD)$), which gives (2b). We are left with case  $\beta_D(\ll)\geq 3$. Then $\varphi(\ll)$ is not a point of simple normal crossings of $\tilde D$, so $\tilde E$ is a degenerate $(-2)$-segment of $\tilde D$.  If $G$ is a component of $E+\ll$ not contracted by $\varphi$ then, since $\pi^*K_{\ov X}=K_{\tilde X}$ and $\pi^*\ov D=\tilde D$, we have $\cf(G;\ov X, r\ov D)=\cf(G;\tilde X, r\tilde D)=r=\cf(G;X,rD)$. By Lemma \ref{lem:ld_increasaes} it follows in particular that $\alpha(\ll)$ is of the second kind. The fact that $\ll$ is not superfluous in $D$ restricts possible types of blowing-ups constituting $\varphi^{-1}$. For $\#\tilde E=2$ we obtain $\ll+E=[3,1,3]$, which gives $\cf(\ll;\tilde X,r\tilde D)=3r-1$, so $r=\frac{1}{2}$ and hence case (3) holds. Assume that $\#\tilde E=1$. Then there exists $m\geq 1$ such that $\ll+E=[1,(2)_{m-1},3]$ or $\ll+E=[2,1,3,(2)_{m-2},3]$, where $[3,(2)_{-1},3]:=[4]$. We obtain $\cf(\ll;\tilde X,r\tilde D)=(2m+1)r-m$ and $\cf(\ll;\tilde X,r\tilde D)=2m(2r-1)$, hence $r=\frac{1}{2}$ and $r=\frac{2m}{4m-1}$, respectively. In the second case the coefficient of the component of the twig $T_1=[3,(2)_{m-2},3]$ meeting $\ll$ equals $u=\frac{2m+r(2m-1)}{4m}$, cf.\ formula \eqref{eq:Bk_rD}. But since the contraction of $T_1$ is a peeling of the second kind, we have $u\leq r$, hence $r\geq \frac{2m}{2m+1}$, which gives $m=1$ and $r=\frac{2}{3}$. This gives part of (3) and (4).

We may therefore assume that $q\in \ov X$ is smooth. By Lemma \ref{lem:lc_singularities} and the above remarks each connected component of $E$ is a twig of $D$ or a $(-2)$-segment of $D$. By negative definiteness of $E+\ll$ in the second case the segment meets $R$ and there can be at most one such. Denote it by $E_0$. We work under the assumption that $E_0=0$ and we comment on the case $E_0\neq 0$ at the end of the proof.

Write $E=T_1+\ldots+T_k$, where $T_1,\ldots, T_k$ are admissible twigs of $D$ and $k\geq 1$. It follows that $\ll\cdot E=k$ and $E\cdot R=0$. Put $\delta=\sum_{i=1}^k\frac{1}{d(T_i)}$ and $\ind\trp=\sum_{i=1}^k \ind(T_i\trp)$. Since $\ll$ is superfluous in $E+\ll$, we have $k\in \{1,2\}$. Assume that $k=1$. Then $E=[(2)_m]$ for some $m\geq 1$. In particular, $\ind(E\trp)=1-\frac{1}{m+1}$. The inequality \eqref{eq:Ups-curve_main1} gives $\beta_{D-E}(\ll)\leq \frac{1}{r}+\frac{1}{m+1}$, where the equality holds if $\alpha(\ll)$ is log exceptional of the second kind. Note that $\ll$ is log exceptional of the first kind if and only if $\beta_D(\ll)<1+\frac{1}{r}$, that is, if $\beta_{D-E}(\ll)<\frac{1}{r}$. If the latter inequality holds we get (2a), otherwise we get (5).

We are left with the case $k=2$. Put $d_i=d(T_i)$. We may and shall assume that $d_1\geq d_2$. Since $q\in\ov X$ is smooth, we have $d(\ll+E)=1$. By Lemma \ref{lem:d(D_1+D_2)} we obtain $1=d(\ll+E)=d_1d_2(1-\ind\trp)$. Hence $\gcd(d_1,d_2)=1$ and $1-\ind \trp=\frac{1}{d_1d_2}>0$. The former implies that $d_1-1\geq d_2\geq 2$. By Lemma \ref{lem:redundant_curve_numerics} we obtain $$d_1d_2r\ell\cdot R\leq d_1+d_2-1+r.$$ We may assume that $\ll$ is not log exceptional of the first kind, as otherwise we get (2a). This gives $0\leq \ll\cdot (K+rD)=-1+r(\ll\cdot R+1)$. We obtain
\begin{equation}\label{eq:cor1}
\frac{1}{\ll\cdot R+1}\leq r \leq\frac{d_1+d_2-1}{d_1d_2\ll\cdot R-1}.
\end{equation}
It follows that $(d_1d_2\ll\cdot R-1)\leq (d_1+d_2-1)(\ll\cdot R+1)$, hence $\ll\cdot R(d_1-1)(d_2-1)\leq d_1+d_2$. Since $d_1-1\geq d_2$, we obtain $\ll\cdot R(d_2-1)\leq 2$. Suppose that $\ll\cdot R=2$. Then $d_2=2$ and $d_1\leq 4$, hence $d_1=3$. Negative definiteness of $\ll+E$ implies that $T_1=[3]$. The inequality \eqref{eq:cor1} gives $r\leq \frac{4}{11}$. But then $T_1\cdot (K+rD)=1-2r>0$, which contradicts the fact the $\alpha$ is a peeling of the second kind.

Thus $\ll\cdot R=1$ and the above inequality reads as 
\begin{equation}\label{eq:cor2}
\frac{1}{2}\leq r \leq\frac{d_1+d_2-1}{d_1d_2-1}.
\end{equation}
It follows that $(d_1-2)(d_2-2)\leq 3$, hence $d_2\leq 3$ and if $d_2=3$ then $d_1\in \{4,5\}$. Assume that $d_2=3$. Since $\ll+E$ contracts to a smooth point, we get $T_2=[3]$ and $T_1=[3,2]$ or $T_2=[2,2]$ and $T_1=[4]$. But in the second case we get $r\leq \frac{6}{11}$, so the contraction of $T_1$ is not a partial peeling of the second kind of $rD$. In the first case we have $r=\frac{1}{2}$, which gives the remaining part of (3).

Finally, assume that  $d_2=2$. Since $\ll+E$ contracts to a smooth point, we get $T_1=[(2)_{m-1},3]$ for some $m\geq 1$. The coefficient of the $(-3)$-curve equals $u=\frac{m(r+1)}{2m+1}$. Since the contraction of $T_1$ is a peeling of the second kind, we get $u\leq r$, hence $1-\frac{1}{m+1}\leq r$. Then \eqref{eq:cor2} gives $\frac{m}{m+1}\leq \frac{2m+2}{4m+1}$, so $m\in \{1,2\}$. For $m=2$ we get $r=\frac{2}{3}$, which gives the remaining part of (4). For $m=1$ we get $\frac{1}{2}\leq r\leq \frac{4}{5}$, which gives (6).

Finally, consider the case when $E_0\neq 0$, that is, $E$ has a (unique) connected component $E_0$ which is a segment of $D$. Write $\alpha=\alpha_0\circ \alpha'$, where $\alpha'$ is the contraction of $E':=E-E_0$. Since $\alpha_0$ is log crepant, $\alpha(\ll)$ is log exceptional of the first (second) kind if $\alpha'(\ll)$ is log exceptional of the first (second) kind. Put $R'=D-\ll-E'$. Then $R'=R+E_0$. We apply the lemma to $(X,rD)$ with $(\alpha',E',R')$ instead of $(\alpha,E,R)$ and we make conclusions for $(\alpha,E,R)$. Due to the negative definiteness of $\ll+E$, the divisor $\ll+E'$ is as in case (2a) with $E'=0$, (2b), (3) with $\ll+E'=[1,3]$ and $\#E_0=1$ or (5) with $E'=0$. In the first case $\ll+E$ is as in case (2a) with $\alpha$ being log crepant and $R\cdot E=1$. In the second case, since $E_0\cdot \ll\neq 0$, $\ll+E$ is itself a twig of $D$, hence is as in (2b) or (2a). In the third case $\ll+E=[2,1,3]$ and $R\cdot E=1$, which is part of case (2a). In the last case, namely, case (5) with $E'=0$ we have $\ll$ which is log exceptional of the first (second) kind, $\ll+E=[1,2,2,\ldots,2]$ and $\alpha$ is log crepant, which is case (2a).

In cases (2)-(6) the divisor $\ll+E$ is a rational chain. It follows that in all cases (1)-(6) $\ll\cdot E\leq 2$ and $E$ is snc. If $\ll+E$ is not snc then we are in case (1) with $D=E$ and $E$ is a degenerate segment of $\ll+E$, hence $\ll+E$ is a degenerate cycle. 
\end{proof}

\begin{prop}[Description of almost log exceptional curves]\label{prop:ful_description_almost_log_exc}  
Assume that $X$ is smooth, $D$ is reduced and $r\in [0,1]\cap \Q$. Let $\ll\nleq D$ be an $\alpha$-almost log exceptional curve of the first (second) kind for some pure partial peeling $\alpha$ of $(X,rD)$ of the first (second) kind. Let $E$ be the sum of connected components of $\Exc \alpha$ meeting $\ll$. Put $R=D-E$. Then one of the following holds (see Fig.\ \ref{Fig:ALE_1} and \ref{Fig:ALE_2}):
\begin{enumerate}[(1)]
\item $\ll$ is superfluous in $D+\ll$,
\item $\ll\cdot D=2$, $r=\frac{1}{2}$, $\ll+E=[1,(2)_{m-1},3]$, $m\geq 1$ and $R$ meets $E+\ll$ at the point $\ll\cap E$. In particular, all components of $\ll+E$ are log exceptional of the second kind.
\item $\ll\cdot D=3$, $r=\frac{m}{2m+1}$, $\ll+E=[1,(2)_{m-1},3]$ for some $m\geq 1$, and $R\cdot E=0$. In particular, $\alpha$ is not of the first kind and $\alpha(\ll)$ is of the second kind.
\item  $\ll\cdot D=3$, $E\cdot R=0$  and $\ll+E$ is one of $[3,1,3]$, $[2,1,4]$ or $[2,1,3,(2)_{m-2},3]$ for some $m\geq 2$. In the first case $r=\frac{1}{3}$ and in the remaining cases $r=\frac{1}{2}$. In particular, $\alpha$ is not of the first kind and $\alpha(\ll)$ is of the second kind.
\item $r<1$, $\ll$ is log exceptional of the first (second) kind and $\ll+E$ contracts to a smooth point.
\item $0<r<1$, $\ll+E=[1,(2)_{k-1}]$ for some $k\geq 1$ and $E$ is a twig or a rod of $D$. Moreover, $\frac{1}{r}\leq \ll\cdot D\leq \frac{1}{r}+s$, where at least one of the inequalities is strict, $s=1$ if $E$ is a rod, $s=1-\frac{1}{k}$ if $E$ is a twig and $\ll$ meets it in a tip of $D$ and $s=\frac{1}{k}$ otherwise.
\item $\ll\cdot D=4$, $E\cdot R=0$, $r=\frac{1}{3}$, $\ll+E=[2,1,3]$, $\alpha$ is not of the first kind and $\alpha(\ll)$ is of the second kind.
\item $\ll\cdot D=3$, $E\cdot R=0$, $\ll+E=[2,1,3,(2)_{k-1}]$ and $\frac{1}{3}\leq r\leq \frac{k+1}{2k+1}$ for some $k\geq 1$.  
\item $\ll\cdot D=3$, $E\cdot R=1$, $\frac{k}{2k+1}\leq r\leq \frac{2(k+1)}{3(2k+1)}$ and $\ll+E=[2,1,3,(2)_{k-1}]$ for some $k\in \{1,2\}$, where $R$ meets the rod $[2]$. 
\item $\ll\cdot D=3$, $E\cdot R=1$, $r=\frac{1}{2}$ and $\ll+E=[2,1,3,(2)_{k-1}]$ and  for some $k\geq 1$, where $R$ meets the last component of $[3,(2)_{k-1}]$. Moreover, $\alpha(\ll)$ is of the second kind. 
\item $\ll\cdot D=3$, $E\cdot R=0$, $r=\frac{1}{2}$ and $\ll+E=[3,1,2,3,2]$ or $\ll+E=[4,1,2,2]$. Moreover, $\alpha(\ll)$ is of the second kind. 
\item $\ll\cdot D=3$, $E\cdot R=0$, $\frac{2}{5}\leq r\leq \frac{7}{15}$ and $\ll+E=[3,1,2,3]$. 
\end{enumerate}
In cases (6), (8), (10) and (12) the inequalities on the right (respectively, left) become equalities if and only if $\alpha(\ll)$ (respectively, $\ll$) is log exceptional of the second kind. In case (9) the inequality on right becomes an equality if and only if $\alpha(\ll)$ is log exceptional of the second kind.
\end{prop}

\begin{figure}[h]
	\begin{minipage}{.5\linewidth}
\subcaptionbox*{(2), $r=\frac{1}{2}$}[0.48\textwidth]{
	\begin{tikzpicture}
		\path[use as bounding box] (0,-0.4) rectangle (2,3.2);
		\draw[thick, name path = R] (0,3) -- (1.6,3);
		\draw[name path = T] (1.4,3.2) -- (1.6,2.2);
		\path [name intersections={of=R and T,by=E}];
		\node[left] at (1.6,2.2) {$\vdots$};
		\node[right] at (1.5,2.1) {\small{$[(2)_{m-1}]$}};
		\draw (1.6,2) -- (1.4,1);
		\draw (1.4,1.2) -- (1.6,0.2);
		\node[left] at (1.6,0.4) {\small{$-3$}};
		\draw[dashed, add = 0.1 and 0] (E) to ($(E)-(1.2,1.2)$);
	\end{tikzpicture}
}
\subcaptionbox*{(3), $r=\frac{m}{2m+1}$}[0.46\textwidth]{
	\begin{tikzpicture}
		\path[use as bounding box] (0,-0.4) rectangle (2,3.2);
		\draw[dashed] (0,3) -- (1.6,3);
		\draw[thick] (0.2,2.4) -- (0.2,3) to[out=90,in=180] (0.4,3.2) -- (0.6,3.2) to[out=0,in=90] (0.8,3) -- (0.8,2.4);
		\draw(1.4,3.2) -- (1.6,2.2);
		\node[left] at (1.6,2.2) {$\vdots$};
		\node[right] at (1.5,2.1) {\small{$[(2)_{m-1}]$}};
		\draw (1.6,2) -- (1.4,1);
		\draw (1.4,1.2) -- (1.6,0.2);
		\node[left] at (1.6,0.4) {\small{$-3$}};
	\end{tikzpicture}
}
\end{minipage}
\begin{minipage}{.2\linewidth}
\subcaptionbox*{(4), $r=\frac{1}{3}$}[\textwidth]{
	\begin{tikzpicture}
		\path[use as bounding box] (0,0) rectangle (2.6,1.2);
		\draw[dashed] (0,1) -- (2.6,1);
		\draw[thick] (0.2,1.2) -- (0.2,0);
		\draw (1.2,1.2) -- (1.2,0);
		\node[left] at (1.2,0.2) {\small{$-3$}};
		\draw (2.4,1.2) -- (2.4,0);
		\node[left] at (2.4,0.2) {\small{$-3$}};
	\end{tikzpicture}
}
\bigskip

\subcaptionbox*{(4), $r=\frac{1}{2}$}[\textwidth]{
	\begin{tikzpicture}
		\path[use as bounding box] (0,0) rectangle (2.6,1.2);
		\draw[dashed] (0,1) -- (2.6,1);
		\draw[thick] (0.2,1.2) -- (0.2,0);
		\draw (1.2,1.2) -- (1.2,0);
		\node[left] at (1.2,0.2) {\small{$-2$}};
		\draw (2.4,1.2) -- (2.4,0);
		\node[left] at (2.4,0.2) {\small{$-4$}};
	\end{tikzpicture}
}
\end{minipage}
\begin{minipage}{.24\linewidth}
\subcaptionbox*{(4), $r=\frac{1}{2}$}[\textwidth]{
	\begin{tikzpicture}
		\path[use as bounding box] (0,-2.4) rectangle (2.6,1.2);
		\draw[dashed] (0,1) -- (2.6,1);
		\draw[thick] (0.2,1.2) -- (0.2,-1);
		\draw (1.2,1.2) -- (1.2,0);
		\node[left] at (1.2,0.2) {\small{$-2$}};
		\draw (2.4,1.2) -- (2.6,0.2);
		\node[left] at (2.5,0.6) {\small{$-3$}};
		\draw (2.6,0.4) -- (2.4,-0.6);
		\node[left] at (2.7,-0.6) {$\vdots$};
		\node[right] at (2.5,-0.7) {\small{$[(2)_{m-1}]$}};
		\draw (2.4,-0.8) -- (2.6,-1.8);
		\draw (2.6,-1.6) -- (2.4,-2.6);
		\node[left] at (2.5,-2.2) {\small{$-3$}};
	\end{tikzpicture}
}
\end{minipage}
\bigskip

\begin{minipage}{.5\linewidth}
\subcaptionbox*{(9), $k=1$, $\frac{1}{3}\leq r\leq \frac{4}{9}$, \\ or $k=2$, $r=\frac{2}{5}$}[0.48\textwidth]{
	\begin{tikzpicture}
		\path[use as bounding box] (0,-1.8) rectangle (2.6,1.2);
		\draw[dashed] (0,1) -- (2.6,1);
		\draw[thick] (0.2,1.2) -- (0.2,0.8) to[out=-90,in=180] (0.4,0.6) -- (1.4,0.6);
		\draw (1.2,1.2) -- (1.2,0);
		\node[left] at (1.2,0.2) {\small{$-2$}};
		\draw (2.4,1.2) -- (2.6,0.2);
		\node[left] at (2.5,0.6) {\small{$-3$}};
		\draw (2.6,0.4) -- (2.4,-0.6);
		\node[left] at (2.7,-0.6) {$\vdots$};
		\node[right] at (2.5,-0.7) {\small{$[(2)_{k-1}]$}};
		\draw (2.4,-0.8) -- (2.6,-1.8);
	\end{tikzpicture}
}
\subcaptionbox*{(10), $r=\frac{1}{2}$}[0.48\textwidth]{
	\begin{tikzpicture}
		\path[use as bounding box] (0,-1.8) rectangle (2.6,1.2);
		\draw[dashed] (0,1) -- (2.6,1);
		\draw[thick] (0.2,1.2) -- (0.2,-1.4) to[out=-90,in=180] (0.4,-1.6) -- (2.8,-1.6);
		\draw (1.2,1.2) -- (1.2,0);
		\node[left] at (1.2,0.2) {\small{$-2$}};
		\draw (2.4,1.2) -- (2.6,0.2);
		\node[left] at (2.5,0.6) {\small{$-3$}};
		\draw (2.6,0.4) -- (2.4,-0.6);
		\node[left] at (2.7,-0.6) {$\vdots$};
		\node[right] at (2.5,-0.7) {\small{$[(2)_{k-1}]$}};
		\draw (2.4,-0.8) -- (2.6,-1.8);
	\end{tikzpicture}
}
\end{minipage}
\begin{minipage}{.24\linewidth}
\bigskip
\subcaptionbox*{(7), $r=\frac{1}{3}$}[\textwidth]{
	\begin{tikzpicture}
		\path[use as bounding box] (0,2) rectangle (2.8,3.2);
		\draw[dashed] (0,3) -- (2.8,3);
		\draw[thick] (0.2,2.4) -- (0.2,3) to[out=90,in=180] (0.4,3.2) -- (0.6,3.2) to[out=0,in=90] (0.8,3) -- (0.8,2.4);
		\draw(1.6,3.2) -- (1.6,2);
		\node[left] at (1.6,2.2) {\small{$-2$}};
		\draw(2.6,3.2) -- (2.6,2);
		\node[left] at (2.6,2.2) {\small{$-3$}};
	\end{tikzpicture}
}
\bigskip

\subcaptionbox*{(11), $r=\frac{1}{2}$}[\textwidth]{
	\begin{tikzpicture}
		\path[use as bounding box] (0,-0.8) rectangle (2.6,1.2);
		\draw[dashed] (0,1) -- (2.6,1);
		\draw[thick] (0.2,1.2) -- (0.2,-0.8);
		\draw (1.4,1.2) -- (1.4,0);
		\node[left] at (1.4,0.2) {\small{$-4$}};
		\draw (2.4,1.2) -- (2.6,0);
		\node[left] at (2.5,0.6) {\small{$-2$}};
		\draw (2.6,0.2) -- (2.4,-1);
		\node[left] at (2.5,-0.6) {\small{$-2$}};
	\end{tikzpicture}
}	
\end{minipage}
\begin{minipage}{.2\linewidth}
\subcaptionbox*{(11), $r=\frac{1}{2}$}[\textwidth]{
	\begin{tikzpicture}
		\path[use as bounding box] (0,-2) rectangle (2.6,1.2);
		\draw[dashed] (0,1) -- (2.6,1);
		\draw[thick] (0.2,1.2) -- (0.2,-1);
		\draw (1.4,1.2) -- (1.4,0);
		\node[left] at (1.4,0.2) {\small{$-3$}};
		\draw (2.4,1.2) -- (2.6,0);
		\node[left] at (2.5,0.6) {\small{$-2$}};
		\draw (2.6,0.2) -- (2.4,-1);
		\node[left] at (2.5,-0.4) {\small{$-3$}};
		\draw (2.4,-0.8) -- (2.6,-2);
		\node[left] at (2.5,-1.6) {\small{$-2$}};
	\end{tikzpicture}
}
\end{minipage}
\caption{Proposition \ref{prop:ful_description_almost_log_exc}, cases (2)-(4), (7), (9)-(11); $\alpha(\ll)$ of the second kind. Thick line indicates $R=D-E$, dashed line is $\ll$. }\label{Fig:ALE_1}
\end{figure}
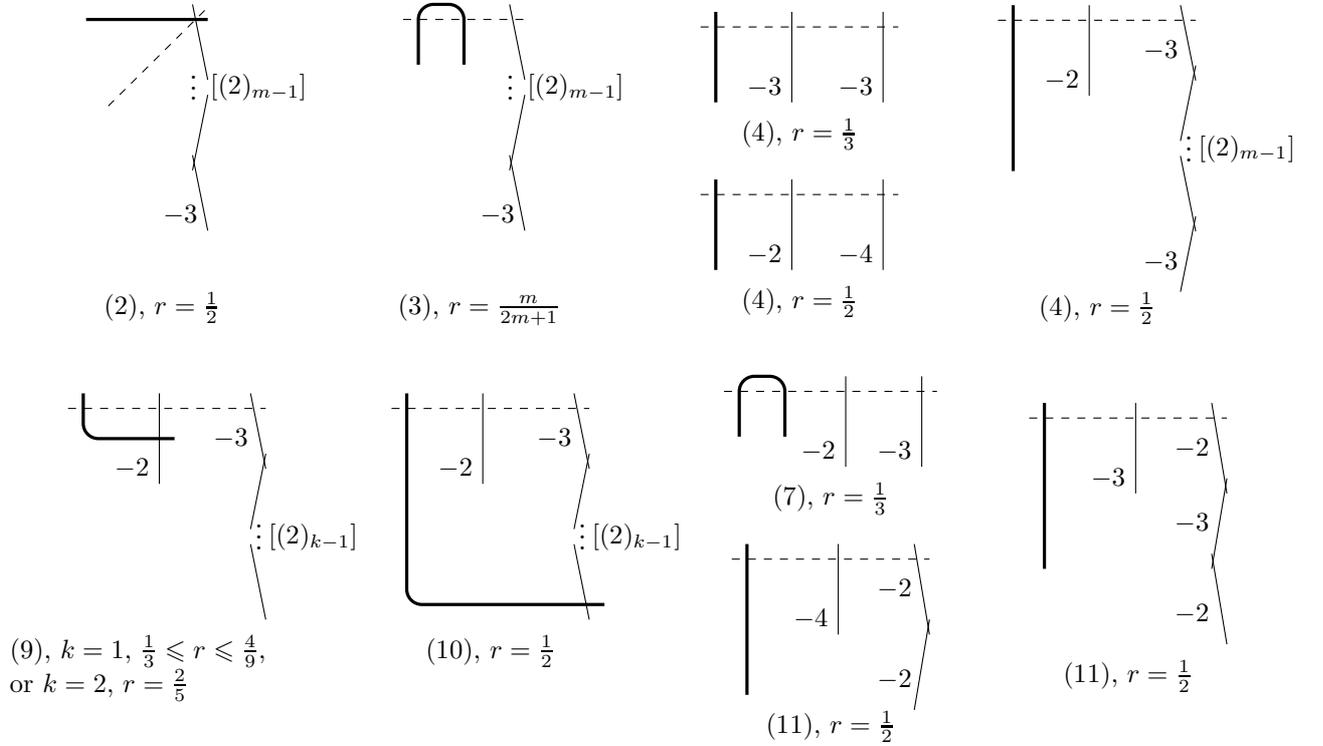

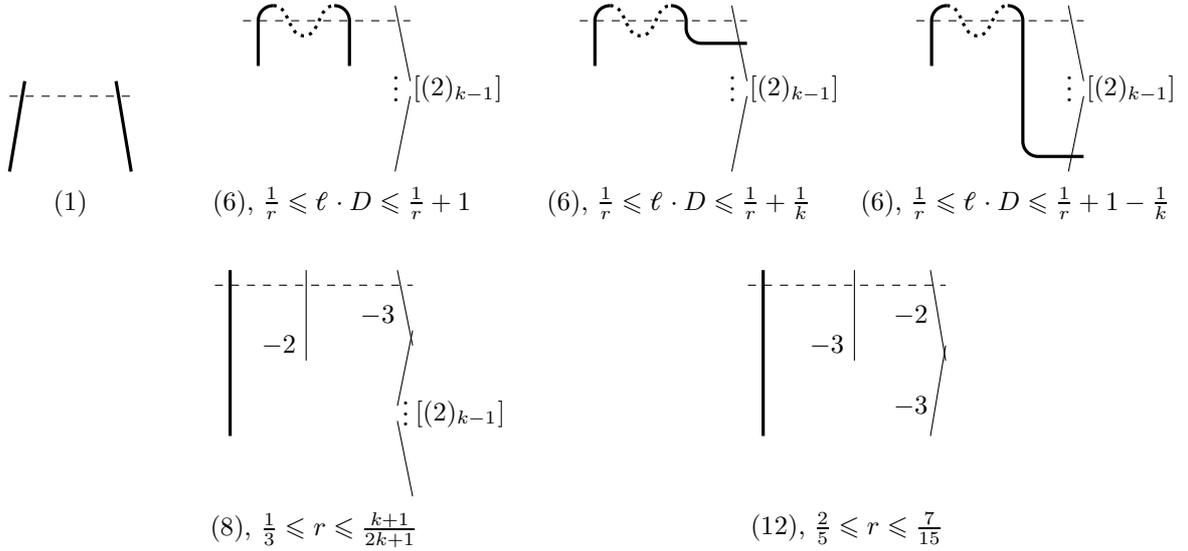
\begin{figure}
	\subcaptionbox*{(1)}[0.15\textwidth]{
		\begin{tikzpicture}
			\path[use as bounding box] (0,0) rectangle (1.6,1.2);
			\draw[thick] (0,0) -- (0.2,1.2);
			\draw[dashed] (0,1) -- (1.6,1);
			\draw[thick] (1.4,1.2) -- (1.6,0);
		\end{tikzpicture}
	}
\subcaptionbox*{(6), $\frac{1}{r}\leq \ll\cdot D \leq \frac{1}{r}+1$}[0.25\textwidth]{
	\begin{tikzpicture}
		\path[use as bounding box] (-0.6,1) rectangle (2,3.2);
		\draw[dashed] (-0.6,3) -- (1.6,3);
		\draw[thick] (-0.4,2.4) -- (-0.4,3) to[out=90,in=180] (-0.2,3.2);
		\draw[thick, dotted] (-0.2,3.2) to[out=0,in=180] (0.2,2.8) to[out=0,in=180] (0.6,3.2);
		\draw[thick] (0.6,3.2) to[out=0,in=90] (0.8,3) -- (0.8,2.4);
		\draw(1.4,3.2) -- (1.6,2.2);
		\node[left] at (1.6,2.2) {$\vdots$};
		\node[right] at (1.5,2.1) {\small{$[(2)_{k-1}]$}};
		\draw (1.6,2) -- (1.4,1);
	\end{tikzpicture}
}
\subcaptionbox*{(6), $\frac{1}{r}\leq \ll\cdot D \leq \frac{1}{r}+\frac{1}{k}$}[0.25\textwidth]{
	\begin{tikzpicture}
		\path[use as bounding box] (-0.6,1) rectangle (2,3.2);
		\draw[dashed] (-0.6,3) -- (1.6,3);
		\draw[thick] (-0.4,2.4) -- (-0.4,3) to[out=90,in=180] (-0.2,3.2);
		\draw[thick, dotted] (-0.2,3.2) to[out=0,in=180] (0.2,2.8) to[out=0,in=180] (0.6,3.2);
		\draw[thick] (0.6,3.2) to[out=0,in=90] (0.8,3) -- (0.8,2.9) to[out=-90,in=180] (1,2.7) -- (1.6,2.7);
		\draw(1.4,3.2) -- (1.6,2.2);
		\node[left] at (1.6,2.2) {$\vdots$};
		\node[right] at (1.5,2.1) {\small{$[(2)_{k-1}]$}};
		\draw (1.6,2) -- (1.4,1);
	\end{tikzpicture}
}
\subcaptionbox*{(6), $\frac{1}{r}\leq \ll\cdot D \leq \frac{1}{r}+1-\frac{1}{k}$}[0.25\textwidth]{
	\begin{tikzpicture}
		\path[use as bounding box] (-0.6,1) rectangle (2,3.2);
		\draw[dashed] (-0.6,3) -- (1.6,3);
		\draw[thick] (-0.4,2.4) -- (-0.4,3) to[out=90,in=180] (-0.2,3.2);
		\draw[thick, dotted] (-0.2,3.2) to[out=0,in=180] (0.2,2.8) to[out=0,in=180] (0.6,3.2);
		\draw[thick] (0.6,3.2) to[out=0,in=90] (0.8,3) -- (0.8,1.4) to[out=-90,in=180] (1,1.2) -- (1.6,1.2);
		\draw(1.4,3.2) -- (1.6,2.2);
		\node[left] at (1.6,2.2) {$\vdots$};
		\node[right] at (1.5,2.1) {\small{$[(2)_{k-1}]$}};
		\draw (1.6,2) -- (1.4,1);
	\end{tikzpicture}
}
\bigskip

\subcaptionbox*{(8), $\frac{1}{3}\leq r\leq \frac{k+1}{2k+1}$}[0.4\textwidth]{
	\bigskip
	\begin{tikzpicture}
		\path[use as bounding box] (0,-1.8) rectangle (2.6,1.2);
		\draw[dashed] (0,1) -- (2.6,1);
		\draw[thick] (0.2,1.2) -- (0.2,-1);
		\draw (1.2,1.2) -- (1.2,0);
		\node[left] at (1.2,0.2) {\small{$-2$}};
		\draw (2.4,1.2) -- (2.6,0.2);
		\node[left] at (2.5,0.6) {\small{$-3$}};
		\draw (2.6,0.4) -- (2.4,-0.6);
		\node[left] at (2.7,-0.6) {$\vdots$};
		\node[right] at (2.5,-0.7) {\small{$[(2)_{k-1}]$}};
		\draw (2.4,-0.8) -- (2.6,-1.8);
	\end{tikzpicture}
}
\subcaptionbox*{(12), $\frac{2}{5}\leq r \leq \frac{7}{15}$}[0.4\textwidth]{
	\begin{tikzpicture}
		\path[use as bounding box] (0,-1.8) rectangle (2.6,1.2);
		\draw[dashed] (0,1) -- (2.6,1);
		\draw[thick] (0.2,1.2) -- (0.2,-1);
		\draw (1.4,1.2) -- (1.4,0);
		\node[left] at (1.4,0.2) {\small{$-3$}};
		\draw (2.4,1.2) -- (2.6,0);
		\node[left] at (2.5,0.6) {\small{$-2$}};
		\draw (2.6,0.2) -- (2.4,-1);
		\node[left] at (2.5,-0.6) {\small{$-3$}};
	\end{tikzpicture}
}
\caption{Proposition \ref{prop:ful_description_almost_log_exc}, cases (1), (6), (8), (12); $\ll$ not log exceptional. Thick line indicates $R=D-E$, dashed line is $\ll$.}\label{Fig:ALE_2}
\end{figure}

\begin{proof}
Consider the case $r=1$. By Lemma \ref{lem:reordering_1st_2nd} we have $\alpha=\alpha_2\circ\alpha_1$, where $\alpha_1$ is a peeling of the first kind and $\alpha_2$ is a log crepant peeling of the second kind. Then $\ll$ is $\alpha_1$-almost log exceptional, so by Lemma \ref{lem:case_r=1} it is superfluous in $D+\ll$. Therefore, we may and will assume that $r<1$ and that $\ll$ is not superfluous in $D+\ll$. Let $\sigma\:\alpha(X)\to \ov X$ be the contraction of $\alpha(\ll)$ and $q\in \ov X$ the image of $\alpha(\ll)$. 

Assume that $q\in \ov X$ is singular. Let $\pi\:\tilde X\to \ov X$ be the minimal resolution of singularities and $\varphi\:X\to \tilde X$ the induced morphism. Put $\ov D=\sigma_*\alpha_*D$, $\tilde E=\Exc \pi$, and $\tilde D=\varphi_*D$; see diagram \eqref{eq:diagram-red_components}. 
By Lemma \ref{lem:peeling_for_squeezed}(3) every connected component of $E$ is a twig, rod or fork of $D$ or it is a $(-2)$-segment of $D$. Since $\ll$ is not superfluous in $D+\ll$, the divisor $\tilde D$ is not snc at $\varphi(\ll)$. Since $\sigma\circ \alpha=\pi\circ \varphi$ is a partial MMP run of the second kind, $\pi$ is a partial peeling of the second kind by Corollary \ref{cor:reordering_MMP}. By Lemma \ref{lem:peeling_for_squeezed} $\tilde E$ is a degenerate $(-2)$-segment of $\tilde D$, cf.\ Figure \ref{fig:degenerate_segment}. Let $G$ be a component of $\varphi_*^{-1}\tilde E$. Since $\pi^*K_{\ov X}=K_{\tilde  X}$ and $\pi^*\ov D=\tilde D$, we have $\cf(G;\ov X, r\ov D)=\cf(G;\tilde X, r\tilde D)=r=\cf(G;X,rD)$. By Lemma \ref{lem:ld_increasaes} it follows that $\alpha(\ll)$ is log exceptional of the second kind. The fact that $\ll$ is not superfluous in $D+\ll$ restricts possible types of blowing-ups constituting $\varphi^{-1}$. If $\#\tilde E=2$ we get $\ll+E=[3,1,3]$ and hence $r=\frac{1}{3}$. Assume that $\#\tilde E=1$. Then $\ll+E=[1,(2)_{m-1},3]$ or $\ll+E=[2,1,3,(2)_{m-2},3]$  for some $m\geq 1$, where $[3,(2)_{-1},3]:=[4]$. In the first case we get $r=\frac{1}{2}$ if $E$ is a twig of $D$ and $r=\frac{m}{2m+1}$ if $E$ is a rod of $D$. In the second case we get $r=\frac{1}{2}$. This gives (2), (3) and (4).

We may therefore assume that $q\in \ov X$ is smooth. We may also assume that we are not in case (5), so $\ll\cdot (K_X+rD)>0$ (respectively, $\ll\cdot( K_X+rD)\geq 0$). In particular, $E\neq 0$ (hence $r\neq 0$) and $\ll\cdot D>\frac{1}{r}$ (respectively, $\ll\cdot D\geq \frac{1}{r}$). Let $\pi\:\tilde X\to \ov X$ be the blowup of $q\in \ov X$, let $U=\Exc \pi$ denote the exceptional component and $\varphi\:X\to \tilde X$ the induced morphism. Again, put $\ov D=\sigma_*\alpha_*D$, $\tilde E=\Exc \pi$, and $\tilde D=\varphi_*D$. We have $\varphi_*^{-1}U\leq E$, as otherwise we would have $\alpha(X)=\tilde X$, which is impossible, as $\alpha$ is pure. Put $\mu=\pi_*^{-1}\ov D\cdot U=\mult_q\ov D$ and write $\varphi^*U=w\ll+\sum w_i E_i$, where the sum runs over all components $E_i$ of $E$. Since $q\in \ov X$ is smooth, $\mu$, $w$ and all $w_i$ are positive integers. We have $\pi^*(K_{\ov X}+r\ov D)=K+r\pi_*^{-1}\ov D+(r\mu-1)U$. Since $\sigma\circ \alpha$ is a partial MMP run of the second kind, $\cf(U;\ov X,r\ov D)\leq \cf (U,X,rD)=r$, hence $\mu \leq 1+\frac{1}{r}$. We have therefore 
\begin{equation}\label{eq:ALE_1}
\mu-1\leq \frac{1}{r}\leq \ll\cdot D.
\end{equation}
Moreover, at least one of the inequalities is strict. Indeed, otherwise $\ll$ is log exceptional of the second kind and at the same time $\cf(U;\ov X,r\ov D)=\cf (U,X,rD)$, so by Lemma \ref{lem:ld_increasaes} $\alpha(\ll)$ is log exceptional of the second kind, contrary to our assumptions. Put $R=D-E$. Since $\alpha(\ll)$ is log exceptional of the first or second kind, we have $0\geq \alpha(\ll)\cdot (K_{\alpha(X)}+r\alpha_*D)=\ll\cdot (K_X+E-\Bk E)+r\ll\cdot \alpha^*\alpha_*(R)$, hence $\ll\cdot (E-\Bk E)+r\ll\cdot \alpha^*\alpha_*(R)\leq 1$. We obtain
\begin{equation}\label{eq:ALE_2}
\frac{1}{r}\ll\cdot (E-\Bk E)+\ll\cdot \alpha^*\alpha_*(R)\leq \frac{1}{r}\leq \ll\cdot D.
\end{equation}
Again, at least one of the inequalities is strict. 

Assume that $w=1$. Then $\ll+E=[1,(2)_{k-1}]$ for some $k\geq 1$, hence $\Bk E=E$. The divisor $E$ is not a segment of $D$, as otherwise $\alpha^*\alpha_*R=D$ and we would have equalities in \eqref{eq:ALE_2}. Thus $E$ is a twig or a rod of $D$. The inequalities \eqref{eq:ALE_2} give additional restrictions on $\ll\cdot D$. We get (6). 

We may now assume that $w\geq 2$.  Assume that $\mu\geq 4$. Then $r\leq \frac{1}{3}$.  For every component $U$ of $E$ we have $\cf(U,\ov X,r\ov D)\leq\cf(X,rD)\leq r$. For a rod $U=[3]$ of $D$ have $\cf(U,\ov X,\ov D)=\frac{1}{3}$, so by Lemma \ref{lem:Alexeev_dlt} connected components of $E$ consist of $(-2)$-curves or are equal to $[3]$, and in the latter case $r=\frac{1}{3}$ and $\alpha$ is not of the first kind. If $r\neq \frac{1}{3}$ then due to negative definiteness of $\ll+E$ we get $\ll+E=[1,2,\ldots,2]$, contrary to the assumption $w\geq 2$. Thus $r=\frac{1}{3}$ and $\mu=4$. Since $w\geq 2$, it follows that $\ll+E=[2,1,3]$ and $\alpha$ is not of the first kind. The inequality \eqref{eq:ALE_2} reads as $1+\ll\cdot \alpha^*\alpha_*(R)\leq 3<\ll\cdot D$, which gives $\ll\cdot (\alpha^*\alpha_*(R)-R)\leq 0$. This is possible only if $\alpha^*\alpha_*(R)=R$, that is, $R\cdot E=0$. Then $R\cdot w\ll=\mu=4$, so $\ll\cdot R=2$. This gives (7).

We may now assume that $\mu\leq 3$. Note that since $E+\ll$ contracts to a smooth point, $\ll$ is superfluous in $E+\ll$. But $\ll$ is not superfluous in $D+\ll$, so we infer that $\ll\cdot R\geq 1$. We have $w\geq 2$ and $\mu=R\cdot(w\ll+\sum w_i E_i)\geq w R\cdot \ll+R\cdot E\geq 2R\cdot\ll+R\cdot E\geq 2$. It follows that $\ll\cdot R=1$ and $\mu\geq w+R\cdot E$. Since $\ll$ is not superfluous in $D$, the divisor $E$ has two connected components. We obtain $\ll\cdot D=3$. Assume that $\mu=2$. Then $R\cdot E=0$ and $w=2$, which gives  $\ll+E=[2,1,3,(2)_{k-1}]$ for some $k\geq 1$. The inequality \eqref{eq:ALE_2} reads as $\frac{1}{r}(1-\frac{k+1}{2k+1})+1\leq \frac{1}{r}\leq 3$. This gives (8). 

We are left with the case $\mu=3$. We have $\ll\cdot R=1$ and $R\cdot \sum w_i E_i=3-w$. By \eqref{eq:ALE_1}  $\frac{1}{3}\leq r\leq \frac{1}{2}$. Since $\sigma\circ \alpha$ is a partial MMP run of the second kind, for every component $E_i$ of $E$ we have  \begin{equation}\label{cf:compute_1}
\cf(E_i;\alpha(X),r\alpha_*D)\leq \cf(E_i;X,rD)=r. 
\end{equation}
Assume that $w=2$. Then again $\ll+E=[2,1,3,(2)_{k-1}]$ for some $k\geq 1$ and $R\cdot E=1$. Let $E_0\leq E$ be the $(-3)$-curve and $E_1=[2]$ the $(-2)$-rod of $E$. If $R$ meets $E_1$ then by \eqref{eq:Bk_rD} we get $\cf(E_0;\alpha(X),r\alpha_*D)=\frac{k}{2k+1}\leq r$. In this case by \eqref{eq:ALE_2} we have $r\leq \frac{2k+2}{6k+3}$,  which gives (9). If $R$ meets $[3,(2)_{k-1}]$ not in $E_0$ then $k\geq 2$ and the above inequalities give $\frac{1}{2}\leq r\leq \frac{k+1}{4k}$, which is impossible. Thus $R$ meets $E_0$. The above inequalities imply that $r=\frac{1}{2}$ and $k=1$, which gives (10). Finally, assume that $w=3$. Now $R\cdot E=0$. Since $E$ has two connected components and $\mu=3$, the divisor $\ll+E$ is a chain, either $[3,1,2,3]$ or $[2,2,1,4,(2)_{k-1}]$ or $[3,1,2,3,(2)_{k}] $ for some $k\geq 1$. In the second and third case we obtain $k=1$ and $r=\frac{1}{2}$, which gives (11). In the first case $\frac{2}{5}\leq r\leq \frac{7}{15}$, which gives (12). 
\end{proof}

\medskip
\subsection{Almost minimalization for uniform $\epsilon$-dlt and $\epsilon$-lc surfaces}

For definitions of GMRLC, $\epsilon$-dlt and $\epsilon$-lc log surfaces see Definitions \ref{def:GMRLC} and \ref{dfn:eps-dlt}. We recall that an $(1-r)$-lc log surface is in particular log canonical, hence GMRLC (for $r\neq 1$ it is dlt, hence $\Q$-factorial). It follows that the canonical divisor, components of the boundary and all log exceptional curves of the first kind are $\Q$-Cartier. Also, the image of a GMRLC log surface after a contraction of any log exceptional curve of the first or second kind is a GMRLC log surface.

\begin{lem}\label{lem:resolution_is_peeling}
Assume 	that $(\ov X,r\ov D)$, where $\ov D$ is reduced and $r\in [0,1]\cap \Q$, is $(1-r)$-lc. Let $\alpha\:X\to \ov X$ be a proper birational morphism from a $\Q$-factorial (normal) surface. Put $D=\alpha_*^{-1}\ov D+\Exc \alpha$. Then $\alpha\:(X,rD)\to (\ov X,r\ov D)$ is a partial peeling of the second kind. If $(\ov X,r\ov D)$ is $(1-r)$-dlt then $\alpha$ is of the first kind, unless $r=1$ and some center of $\alpha$ is a normal crossing point of $(\ov X,\ov D)$. 
\end{lem}

\begin{proof}Let $E$ be a prime component of $\Exc \alpha$. We have $\cf(E;\ov X, r\ov D)\leq r=\cf(E;X,rD)$, hence $\alpha$ is a partial peeling of the second kind by Corollary \ref{cor:improving_ld_gives_MMP}(2). Assume that $(\ov X,r\ov D)$ is $(1-r)$-dlt and that either $r<1$ or $r=1$ and $\alpha(E)$ is not a normal crossing point of $(\ov X,\ov D)$. It is sufficient to argue that $\cf(E;\ov X, r\ov D)<r$, because then $\alpha$ is of the first kind by Corollary \ref{cor:improving_ld_gives_MMP}(1). Composing with a resolution of singularities of $X$ we may assume that $\alpha$ is a log resolution. If $r<1$ then we are done by the definition of a $(1-r)$-dlt surface, cf. Remark \ref{rem:dlt_resolution_dependence}. Assume that $r=1$ and write  $\alpha=\pi\circ \alpha'$, where $\pi$ is the minimal log resolution of the singularity $\alpha(E)\in \ov X$. If $\cf(E;\ov X, r\ov D)=1$ then $\alpha'$ contracts $E$ and by Lemma \ref{lem:ld_increasaes} $\pi$ is the identity, hence $\alpha(E)$ is a normal crossing point on $(\ov X,\ov D)$.
\end{proof}

\begin{rem}
In Lemma \ref{lem:resolution_is_peeling} the log surface $(X,rD)$ is not necessarily $(1-r)$-lc. For instance, for $r=0$ consider a smooth projective surface containing the chain $[1,3,1]$. Let $X$ be its image after the contraction of the middle curve. Then $X$ is not canonical and the contraction of the remaining part of the chain maps it onto a smooth surface $\ov X$.
\end{rem}

\begin{cor}[Resolutions of $(1-r)$-log canonical surfaces]\label{cor:lc_sing_restrictions_r}
Let $(\ov X,r\ov D)$, where $\ov D$ is reduced and $r\in [0,1]\cap \Q$, be a germ of a $(1-r)$-lc log surface. Assume that $\ov X$ is singular. Let $(X,D)\to (\ov X,\ov D)$, where $D$ is the reduced total transform of $\ov D$, be the minimal resolution. Then the exceptional divisor $E$ is as in Lemma \ref{lem:lc_singularities} and the following hold.
\begin{enumerate}[(1)]
\item If $r=1$ then $D$ is snc or $D=E$ and $E$ is a nodal rational curve.
\item If $r<1$ then $E$ is as in case (1), (2) or (5) of Lemma \ref{lem:lc_singularities} and in case (5) $E$ consists of $(-2)$-curves.
\item Assume that $r<1$ and $D$ is not snc. Then $E=[(2)_k]$, where $k\in \{1,2\}$, is a degenerate segment and $D-E\neq 0$. Moreover, denoting the unique common point of $E$ and $D-E$ by $q$, we have:
\begin{enumerate}[(a)]
\item $k=2$, $q\in D-E$ is smooth and $r\leq \frac{1}{2}$.
\item $k=1$, $q\in D-E$ has multiplicity $2$ and $r\leq \frac{1}{2}$,
\item $k=1$, $q\in D-E$ is smooth and $r\leq \frac{2}{3}$,
\end{enumerate}
If $(X,rD)$ is $(1-r)$-dlt then the inequalities are strict.
\item $(X,rD)$ is $(1-r)$-lc. It is $(1-r)$-dlt if $(\ov X,r\ov D)$ is $(1-r)$-dlt.
\end{enumerate}
\end{cor}

\begin{proof} The morphism $(X,rD)\to (\ov X,r\ov D)$ is a pure partial peeling of the second kind by Lemma \ref{lem:resolution_is_peeling}. 

(1) If $r=1$ and $D$ is not snc then by Lemma \ref{lem:lc_singularities} $E$ is a degenerate segment or a degenerate cycle, hence the minimal resolution is log crepant. For cases other than the rational nodal curve we easily find in the minimal log resolution curves with coefficients bigger than $1$. This gives (1).

(2), (3) Part (2) follows from Lemma \ref{lem:peeling_for_squeezed}. The additional numerical restrictions in case (5) come from a straightforward computation of coefficients of exceptional divisors of the minimal log resolution dominating the minimal resolution. For instance, in the case when $q\in D-E$ is a cusp, writing the exceptional divisor of the log resolution $\pi\:(\tilde X,\tilde D)\to (\ov X,\ov D)$ as $E_0+E_1+\ldots+E_n+E_{n+1}=[3,(2)_{n-2},3,1,2]$, where $n\geq 2$, we have 
\begin{equation}\label{eq:cusp-segment-resolution}
\pi^*(K_{\ov X}+r\ov D)=K_{\tilde X}+r\pi_*^{-1}\ov D+\sum_{i=0}^{n-1}(i(2r-1)+r)E_i+2n(2r-1)E_{n}+n(2r-1)E_{n+1},
\end{equation}
so the coefficients of exceptional components are not bigger than $r$ if and only if $r\leq \frac{1}{2}$.

(4) We may assume that $D$ is not snc, as otherwise the claim holds. If $r=1$ then $D=E$ is a nodal rational curve and if $r<1$ then by Lemma \ref{lem:peeling_for_squeezed}(3) $E$ is a degenerate $(-2)$-segment of $D$. In both cases the minimal resolution is log crepant, so the claim holds.
\end{proof}

We discuss the effect of squeezing on log singularities for surfaces of type $(X,rD)$, with $X$ smooth and $D$ reduced. A partial squeezing is an almost minimalization of a partial peeling (which is a partial MMP run). Since partial squeezing itself is not a partial MMP run on $(X,rD)$, it does not necessarily respect $(1-r)$-(divisorial) log terminality. This is exactly the case in the following example.

\begin{ex}[Partial squeezing does not respect $(1-r)$-log terminality of $(X,rD)$] \label{ex:(1-r)-log_terminality}
Let $r\in (0,1]\cap \Q$ and $n\in \N_+$ be such that the interval $[\frac{1}{r},\frac{1}{r}+\frac{1}{n}]$ contains some positive integer $N$. Blowing up over a point $p\in \P^2$ we create a $\P^1$-fibered surface $X$ with a unique degenerate fiber $F=[2,\ldots,2,1,n]$ and a section $U$ which is a $(-1)$-curve meeting a tip of $F$ belonging to a maximal $(-2)$-chain $\Delta\subset F$. Let $\alpha\:X\to \ov X$ be the contraction of $\Delta$ and let $D=F_1+\ldots +F_N+U+\Delta$, where $F_1,\ldots,F_N$ are distinct non-degenerate fibers. 
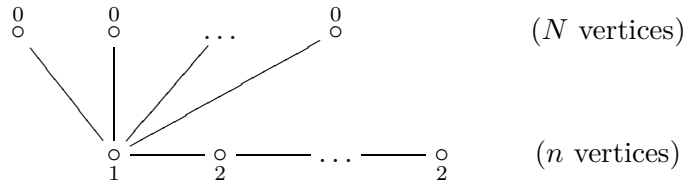
\begin{figure}[h]
{$\xymatrix{
{\overset{0}{\circ}}\ar@{-}[rd]  & {\overset{0}{\circ}}\ar@{-}[d] & {\ldots}\ar@{-}[ld] &  {\overset{0}{\circ}}\ar@{-}[lld] & {}&{(N \text{\ vertices})  }\\
{} &{\underset{1}{\circ}}\ar@{-}[r] & {\underset{2}{\circ}}\ar@{-}[r] & {\ldots}\ar@{-}[r] &{\underset{2}{\circ}} & {(n \text{\ vertices})}}$}

\caption{The dual graph of $D$ in Example \ref{ex:(1-r)-log_terminality}.}
\end{figure}
Put $\ov D=\alpha_*D$ and $\ov U=\alpha(U)$. Then $\alpha\:(X,rD)\to (\ov X,r \ov D)$ is a pure peeling, hence $(\ov X,r\ov D)$ is $(1-r)$-dlt. We have (cf. \eqref{eq:Bk_rD}) $$\ov U^2=-\frac{1}{n}<0 \quadtext{and} \ov U\cdot (K_{\ov X}+r\ov D)=r(N-\frac{1}{n})-1\leq 0,$$ so $\ov U$ is log exceptional of the first or second kind and hence $U$ is $\alpha$-redundant of the first or second kind. Let $\sigma\:X\to X'$ be the contraction of $U$. Put $D'=\sigma_*D$. Then $\cf(U;X',rD')=-1+r+rN\geq r$, so while $(X,rD)$ is $(1-r)$-dlt, $(X',rD')$ is not. If $\frac{1}{r}<N< \frac{1}{r}+\frac{1}{n}$  then $U$ is $\alpha$-redundant of the first kind and $(X',rD')$ is not even $(1-r)$-lc.

We note that although $\ov U$ is log exceptional on $(\ov X, r\ov D)$ and $U$ is redundant on $(X,rD)$, we have $\ov U\cdot (K_{\ov X}+r(\ov D-\ov U))=rN-1\geq 0$, so $\ov U$ is not log exceptional on $(\ov X, r(\ov D-\ov U))$ and hence $U$ is not almost log exceptional on $(X,r(D-U))$.
\end{ex}

\begin{ex}[Example with $\psi=\psi_{\am}$]\label{ex:psi_am_is_MMP-order-of-witnessing} 
Consider a log surface $(X,rD)$ with $X$ smooth, $r\in (0,1]\cap \Q$, $D$ reduced containing a twig $E=[(2)_{k-1}]$ for some $k\geq 1$, and containing a $(-1)$-curve $\ll$ such that $\ll\cdot E=1$ and for $R=D-\ll-E$ we have $0\leq \ll\cdot R-\frac{1}{r}< \frac{1}{k}$; see case (5) of Proposition \ref{prop:ful_description_redundant}. Let $\alpha$ be the contraction of $E$, let $\psi$ be the contraction of $E+\ll$ and let $E_1,\ldots, E_{k-1}$ be the subsequent components of $E$ enumerated so that $E_1$ is a tip of $D$ and $E_{k-1}$ meets $\ll$.
\begin{figure}[h]
	\begin{tikzpicture}[scale=1.5]
		\draw (-1,3) -- (1.6,3);
		\node at (-0.9,2.8) {\small{$\ll$}};
		\node at (-0.9,3.2) {\small{$-1$}};
		\node at (-0.65,2.1) {\small{$R$}};
		\draw[thick] (-0.4,1.9) -- (-0.4,3) to[out=90,in=180] (-0.2,3.2);
		\draw[thick, dotted] (-0.2,3.2) to[out=0,in=180] (0.2,2.8) to[out=0,in=180] (0.6,3.2);
		\draw[thick] (0.6,3.2) to[out=0,in=90] (0.8,3) -- (0.8,2.7); 
		\draw(1.4,3.2) -- (1.6,2.2);
		\node[right] at (1.4,2.7) {\small{\ $E_{k-1}$}};
		\node[left] at (1.6,2.2) {$\vdots$};
		\node[right] at (1.5,2.1) {\small{\ $[(2)_{k-1}]$}};
		\draw (1.6,2) -- (1.4,1);
		\node[right] at (1.4,1.4) {\small{\ $E_1$}};
	\end{tikzpicture}
\caption{The graph of $D$ in Example \ref{ex:psi_am_is_MMP-order-of-witnessing}.}
\end{figure}
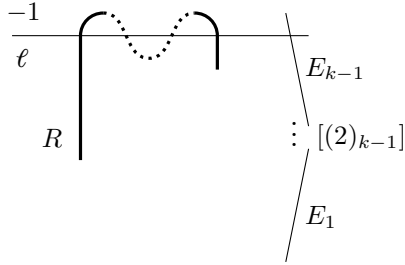
We have $\ll\cdot (K+rD)=-1+r\ll\cdot R\geq 0$, so $\ll$ is not log exceptional of the first kind. On the other hand, $\alpha(\ll)\cdot (K_{\alpha(X)}+r\alpha_*D)=-1+r(\ll\cdot R-\frac{1}{k})< 0$, so $\alpha(\ll)$ is log exceptional of the first kind, that is, $\ll$ is $\alpha$-redundant on $(X,rD)$. Therefore, contracting subsequently $E_1, E_2, \ldots, E_{k-1}, \ll$ we see that $\psi$ is a partial MMP run, a partial peeling of $(X,rD)$. On the other hand, under this contraction the image of $\ll$ is a smooth point of $\psi(X)$, hence $\psi_\am=\psi$, i.e. $\psi$ is the almost minimalization morphism, too, namely a $K_X$-MMP run. The order of contractions witnessing the latter is, of course, the opposite one: $\ll, E_{k-1},\ldots, E_1$.
\end{ex}

\begin{cor}[The effect of squeezing on log singularities]\label{cor:singularities_after_squeezing}
Assume that $X$ is smooth, $D$ is reduced, $r\in [0,1]\cap \Q$ and $(X,rD)$ is $(1-r)$-dlt ($(1-r)$-lc). Let $\sigma\:(X,rD)\to (X',rD')$ be a partial squeezing of the first (second) kind. Then $X'$ is smooth and the following hold:
\begin{enumerate}[(1)]
\item If $(X',rD')$ contains no redundant curve of the first kind meeting $\Supp \sigma_*(\Exc \sigma)$ (in particular, if $\sigma$ is a squeezing) then it is $(1-r)$-dlt (respectively, $(1-r)$-lc).
\item If $\frac{1}{r}\in \N\cup\{\8\}$ then $(X',rD')$ is $(1-r)$-lc.
\end{enumerate}
\end{cor}

\begin{proof}
Since $X$ is smooth, redundant curves of the first and second kind are in particular $(-1)$-curves, so the smoothness of $X'$ is clear. We may assume that $\sigma\neq \id$, hence $r\neq 0$. Let $\tilde E$ be the exceptional divisor of some pure partial peeling morphism $\alpha$ of the first (second) kind and let $\ll\leq D$ be an $\alpha$-redundant curve of the first (first or second) kind. Let $E$ be the sum of connected components of $\tilde E$ meeting $\ll$. By Corollary \ref{cor:reordering_MMP} and \eqref{eq:composing_almost_min} we may assume that $\tilde E=E$, so we have $\ll\leq \Exc \sigma\leq \ll+E$. Denote by $\sigma'\:X\to Y$ the contraction of $\ll$ and by $\ov \sigma$ the contraction of $\alpha(\ll)$. We may assume that $E\neq 0$ and that $\ll$ is not log exceptional of the first or second kind, as otherwise we are done by Corollary \ref{lem:eps-lc_is_respected} and induction. If $r=1$ then we are in case (1) of Proposition \ref{prop:ful_description_redundant} with $E\neq 0$, which implies that $\ll$ meets two distinct components of $D$ and hence we are done by induction. We may thus assume that $r<1$. We are therefore in case (4), (5) or (6) of Proposition \ref{prop:ful_description_redundant} with strict inequalities on the left in the defining conditions for the two latter cases. If $\frac{1}{r}$ is an integer this gives a contradiction. 

It remains to prove (1). Let $q=\bar \sigma\circ \alpha(\ll)$. Assume that $q\in\ov X$ is smooth. We have $(\bar \sigma\circ \alpha)_\am=\bar \sigma\circ \alpha$, so $\bar \sigma\circ \alpha$ is a partial squeezing of the first (second) kind. By Lemma \ref{lem:relativeMMP_unique} it factors through $\sigma$. By the assumption on $\sigma$ we have in fact $\sigma=\bar \sigma\circ \alpha$, so $\sigma$ is a peeling of $rD$ and hence $(X',rD')$ is $(1-r)$-dlt (respectively, $(1-r)$-lc) by Corollary \ref{lem:eps-lc_is_respected}. Assume that $q\in \ov X$ is singular. Then we are in case (4) of Proposition \ref{prop:ful_description_redundant} with $E=U_1+U_2$, where $U_1=[2]$ and $U_2=[4]$. The curve $\sigma'(U_1)$ is log exceptional of the first kind, hence by assumption $\sigma$ contracts exactly $\ll+U_1$. Contracting these curves in the reversed order we see that $\sigma$ is at the same time a peeling of the second kind, hence we are done.
\end{proof}

We have seen in Example \ref{ex:composition_of_nef} that a composition of pure peeling morphisms does not have to be pure. In the following lemma we study this problem in more detail.

\begin{lem}[Non-purity after taking a minimal resolution]\label{lem:peeling_plus_resolution}
Assume that $(X,rD)$, where $D$ is a connected reduced divisor and $r\in[0,1]\cap \Q$, is $(1-r)$-dlt (respectively, $(1-r)$-lc). Denote by $\pi\:\tilde X\to X$ the minimal resolution of singularities and put $\tilde D=\pi^{-1}_*(D)+\Exc \pi$. Let $\alpha$ be a pure partial peeling of $(X,rD)$ of the first (second) kind. Then $\tilde \alpha:=\alpha\circ \pi$ is a partial peeling of $(\tilde X,r\tilde D)$ of the first (respectively, second) kind. If $\tilde \alpha$ is not pure then the following hold.
\begin{enumerate}[(1)]
\item Every connected component of $\Exc \tilde \alpha_\am$ is a non-degenerate segment or a non-maximal twig of $\tilde D$.
\item If $\Exc \alpha$ is irreducible then  $\pi_*^{-1}\Exc \alpha\cdot \Exc \pi=2$ and $\pi_*^{-1}\Exc \alpha\cdot \Bk (\Exc \pi)\leq 1$. 
\item If $r<1$ or $\pi$ contracts no segment of $\tilde D$ then $D$ is irreducible.
\end{enumerate}
\end{lem}

\begin{proof}\label{prop:psi_am-is_(X,D)_MMP_run_too}

By Lemma \ref{lem:resolution_is_peeling} $\pi$ is a peeling of the first (second) kind of $(\tilde X,r\tilde D)$, hence so is $\tilde \alpha=\alpha\circ \pi$. Assume that $\tilde \alpha$ is not pure. In particular, $\Exc \alpha\neq 0$.

(1), (2) Let $Y\to \alpha(X)$ be the minimal resolution of singularities of $\alpha(X)$. Clearly, $\tilde \alpha_\am\:\tilde X\to Y$ is the induced morphism between resolutions. A partial MMP run decreases coefficients, and by Lemma \ref{lem:pure_composition} if we decompose a pure peeling into two peelings then both of them are pure. It follows that by induction we may assume that $\ll:=\Exc \alpha$ is irreducible. Put $\tilde \ll=\alpha_*^{-1}\ll$, $E=\Exc \pi$ and $R=\tilde D-\tilde \ll-E$. Since $D$ is connected, each connected component of $E$ meets $\ll$. Since $\pi$ is pure, $\tilde \ll$ is the only redundant component of $\tilde D$. By Lemma \ref{lem:lc_singularities}  connected components of $E$ are twigs, half-benches and segments of $\tilde D$. Since $\alpha$ is pure, $\ll\cdot K_X\geq 0$. Note that twigs and segments are chains and half-benches are chains of type $[2,b,2]$ or admissible forks of type $(b;2,2,n)$, for some $b, n\geq 2$. We have $\ll\cdot K_X=\tilde \ll\cdot (K_{\tilde X}+E-\Bk E)<\tilde \ll\cdot (K_{\tilde X}+E)=\tilde \ll\cdot E-1$. By Proposition \ref{prop:ful_description_redundant} $\tilde \ll\cdot E\leq 2$, hence $\tilde \ll\cdot E=2$ and $\tilde \ll\cdot \Bk E\leq 1$. It follows that $\tilde \ll$ meets no $(-2)$-rod of $E$. Indeed, if $E_0$ is a $(-2)$-chain which is a connected component of $E$ then $\Bk E_0=E_0$, hence $\tilde \ll\cdot \Bk (E-E_0)\leq 0$, which is impossible. By Corollary \ref{cor:lc_sing_restrictions_r} $\tilde D$ is snc at the points of $E$. We infer that we are in case (1) or (2) of Proposition  \ref{prop:ful_description_redundant} and we have $\tilde \ll\cdot R=0$. Therefore, connected components of $E$ are non-degenerate segments, half-benches and twigs of $\tilde D$. Moreover, $\tilde \ll$ is superfluous in $\tilde D$, unless $r=1$ and $\tilde D=E+\tilde \ll$ is a rational cycle consisting of two components, where $E^2\leq -5$ due to the negative definiteness of $E+\tilde \ll$.

We have $\Exc \tilde \alpha_\am\leq E+\tilde \ll$. From the description of the connected components of $E$ we see that  $\Exc \tilde \alpha_\am$ is a non-degenerate segment or a twig of $\tilde D$. Suppose that it is a maximal twig (possibly a rod) of $\tilde D$. Then it is equal to $E_1+\ll+E_2$, where $E_1$ is a chain and a connected component of $E$, and $E_2$ is either a chain which is a connected component of $E$ or the long twig contained in the central chain of a half-bench meeting $\tilde \ll$. Using Definition \ref{dfn:Bark} we rewrite the inequality $\tilde \ll\cdot \Bk E\leq 1$ as $\ind\trp(E_1)+\ind\trp(E_2)+\frac{1}{d(E_1)}+\frac{u}{d(E_2)}\leq 1$, where $u=1$ if $E_2$ is a connected component of $E$ and $0<u\leq 1$ is as in Definition \ref{dfn:Bark}(3) otherwise. By Lemma \ref{lem:d(D_1+D_2)} we have \begin{equation}
1-\ind\trp(E_1)-\ind\trp(E_2)=\frac{d(E_1+\ll+E_2)}{d(E_1)d(E_2)},
\end{equation}
so we obtain $ud(E_1)+d(E_2)\leq d(E_1+\ll+E_2)$. Since $E_1+\ll+E_2$ contracts to a smooth point, we have $d(E_1+\ll+E_2)=1$, which is a contradiction, as $d(E_1), d(E_2)\geq 2$.

(3) We may assume that $\alpha=\ctr_{\ll}\circ \gamma$, where $\tilde \gamma:=\gamma\circ \pi$ is pure and $\ll$ is such that $\tilde \ll:=\tilde \gamma_*^{-1}\ll$ is redundant. Put $E=\Exc \tilde \gamma$ and $E'=\Exc \pi$. By the proof of (1) and (2) $\tilde \ll$ is a superfluous component of $\tilde D$ with $\tilde \ll\cdot E=2$, $\tilde \ll\cdot \Bk E\leq 1$ and $\tilde \ll\cdot R=0$. Suppose that $r<1$ and a connected component of $E'$ is a segment of $\tilde D$, call it $T'$. Then by Lemma \ref{lem:peeling_for_squeezed}(3) the segment consists of $(-2)$-curves, so $\Bk T'=T'$ and hence the inequality $\tilde \ll\cdot \Bk E\leq 1$ fails. We may therefore assume that $\pi$ contracts no segment of $\tilde D$. By Lemma \ref{lem:lc_singularities} connected components of $E$ are segments, half-benches and twigs of $\tilde D$.  Suppose that some connected component of $E$, call it $T$, is a segment of $\tilde D$. Since no connected component of $E'$ is a segment of $\tilde D$, we infer that $E'$ is disjoint from $T$. Then $T+\tilde \ll$ is a segment of $\tilde D$, hence it does not meet half-benches and twigs of $\tilde D$, which by the connectedness of $E+\tilde \ll$ gives $E=T$. But then $\pi=\id$ and $\alpha$ is not pure; a contradiction.

Therefore, connected components of $E$ are half-benches and twigs of $\tilde D$. Since $\tilde \ll\cdot R=0$, it follows that $R=0$, hence $D=\pi_*(E-E')+\ll$. Since $\alpha$ is pure, we have $\gamma_*^{-1}\ll\cdot K_X\geq 0$, so by the computation done in the proof of (1) and (2) we get in particular $\tilde \ll\cdot E'=2$. It follows that $E'=E$, so $D$ is irreducible. 
\end{proof}

As we have seen in Example \ref{ex:psi_am_is_MMP-order-of-witnessing}, it can happen than given an MMP run $\psi$ on a log surface $(X,rD)$ its almost minimalization $\psi_{\am}$  is also an MMP run for the same log surface (by definition it is an MMP run on the surface $X$). In this case, due to properties of the MMP run, if $(X,rD)$ is $(1-r)$-dlt or $(1-r)$-lc then so is its almost minimal model. In Lemma \ref{lem:peeling_am-is-MMP} we will show that this situation is a rule when $r<\frac{1}{2}$. The following example shows on the other hand, that for $r\geq \frac{1}{2}$ one cannot expect $\psi_{\am}$ to be an MMP run on the log surface in general.

\begin{ex}[Almost minimalization may not be an MMP run] We modify Example \ref{ex:composition_of_nef}.
Consider a smooth projective surface $X$ with a divisor $[n_1,1,n_2]$ on it, where $n_1\geq n_2\geq 2$ are positive integers such that $\frac{1}{n_1}+\frac{1}{n_2}> \frac{1}{2}$ and $(n_1,n_2)\neq (2,2)$. Let $\ll$ be the $(-1)$-curve and let $D_i$, $i=1,2$ be the component with self-intersection $-n_i$. Let $\psi\:X\to \ov X$ be the contraction of $\ll+D_1+D_2$ and let $\alpha\:X\to Y$ be the contraction of $D_1+D_2$. Then $\sigma=\psi\circ \alpha^{-1}\:Y\to \ov X$ is the contraction of $\alpha_*\ll$. We have $\alpha^*K_Y=K_X+(1-\frac{2}{n_1})D_1+(1-\frac{2}{n_2})D_2$, hence $\alpha_*\ll\cdot K_Y=1-\frac{2}{n_1}-\frac{2}{n_2}<0$. It follows that $K_X$ is $\alpha$-nef and $\alpha(\ll)$ is log  exceptional of the first kind.

Assume that $r>1-\frac{2}{n_1}$. Then $\alpha$ is a peeling of $D=D_1+D_2$ and $\psi$ is an MMP run on $(X,rD)$. If $n_2=2$ then $\Exc \psi_{\am}=\ll+D_2$, hence $\psi_{\am}$ is an MMP run on $(X,rD)$. But take $n_2=3$. Then $n_1\in \{3,4,5\}$ and the bound on $r$ reads as $r>\frac{1}{3}$. We have now $\Exc \psi_{\am}=\ll$ and since $\ll\cdot (K_X+rD)=-1+2r$, for $r\geq \frac{1}{2}$ the almost minimalization $\psi_{\am}$ is not an MMP run on $(X,rD)$.
\end{ex}

\begin{lem}\label{lem:peeling_am-is-MMP}
 Let $(X,rD)$, with $D$ reduced and $r\in [0,1]\cap \Q$, be a log surface which is $(1-r)$-dlt (respectively, $(1-r)$-lc). Let $\psi=\ctr_{\alpha(\ll)}\circ \alpha$, where $\alpha$ is a pure partial peeling of the first (second) kind of $rD$ and $\ll\subseteq X$ is such that $\alpha(\ll)$ is log exceptional on $(\alpha(X),r\alpha_*D)$ of the first (second) kind with $\ll\cdot K_X<0$. Let $\tilde \ll$ and $\tilde D$ denote the proper transform of $\ll$ and the reduced total transform of $D$ on the minimal resolution of singularities of $X$, respectively. If $\psi_{\am}$ is not an MMP run of the first (second) kind on $(X,rD)$ then the following conditions hold:
 \begin{enumerate}[(1)]
\item $\psi(\ll)\in \psi(X)$ is a singular point,
\item The image of $\tilde \ll$ on the minimal resolution of singularities of $\psi(X)$ is a normal crossing point of the image of $\tilde D$, simple normal crossing if $r<1$,  and hence $\tilde \ll$ is a superfluous component of $(\tilde D+\tilde \ll)\redd$.
\item If $\ll\leq D$ then $r=1$, $\ll$ is log exceptional of the second kind and the minimal resolution of singularities of $X$ is log crepant in a neighborhood of $\ll$.
\item If $\ll\nleq D$ then $r\geq \frac{1}{2}$; in case $r= \frac{1}{2}$ the curve $\ll$ is log exceptional of the second kind and the minimal resolution of singularities is log crepant  in a neighborhood of $\ll$.
\end{enumerate}
\end{lem}

\begin{proof}
Let $\ov X=\psi(X)$ and $\ov D=\psi_*D$. Let $\pi\:\tilde X\to X$ be the minimal resolution of singularities and let $\tilde \alpha=\alpha\circ \pi$. Put $\tilde \ll=\pi_*^{-1}\ll$, $\tilde D=\pi_*^{-1}D+\Exc \pi$, $\tilde \psi:=\psi\circ \pi$ and $E=\Exc \tilde \alpha$. Let $q=\psi(\ll)$. By Lemma \ref{lem:resolution_is_peeling} $\tilde \alpha$ is a partial peeling of $r\tilde D$ of the first (second) kind. We may assume that $E+\tilde \ll$ is connected, and hence that $\ov X\setminus \{q\}$ is smooth. Since $\pi$ is pure, we have $\tilde \ll\cdot K_{\tilde X}\leq \ll\cdot K_X<0$, so $\tilde \ll$ is a $(-1)$-curve.

Let $\sigma=\ctr_{\ll}$. By Corollary \ref{cor:reordering_MMP} $\psi\circ\sigma^{-1}\:(\sigma(X),r\sigma_*D)\to (\ov X,r\ov D)$ is a partial peeling of the first (second) kind. The curve $\ll$ is not log exceptional of the first (second) kind, because otherwise $(\sigma(X),r\sigma_*D)$ is $(1-r)$-dlt ($(1-r)$-lc), which
by \eqref{eq:composing_almost_min} implies that $\psi_{\am}$ is an MMP run of the first (second) kind. Since $\pi$ is a peeling of the second kind, $\ll\cdot (K_X+rD)\leq \tilde \ll\cdot (K_{\tilde X}+r\tilde D)$ and the inequality is strict, unless $\pi$ is log crepant  in a neighborhood of $\ll$. In particular, $\tilde \ll$ is not log exceptional of the first (second) kind. 

Suppose that $q\in \ov X$ is smooth. Then $\ov X$ is smooth, so $\tilde \psi=\tilde \psi_\am$, hence $\tilde \psi$ is a partial $K_{\tilde X}$-MMP run of the first (second) kind. By Corollary \ref{cor:reordering_MMP} $\psi$ is a partial $K_X$-MMP run of the first (second) kind, too. Thus $\psi_\am=\psi$, so $\psi_{\am}$ is an MMP run of the first (second) kind; a contradiction. This gives (1). 

Put $\tilde D^-=\tilde D-\tilde \ll$ if $\tilde \ll\leq \tilde D$ and $\tilde D^-=\tilde D$ otherwise. Let $\theta\:Y\to \ov X$ be the minimal resolution of singularities of $\ov X$. Then  $\tilde \psi_\am\:\tilde X\to Y$ is the induced morphism of resolutions. 
\begin{equation*}
	\begin{tikzcd}
		\widetilde{X} \ar[rr, bend left=60, looseness=0.5, "\tilde \psi_\am"] \ar[d,"\pi"]  & Y' \ar[r, "\varphi"] \ar[d,"\theta' "] & Y \ar[ddl,"\theta ", bend left=30] \\ 
		X \ar[r,"\psi_{\textnormal{am}}"] \ar[d,"\alpha"] \ar[rd,"\psi"] & X' \ar[d,"\psi_{\min}"] & \\
		\alpha(X) \ar[r,"\textnormal{ctr}_{\alpha(\ll)}"'] & \overline{X} & 
	\end{tikzcd}
	\label{fig:aMMP_not_MMP}
\end{equation*}	
Suppose that $\tilde \psi_\am(\ll)$ is not a normal crossing point of $(\tilde \psi_\am)_*\tilde D$. By Corollary \ref{cor:lc_sing_restrictions_r} $r<1$ and $\Exc \theta$ is a $(-2)$-segment, hence $\theta$ is crepant and log crepant. Let $\theta'\:Y'\to \psi_\am(X)$ be a minimal resolution of singularities. Then $\psi_{\min}\circ \theta'$ is some resolution of singularities of $\ov X$, hence $\psi_{\min}\circ \theta'=\theta\circ \varphi$ for some partial $K_{Y'}$-MMP $\varphi$. Put $X'=\psi_{\am}(X)$. Since $\psi_{\min}$ is pure and $\theta'$ is minimal, we have $\psi_{\min}^*K_{\ov X}\geq K_{X'}$ and $(\theta')^*K_{X'}\geq K_{Y'}$, hence $(\psi_{\min}\circ \theta')^*K_{\ov X}\geq K_Y'$. On the other hand, $\varphi$ is a $K_{Y'}$-MMP, so we have $K_{Y'}\geq \varphi^*K_Y$. But $\theta^*K_{\ov X}=K_Y$, hence $K_{Y'}\geq (\theta\circ \varphi)^*K_{\ov X}$. We infer that $\theta\circ \varphi$ is crepant, hence $\varphi$ is crepant. It follows that $\varphi=\id_{Y'}$. Since $\theta$ is crepant and log crepant, so is $\psi_{\min}$, too. It follows that for every component $G$ of $\Exc \psi_\am$ we have $\cf(G;\psi_\am(X), r(\psi_\am)_*D)=\cf(G;\ov X, r\ov D)$. Since $\psi$ is an MMP run of the first (second) kind, it follows by Lemma \ref{lem:ld_increasaes} that so is $\psi_\am$, too; a contradiction.

Thus $\tilde \psi_\am(\ll)$ is a normal crossing point of $(\tilde \psi_\am)_*\tilde D$. If it is not simple normal crossing then by Corollary \ref{cor:lc_sing_restrictions_r} $r=1$ and $\Exc \theta$ is a nodal rational curve. It follows that in any case $\tilde \ll$ is superfluous in $\tilde D^-+\tilde \ll$. Indeed, suppose otherwise. Then $\tilde D^-=E$ and $E+\tilde \ll$ is a rational cycle consisting of two components. Then $E^2\leq -5$ due to the negative definiteness of $E+\tilde \ll$. If $\tilde \ll\leq \tilde D$ then both $E$ and $\tilde \ll$ are log exceptional of the second kind, contrary to the assumption on $\psi_{\am}$. Thus $\tilde \ll\nleq \tilde D$ and $$\alpha(\ll)\cdot (K_{\alpha(X)}+\alpha_*D)=\alpha(\ll)\cdot K_{\alpha(X)}=\ll\cdot (K_X+(1+\tfrac{2}{E^2})E)=1+\tfrac{4}{E^2}>0;$$ a contradiction. This gives (2).

(3) Consider the case $\tilde \ll\leq \tilde D$. We have $$0\leq \ll\cdot (K_X+rD)\leq \tilde \ll\cdot (K_{\tilde X}+r\tilde D)=-1+r(\beta_{\tilde D}(\tilde\ll)-1),$$ so by (2) $\beta_{\tilde D}(\tilde\ll)=2$, $r=1$, $\ll$ is log exceptional of the second kind and $\pi$ is log crepant  in a neighborhood of $\ll$. 

(4) Consider the case $\tilde \ll\not \leq \tilde D$. By (2) we have now $$0\leq \ll\cdot (K_X+rD)\leq \tilde \ll\cdot (K_{\tilde X}+r\tilde D)=-1+r\tilde \ll\cdot \tilde D\leq 2r-1,$$ hence $r\geq \frac{1}{2}$. If the equality holds then we argue as in (3).
\end{proof}

\begin{prop}\label{prop:peeling_am-is-MMP}
 Let $(X,rD)$, with $D$ reduced and $r\in [0,1]\cap \Q$, be a log surface which is $(1-r)$-dlt (respectively, $(1-r)$-lc). Assume that $\psi$ is a partial MMP run on $(X,rD)$ of the first (respectively, second) kind. If $r<\frac{1}{2}$ or $\Exc \psi\leq D$ and $r\neq 1$ (respectively, if $r\leq \frac{1}{2}$ or $\Exc \psi\leq D$) then $\psi_{\am}$ is an MMP run of the first (respectively, second) kind on $(X,rD)$.
\end{prop}

\begin{proof}
Using Lemma \ref{lem:aMMP_properties}, by induction we may assume that $\psi=\ctr_{\alpha(\ll)}\circ \alpha$, where $\alpha$ is a pure partial peeling of the first (second) kind of $rD$ and $\ll\cdot K_X<0$. Now the proposition follows from Lemma \ref{lem:peeling_am-is-MMP}.
\end{proof}

\medskip
\subsection{Proof of Theorem \ref{thm:aMM_respects_(1-r)-dlt} and a discussion of assumptions}\label{ssec:proof_of_Thm}

The following examples show that in Theorem \ref{thm:aMM_respects_(1-r)-dlt}(1),(2) the assumption that $X$ is smooth or $r\leq \frac{1}{2}$ cannot be omitted, as otherwise the $(1-r)$-log terminality and $(1-r)$-log canonicity may be destroyed by an almost minimalization. In particular, unlike in Proposition \ref{prop:peeling_am-is-MMP}, $\psi_\am$ does not have to be an MMP run for $(X,rD)$.

\begin{ex}[Non-$(1-r)$-lc almost minimal models of $(X,rD)$ for $r\neq 1$]\label{ex:aMM_not_(1-r)dlt}
Let $a\geq 2$ be an integer and let $r\in (1-\frac{1}{3a-4},1]\cap \Q$. Consider a log smooth surface $(\tilde X,\tilde D)$, where $\tilde D=[3,1,3,a,0]$. Denote by $T_1, \ldots, T_5$ subsequent components of $\tilde D$ and by $\sigma\:\tilde X\to X$ the contraction of $T_3+T_4$. Put $D=\sigma_*(\tilde D-T_2)$, $\ll=\sigma(T_2)$ and $E=\sigma(T_1)$, see Figure \ref{fig:sing_collision}.
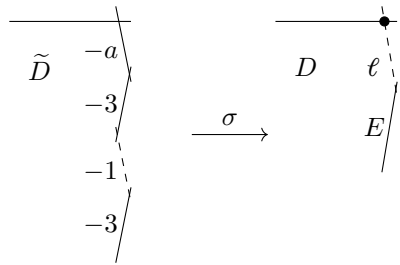
\begin{figure}[h]
\begin{tikzpicture}
	\begin{scope}
		\draw(-0.4,3) -- (1.2,3);
		\draw (1,3.2) -- (1.2,2.2);
		\node at (0.8,2.6) {\small{$-a$}};
		\draw (1.2,2.4) -- (1,1.4);
		\node at (0.8,1.9) {\small{$-3$}};
		\draw[dashed] (1,1.6) -- (1.2,0.6);
		\node at (0.8,1) {\small{$-1$}};
		\draw (1.2,0.8) -- (1,-0.2);
		\node at (0.8,0.3) {\small{$-3$}};
		\node at (0,2.4) {\small{$\tilde{D}$}};
		\draw[->] (2,1.5) -- (3,1.5);
		\node at (2.5,1.7) {\small{$\sigma$}};
	\end{scope}
	\begin{scope}[shift={(3.5,0)}]
		\draw[name path = A] (-0.4,3) -- (1.2,3);
		\draw[dashed, name path = B] (1,3.2) -- (1.2,2);
		\node at (0.9,2.4) {\small{$\ll$}};
		\draw (1.2,2.2) -- (1,1);
		\node at (0.9,1.6) {\small{$E$}};
		\node at (0,2.4) {\small{$D$}};
		\path [name intersections={of=A and B,by=C}];
		\filldraw (C) circle (0.06);
	\end{scope}
\end{tikzpicture}
\caption{The log surface $(X,D)$ in Example \ref{ex:aMM_not_(1-r)dlt}}\label{fig:sing_collision}
\end{figure}
Let $\alpha$ be the contraction of $E$ and $\psi\:X\to \ov X$ the contraction of $E+\ll$. We have $E\cdot (K_X+rD)=1-3r<0$. By \eqref{eq:Bk_rD} we have $\cf(T_3;X,rD)=\frac{2a-2+r}{3a-1}$, and  $\cf(T_4;X,rD)=\frac{3a-5+3r}{3a-1}$. Since $r>1-\frac{1}{3a-4}$, we infer that $(X,rD)$ is $(1-r)$-dlt and $\alpha$ is a pure partial peeling. Since $r>\frac{1}{2}$, we get $\ll\cdot (K_X+rD)=\frac{3ar-a-1}{3a-1}>0$. We compute $\alpha(\ll)\cdot (K_{\alpha(X)}+r\alpha_*D)=\frac{r-4/3}{3a-1}<0$. We have $\ctr_\ll(E)\cdot K_{\ctr_\ll(X)}\geq 0$, so $\psi_\am$ contracts only $\ll$. We conclude that:
\begin{enumerate}[(a)]
\item $(X,rD)$ is $(1-r)$-dlt,
\item $\ll$ is $\alpha$-almost log exceptional and not log exceptional on $(X,rD)$,
\item $(X',rD'):=\psi_\am(X,rD)$, the almost minimal model of $(X,rD)$, is not $(1-r)$-dlt; it is not $(1-r)$-lc, unless $r=1$ or $a=2$,
\item $\psi_\am$ is not an MMP run for $(X,rD)$.
\item $(X,rD)$ is $(1-r)$-dlt and $\ll$ is $\alpha$- almost log exceptional, but $(X,r(D+\ll))$ is not $(1-r)$-lc for $r\neq 1$ and not $(1-r)$-dlt for $r=1$.
\end{enumerate}
Part (d) is a consequence of (a) and (c). For (c) we compute $\cf(T_3;X',rD')=1-\frac{3}{2a-1}(1-r)\geq r$ and $\cf(T_4;X',rD')=1-\frac{(a+1)}{2a-1}(1-r)\geq r$, so $(X,rD)$ is not $(1-r)$-dlt and it is not even $(1-r)$-lc, unless $r=1$ or $a=2$. For (e) note that $\ll^2=-1+\frac{a}{3a-1}=\frac{1-2a}{3a-1}$, so $$\ll\cdot (K_X+r(D+\ll))=\tfrac{1}{3a-1}(3ar-a-1+r(1-2a))=\tfrac{(r-1)(a+1)}{3a-1}\leq 0,$$ so $\ll$ is log exceptional on $(X,r(D+\ll))$ of the first kind if $r\neq 1$ and of the second kind otherwise. It follows that if $(X,r(D+\ll))$ is $(1-r)$-lc for $r\neq 1$ or if it is $(1-r)$-dlt for $r=1$ then so is $(X',rD')$, too, contrary to part (c).
\end{ex}

\begin{ex}[Non-$(1-r)$-lc almost minimal models of $(X,rD)$ for $r=1$]\label{ex:aMM_not_(1-r)dlt_2}
Let $n\geq 2$ be an integer. Blow up over a chosen fiber of the $\P^1$-fibration of $\P^1\times \P^1$, so that its reduced total transform is a fork with branching component $B=[n+2]$ and three maximal twigs - $T_1=[3]$, $T_2=[2,2]$ and $T_3=L+E_1+\ldots+E_n+L'=[1,(2)_{n-1},3,1]$; see Figure \ref{fig:non-lc_aMM}. 
\begin{figure}[h]
	\begin{tikzpicture}
	\begin{scope} 		\draw (0,2.8) -- (0.2,1.6); 		\node at (0.3,2.4) {\small{$T_1$}}; 		\node[left] at (0.1,2.4) {\small{$-3$}}; 		\draw (0.2,1.8) -- (0,0.6); 		\node[left] at (0.1,1.2) {\small{$-n-2$}}; 		\node[left] at (0.15,0) {\small{$-2$}}; 		\node at (0.35,0.25) {\small{$T_2$}}; 		\draw (0,0.8) -- (0.2,-0.4); \draw (0,-0.3) -- (1.2,-0.1); \node at (0.9,0.05)  {\small{$-2$}};  	\draw[dashed] (0,1.2) -- (1.2,1); 		\node at (0.6,1.3) {\small{$-1$}}; 		\node at (0.5,0.9) {\small{$L$}};
	\draw (1,1) -- (2.2,1.2); 		\node at (1.5,1.3) {\small{$-2$}}; 		\node at (1.8,0.9) {\small{$E_1$}}; 		\node at (2.35,1.05) {$\dots$}; 		\draw (2.4,1.2) -- (3.6,1); 		\node at (3,1.3) {\small{$-2$}}; 		\node at (3,0.9) {\small{$E_{n-1}$}}; 		\draw (3.4,1) -- (4.6,1.2); 		\node at (3.9,1.3) {\small{$-3$}}; 		\node at (4.2,0.9) {\small{$E_n$}}; 		\draw[dashed] (4.4,1.2) -- (5.6,1); 		\node at (5,1.3) {\small{$-1$}}; 		\node at (4.9,0.9) {\small{$L'$}};  	\node at (2.8,-0.8) {$\widetilde{D}-\widetilde{F}$, $\rho(\widetilde{X})=n+7$}; 		\draw[->] (6,1.1) -- (7.4,1.1); 	\node at (6.7,1.3) {\small{$\pi$}};
	\end{scope}
	\begin{scope}[shift={(8.5,0)}] 		\draw[name path = T] (0.2,2.1) -- (-0.1,0.3); 		\node[left] at (0.2,1.8) {\small{$\pi(T_1)$}}; 		\draw[name path = L, dashed] (-0.2,1.2) -- (1.2,0.9); 		\path [name intersections={of=T and L,by=A}]; 		\filldraw (A) circle (0.07); 		\node at (0.4,0.85) {\small{$\ll$}};
		\draw[name path = L', dashed] (0.8,0.9) -- (2.2,1.2); 		\node at (1.8,0.85) {\small{$\ll'$}};
		\path [name intersections={of=L' and L,by=B}]; 		\filldraw (B) circle (0.07); 	\draw[->] (2.6,1.1) -- (4,1.1);
		\node at (3.3,1.3) {\small{$\sigma$}}; 		\node at (0.8,-0.8) {$D-F$, $\rho(X)=4$};
	\end{scope}
	\begin{scope}[shift={(13.5,0)}]
		\draw[name path = T] (0.2,2.1) -- (-0.1,0.3); 	\node[left] at (0.2,1.8) {\small{$\sigma\circ \pi(T_1)$}}; 		\draw[name path = L, dashed] (-0.2,1.2) -- (1.6,0.9); 		\path [name intersections={of=T and L,by=A}]; 		\filldraw (A) circle (0.07); 		\node at (0.8,0.7) {\small{$\sigma(\ll')$}}; 	\node at (0.5,-0.8) {$\rho(X')=3$};
	\end{scope}
	\end{tikzpicture}
	\caption{Example \ref{ex:aMM_not_(1-r)dlt_2}}	\label{fig:non-lc_aMM}
\end{figure}

Denote the resulting surface by $\tilde X$ and put $\tilde D=B+T_1+T_2+T_3-L-L'+\tilde F$, where $\tilde F$ is some smooth fiber. Let $\pi\:\tilde X\to X$ be the contraction of $B+T_2+T_3-L-L'$. Put $F=\pi(\tilde F)$, $D=\pi_*\tilde D=\pi(T_1)+F$ and $\ll=\pi(L)$. Then $(X,D)$ is dlt and $\rho(X)=4$. We have $d(B+T_2)=3n+4$ and 
$$\ll\cdot (K_X+D)=L\cdot (K_{\tilde X}+\tilde D-\Bk_{\tilde D}\Exc \pi)=1-\tfrac{2n}{2n+1}-\tfrac{1}{d(B+T_2)}>0.$$ Since $\pi(T_1)\cdot (K_X+D)<0$, the contraction of $\pi(T_1)$, which we denote by $\alpha$, is a peeling of $D$. We have $d(T_1+B+T_2)=9(n+1)$. We compute 
\begin{equation*}
\alpha(\ll)\cdot K_{\alpha(X)}=L\cdot (K_{\tilde X}+\tilde D-\Bk \tilde D)=\tfrac{1}{2n+1}-\tfrac{d(T_1)+d(T_2)}{d(T_1+B+T_2)}=\tfrac{1}{2n+1}-\tfrac{2}{3n+3}<0. 
\end{equation*}
Since $\tilde D-\tilde F+L$ is negative definite, we infer that $\ll$ is $\alpha$-almost log exceptional but not log exceptional. Let $\psi\:(X,D)\to (\ov X,\ov D)$ be the contraction of $\pi(T_1)+\ll$. Note that $\psi\circ \pi(T_1)$ necessarily a log terminal singularity (the exceptional divisor of the minimal resolution is a fork of type $(2;2,3,3)$). We have $\rho(\ov X)=2$ and $\ov D=\psi(F)$, which implies that $(\ov X,\ov D)$ is minimal, hence $\psi$ is an MMP. Let $\sigma\:(X,D)\to (X',D')$ be the contraction of $\ll$. We have $\ll\cdot K_X<0$ and $\sigma\circ\pi(T_1)\cdot K_{X'}\geq 0$, so $\sigma$ is the almost minimalization of $\psi$. The singular point of $X'$ is a canonical singularity with the exceptional divisor of the minimal resolution $[2,2,2,2]$. But the latter chain meets the proper transform of $D$ in the middle component, hence  $(X',D')$, the almost minimal model of $(X,D)$, is not log canonical.
\end{ex}

Theorem \ref{thm:aMM_respects_(1-r)-dlt} follows from the following result.

\begin{thm}\label{thm:aMM_respects_(1-r)-dlt_2}
Let $(X,D)$ be a log surface with a reduced boundary and let $r\in [0,1]\cap \Q$. Assume that $(X,rD)$ is $(1-r)$-lc and that $X$ is smooth or $r\leq \frac{1}{2}$. Let $\psi\:(X,rD)\to (\ov X,r\ov D)$ be a partial MMP run of the second kind and let $\psi_{\am}$ be the almost minimalization of $\psi$. Put $X'=\psi_{\am}(X)$ and $D'=(\psi_{\am})_*D$. Then the following hold.
\begin{enumerate}[(1)]
\item $(X',rD')$ is $(1-r)$-lc.
\item If $\psi$ is of the first kind and $(X,rD)$ is $(1-r)$-dlt then $(X',rD')$ is $(1-r)$-dlt.
\item If $X$ is smooth and $\frac{1}{r}\in \N\cup \{\8\}$ then $\psi_{\am}$ can be decomposed into elementary steps so that every intermediate model is $(1-r)$-lc.
\end{enumerate}
\end{thm}

\begin{proof}
We argue by induction on $\#\Exc \psi$. By Lemma \ref{lem:aMMP_properties} we may assume that $\psi=\ctr_{\alpha(\ll)}\circ \alpha$, where $K_X$ is $\alpha$-nef and $\ll\subseteq X$ is such that $\ll\cdot K_X<0$. By Corollary \ref{cor:nef_gives_peeling} $\alpha=\alpha_2\circ \alpha_1$, where $\alpha_1$ is a pure partial peeling of the first (second) kind and $\alpha_2$ is crepant and log crepant. Then $\alpha_1(\ll)$ is log exceptional on $(\alpha_1(X),(\alpha_1)_*D)$. By induction and by Corollary \ref{cor:reordering_MMP} we may assume that $\alpha_2=\id$, i.e.\ $\alpha$ is a pure partial peeling of the first (second) kind. Let $\pi\:\tilde X\to X$ be the minimal resolution of singularities and let $\tilde \alpha=\alpha\circ \pi$. Put $\tilde \ll=\pi_*^{-1}\ll$, $\tilde D=\pi_*^{-1}D+\Exc \pi$, $\tilde \psi:=\psi\circ \pi$ and $E=\Exc \tilde \alpha$. Let $q=\psi(\ll)$. We denote the contraction of $\ll$ by $\sigma$.  

(1), (2) Note that (1) holds if $\psi$ is log crepant. Indeed, then $\psi_\am$ is log crepant, and since $(X,rD)$ is $(1-r)$-lc, so is its image under $\psi_\am$, too. By Remark \ref{rem:aMM_2nd_kind} it follows that we may further assume that $\psi$ is of the first kind. Then $\alpha$ and $\alpha(\ll)$ are of the first kind. We need to show that if $(X,rD)$ is $(1-r)$-dlt (respectively, $(1-r)$-lc) then so is $(X',rD')$, too.

Clearly, we may assume that $E+\tilde \ll$ is connected, hence $\ov X\setminus \{q\}$ is smooth. We may assume that $\psi_{\am}$ is not an MMP run on $(X,rD)$, as otherwise the claim follows from Lemma \ref{lem:ld_increasaes}. In particular, $E\neq 0$. We are thus in the situation of Lemma \ref{lem:peeling_am-is-MMP}. Then $q\in \ov X$ is a singular point and $\tilde \ll$ is a superfluous component of $(\tilde D+\tilde \ll)\redd$.

Consider the case $\ll\leq D$. By Lemma \ref{lem:peeling_am-is-MMP}(3) $r=1$, $\ll$ is of the second kind and $\pi$ is log crepant  in a neighborhood of $\ll$. In particular, $\psi_{\am}$ is of the second kind. Thus we are in the situation when $\psi$ is of the first kind and $(X,D)$ is dlt, hence $\pi$ is the identity. It follows that $X$ is smooth in a neighborhood of $\ll$ and $\ll$ is superfluous in $D$, which implies that $(\sigma(X),\sigma_*D)$ is dlt. Then $\psi':=\psi\circ \sigma^{-1}$ is a partial MMP run on $(\sigma(X),\sigma_*D)$ of the first kind and the theorem follows in this case by induction on $\#\Exc \psi$.

We may thus assume that $\ll\nleq D$. Consequently, we may assume that $(X,r(D+\ll))$ is not $(1-r)$-dlt (respectively, not $(1-r)$-lc). Indeed, by Lemma \ref{lem:a.l.e._is_redundant}(2) $\psi$ is an MMP run on $(X,r(D+ \ll))$, so the theorem follows in this case from the case we have just settled.

Since $\tilde \ll$ is superfluous in $(\tilde D+\tilde \ll)\redd$, we get a contradiction if $X$ is smooth. By assumptions of the theorem and by Lemma \ref{lem:peeling_am-is-MMP}(4) we are left with the case $r=\frac{1}{2}$ with $\psi$ of the first kind and $(X,\frac{1}{2}D)$ being $\frac{1}{2}$-log terminal. Again, since $\pi$ is log crepant  in a neighborhood of $\ll$, it is in fact the identity, so $X$ is smooth in the neighborhood of $\ll$. It follows that the image of $(X,\frac{1}{2}D)$ remains $\frac{1}{2}$-log terminal after the contraction of $\ll$, so the theorem follows again from the case when $\ll\leq  D$.

(3) For $r=0$ the claim is clear, so assume that $X$ is smooth and $\frac{1}{r}\in \N$. By induction it is sufficient to argue that $(\sigma(X),r\sigma_*D)$ is $(1-r)$-lc. By Corollary \ref{lem:eps-lc_is_respected} we may assume that $\ll$ is not log exceptional of the first or second kind. In particular, $r\neq 0$. Also, if $\alpha$ is log crepant then we are done. By Remark \ref{rem:aMM_2nd_kind} we may therefore assume that $\alpha$ is of the first kind. By Lemma \ref{cor:singularities_after_squeezing}(2) we may also assume that $\ll\nleq D$. Since $\alpha$ is of the first kind, we are in case (1), (6), (8), or (10)-(12) of  Proposition \ref{prop:ful_description_almost_log_exc}. In case (1) the claim is clear. In cases (8) and (10)-(12) we have $r=\frac{1}{2}$, $\ll$ meets $D$ normally and $\cf(\ll;\sigma(X),r\sigma_*D)=\frac{1}{2}$, so we are done. 

We are left with case (6) with $\frac{1}{r}<\ll\cdot D\leq \frac{1}{r}+s\leq \frac{1}{r}+1$ and $0<r<1$. Since $\frac{1}{r}$ is an integer, we get $s=1$ and $\ll\cdot D=\frac{1}{r}+1$. Then $\cf(\ll;\sigma(X),r\sigma_*D)=r$. This time $\ll$ can meet $D$ non-normally. To show that $(\sigma(X),r\sigma_*D)$ is $(1-r)$-lc, it remains to argue that for every exceptional prime divisor $U$ over $X$ we have $\cf(U;X,r(D+\ll))\leq r$. Let $Y\to X$ be a minimal log resolution such that $U$ is a divisor on $Y$ and let $\tau$ be the contraction of $U$. Since $(X,rD)$ is $(1-r)$-lc, it is sufficient to consider the case when the image of $U$ on $X$ is contained in $\ll\cap R$. Denote the reduced total transform of $D+\ll$ on $Y$ by $B$ and put $Y'=\tau(Y)$ and $B'=\tau_*B$. We have $\tau^*(K_{Y'}+rB')=K_Y+r\tau_*^{-1}B'+u\cdot U$ with $u=\cf(U;X,r(D+\ll))=r \tau_*^{-1}B'\cdot U-1=r\mult_{\tau(U)}(\tau_*B)-1$, where $\mult_q(\tau_*B)$ denotes the multiplicity of the divisor $\tau_*B$ at the point $q$. Put $R=D-E$. Since $r=\frac{1}{\ll\cdot R}$, we see that $u\leq r$ if and only if $\mult_{\tau(U)}(\tau_*B)\leq \ll\cdot R+1$. We are therefore done by Lemma \ref{lem:multiplicitity_lemma} applied to $R+\ll$.
\end{proof}

We denote the multiplicity of a divisor $D$ at $q$ by $\mult_q(D)$ and the local intersection index of divisors $D$ and $B$ at a point $q$ by $D\cdot_q B$. 

\begin{lem}\label{lem:multiplicitity_lemma}
Let $D$ be a reduced divisor on a smooth projective surface $X$. Let $\tau\:(Y,D^Y)\to (X,D)$ be a proper birational morphism from a smooth surface, where $D^Y$ is the reduced total transform of $D$. If $D$ has a smooth component $L$ then for every $q\in Y$ over $L\cap (D-L)$ we have
\begin{equation}
\mult_q(D^Y)\leq L\cdot_{\tau(q)}(D-L)+1.
\end{equation}

\end{lem}

\begin{proof}
We argue by induction with respect to $\#\Exc \tau$.  If $\tau=\id$ then, since $L$ is smooth at $q$, we have $\mult_q(D-L)\leq L\cdot_{q}(D-L)$, so we are done. Assume that $\#\Exc \tau\neq 0$ and write $\tau=\sigma\circ\tau'$, where $\sigma\:Z\to X$ is a blowup. Put $q'=\tau'(q)$ and $U=\Exc \sigma$. Write $D^Z=D'+L'$, where $L'$ is the proper transform of $L$. By assumption $L\cdot_{\tau(q)}(D-L)\geq 1$, so to prove the above inequality we may assume that $q$ is not a normal crossing point of $D^Y$ and hence we may assume that all exceptional curves of $\tau'$ lie over $D'-U$. In particular, $q'\in D'-U$ and we have $D^Y=D'^Y+L'$. Then $\mult_q(D^Y)=\mult_q(D'^Y)+[q\in L']$, where by definition $[q\in L']$ equals $1$ if $q\in L'$ and equals $0$ otherwise. By induction we obtain $$\mult_q(D^Y)\leq U\cdot_ {q'}(D'-U)+1+[q\in L'].$$ The total transform of $L$ on $Z$ is $L'+U$ and the proper transform of $D-L$ is $D'-U$, so by the projection formula $$L\cdot_{\tau(q)} (D-L)=\sum_{s\in U} (L'+U)\cdot_s (D'-U)\geq (L'+U)\cdot_{q'} (D'-U),$$ which gives $\mult_q(D^Y)\leq L\cdot_{\tau(q)}(D-L)+1+[q\in L']-L'\cdot_{q'} (D'-U)$. But since $q'\in D'-U$, we have $[q\in L']\leq L'\cdot_{q'} (D'-U)$, which completes the proof.
\end{proof}

\begin{rem} Assume that a log surface $(X,rD)$, with $r\in [0,1]\cap \Q$ and $D$ reduced is $(1-r)$-dlt (respectively, $(1-r)$-lc). Decompose an almost minimalization of $(X,rD)$ as in Lemma \ref{lem:effective_almost_minimalization}. If $X$ is smooth or $r\leq \frac{1}{2}$ then Theorem \ref{thm:aMM_respects_(1-r)-dlt_2} implies that each of the intermediate models $(X_i,D_i)$, $i\geq 1$ in that lemma is $(1-r)$-dlt (respectively, $(1-r)$-lc).
\end{rem}

The following two examples show that in Theorem \ref{thm:aMM_respects_(1-r)-dlt}(3) the assumption $\frac{1}{r}\in \N\cup \{\8\}$ cannot be omitted, as otherwise the intermediate models may be non-$(1-r)$-lc.

\begin{ex}[Almost minimalization of $(X,rD)$ with non-$(1-r)$-lc intermediate models; $r> \frac{1}{2}$]\label{ex:optimal_ass_1}
 Let $\ov D\subseteq \ov X$ be a cuspidal cubic on $\P^2$ with the cusp at $p\in \ov D$ and let $\psi\:(X,D)\to (\ov X,\ov D)$ be the minimal log resolution, with $D$  being the reduced total transform of $\ov D$. We have $\Exc \psi=E_1+\ll+E_2=[2,1,3]$ and $(D-\Exc \psi)\cdot \ll =1$. Assume that $r\in [\frac{1}{2},\frac{4}{5}]$, as in case (6) of Proposition \ref{prop:ful_description_redundant}, see Fig.\ \ref{fig:bad_intermediate_model}.

\begin{figure}[h]
	\begin{tikzpicture}
		\path[use as bounding box] (0,-1) rectangle (2.6,1.2);
		\draw[dashed] (0,1) -- (2.6,1);
		\draw[thick] (0.2,1.2) -- (0.2,-0.8);
		\draw (1.2,1.2) -- (1.2,0);
		\node[left] at (1.2,0.2) {\small{$-2$}};
		\draw (2.4,1.2) -- (2.4,0);
		\node[left] at (2.4,0.2) {\small{$-3$}};
	\end{tikzpicture}
\caption{Example \ref{ex:optimal_ass_1}. Thick line is $D-\Exc \psi$, dashed line is $\ll$; $\frac{1}{2}\leq r\leq \frac{4}{5}$.}\label{fig:bad_intermediate_model}
\end{figure}
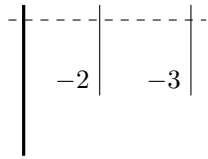
Let $\alpha$ be the contraction of $E_1+E_2$. It is a pure peeling of the first kind if $r\neq \frac{1}{2}$ and of the second kind otherwise. Since $\alpha(\ll)\cdot \alpha_*(K_X+rD)=\frac{5}{6}(r-\frac{4}{5})$, $\ll$ is $\alpha$-redundant of the first kind if $r\neq \frac{4}{5}$ and of the second kind otherwise. Thus $\psi$ is an MMP run (complete, as $\rho(\P^2)=1$) of the first kind if $r\in (\frac{1}{2},\frac{4}{5})$ and of the second kind if $r=\frac{1}{2}$ or $r=\frac{4}{5}$. The almost minimalization is by definition a $K_X$-MMP over $\ov X$. The unique decomposition of $\psi_{\am}$ as a $K_X$-MMP is $\psi_{\am}=\sigma_3\circ \sigma_2\circ \sigma_1$, where $\sigma_1$ contracts $\ll$, $\sigma_2$ contracts $\sigma_1(E_1)$ and $\sigma_3$ contracts $\sigma_2\circ\sigma_1(E_2)$. Since $p\in \ov X$ is smooth, $\psi$ is a minimalization and (its own) almost minimalization at the same time, that is, $\psi_{\am}=\psi$, and hence the same holds for $\sigma_3\circ \sigma_2$ and for $\sigma_1$. The intermediate models in the process of almost minimalization are $(X_i,rD_i)=(\sigma_i(X_{i-1}),r(\sigma_i)_*D_{i-1})$, $i=1,2,3$, where $(X_0,D_0)=(X,D)$ and $(X_3,D_3)=(\ov X,\ov D)$. We compute
\begin{equation}
 \cf(\ll;X_3,rD_3)=6r-4<\cf(\ll;X_2,rD_2)=4r-2<\cf(\ll;X_1,rD_1)=3r-1,
\end{equation}
\begin{equation}
\cf(E_1;X_3,rD_3)=3r-2<\cf(E_1;X_2,rD_2)=2r-1
\end{equation}
and
\begin{equation}
\cf(E_2;X_3,rD_3)=2r-1
\end{equation}
Since $\psi$ is an MMP run of the second kind, $(\ov X,r\ov D)$, which is a minimal and at the same time an almost minimal model, is $(1-r)$-lc, which is visible explicitly from the above computation, as $2r-1<r$, $3r-2<r$ and $6r-4\leq r$. However, in general the intermediate models in the process of almost minimalization are not $(1-r)$-lc. Indeed, we have $\cf(\ll;X_1,rD_1)=3r-1\geq r$, as $r\geq \frac{1}{2}$, so $(X_1,rD_1)$ is not $(1-r)$-lc for $r\neq \frac{1}{2}$ and not $(1-r)$-dlt for $r=\frac{1}{2}$.
\end{ex}

\begin{rem}[Almost minimalization of $(X,rD)$ with non-log canonical intermediate models] We note that in the (unique) process of almost minimalization of the $(1-r)$-lc surface $(X,rD)$ in Example \ref{ex:optimal_ass_1} the intermediate models may be even non-lc. Indeed, we have $\cf(E_1;X_1,rD_1)=3r-1>1$ for $r>\frac{2}{3}$, so this happens for all $r\in  (\frac{2}{3},\frac{4}{5}]$. 
\end{rem}

\begin{ex}[Almost minimalization of $(X,rD)$ with non-$(1-r)$-lc intermediate models; $r\leq \frac{2}{3}$]\label{ex:optimal_ass_2}

Let $m\geq 2$ and $k\geq 2$ be integers. Pick $m+1$ fibers $F_1,\ldots, F_m, F_\8$ of the $\P^1$-fibration of $\F_1=\Proj(\cO_{\P^1}\oplus \cO_{\P^1}(1))$. Let $\ov \ll$ be the $(-1)$-section. Blow up over $F_\8\setminus \ll$ so that the reduced total transform of $F_\8$ is the chain $M+E_1+\ldots+E_{k-1}+M'=[1,(2)_{k-1},1]$.  Let $X$ be the resulting surface. Let $R$ and $\ll$ be the proper transforms of $F_1+\ldots+F_m$ and $\ov \ll$.
Put $D=R+\ll+E$, where $E=E_1+\ldots+E_{k-1}$, see Figure \ref{fig:bad_intermediate_model_2}. Then $\rho(X)=2+k$ and $\ll\cdot R=m$. 
\begin{figure}[h]
\begin{tikzpicture}[scale=1.4]
			\draw[dashed] (0,3) -- (1.6,3);
			\draw[thick] (0,2.1) -- (0,3.2);
			\draw[thick] (0.2,2.1) -- (0.2,3.2);
			\node[left] at (0.8,2.6) {$\cdots$};
			\draw[thick] (0.8,2.1) -- (0.8,3.2);
			\draw(1.4,3.2) -- (1.6,2.2);
			\node[left] at (1.6,2.2) {$\vdots$};
			\node[right] at (1.5,2.1) {\small{\ $[(2)_{k-1}]$}};
			\draw (1.6,2) -- (1.4,1);
\end{tikzpicture}
\caption{Example \ref{ex:optimal_ass_2}. Thick line is $R$, dashed line is $\ll$; $\frac{1}{\ll\cdot R}< r\leq \frac{1}{\ll\cdot R-1/k}$.}\label{fig:bad_intermediate_model_2}
\end{figure}
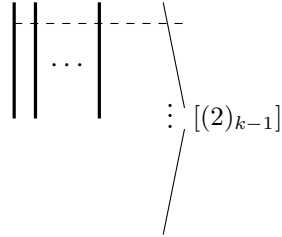

Let $\psi\:X\to \ov X$ be the contraction of $\ll+E$, let $\ov D=\psi_*D$ and let $\alpha$ be the contraction of $E$. Then $\ov X$ is smooth and $\rho(\ov X)=2$. Since $(X,D)$ is log smooth, $(X,rD)$ is $(1-r)$-dlt for every $r\in [0,1]\cap \Q$. Assume that, as in case (5) of Proposition \ref{prop:ful_description_redundant},  $\frac{1}{\ll\cdot R}< r\leq \frac{1}{\ll\cdot R-1/k}$. With this choice of $r$ we check that $\ll$ is not log exceptional on $(X,rD)$, but it is $\alpha$-redundant. Since $\ov X$ is smooth, we have $\psi_{\am}=\psi$. Moreover, the log surface $(\ov X,r\ov D)$ is now minimal. Indeed, suppose that there exists a log exceptional curve $L$ on it. Since components of $\ov D$ have positive self-intersection numbers, $L\nleq \ov D$. The inequality $L\cdot (K_{\ov X}+r\ov D)<0$ implies that $L$ is a $(-1)$-curve with $L\cdot \psi_*F_1=\frac{1}{m}L\cdot \ov D<\frac{1}{rm}<1$. After the contraction of $L$ components of $\ov D$ become lines on $\P^2$ intersecting each other with multiplicity $k\geq 2$; a contradiction. Thus $(\ov X, r\ov D)$ is minimal and $\psi$ is an MMP run.

Let $\sigma\:X\to X'$ be the contraction of $\ll$ and let $D'=\sigma_*D$. The log surface $(X',rD')$ is an intermediate model in the process of almost minimalization of $(X,rD)$. We have $\cf(\ll;X',rD')=r+r\ll\cdot R-1> r$, so  this model is not $(1-r)$-lc. Still, since $\ll\cdot R\geq 2$, we have $\frac{1}{\ll\cdot R-1/k}< \frac{2}{\ll\cdot R+1}$, hence $\cf(\ll;X',rD')=r+r\ll\cdot R-1< 1$, so $(X',rD')$ is dlt. 
\end{ex}

\medskip
\section{Uniform boundaries with coefficient $r\leq \frac{1}{2}$}\label{ssec:half}

Recall that exceptional divisors of minimal resolutions of canonical (du Val) singularities are $(-2)$-chains and admissible $(-2)$-forks. Their Dynkin diagrams are $\rA_k$ for $k\geq 1$, $\rD_k$ for $k\geq 1$ and $\rE_k$ for $k=6,7,8$. We say that a singular point is of \emph{type $\Ast_{k}$, $k\geq 1$} if the exceptional divisor of the minimal resolution is $E=[3,(2)_{k-1}]$  for some $k\geq 1$. If $E$ is a connected component of $D$ we say that it is an $\Ast_{k}$-rod of $D$. A singularity is of \emph{type $\Ast$} if it is of type $\Ast_k$ for some $k\geq 1$.  If the base field has characteristic zero then an $\Ast_{k}$-singularity is exactly a Hirzebruch--Jung surface singularity $\tfrac{1}{2k+1}(1,k)$, that is, a germ of $\{z^{2k+1}=xy^{k}\}$, see \cite[III.5]{BHPV_complex_surfaces}. For the definition of the coefficient of $(X,D)$ see Section \ref{ssec:discrepancies}.

\begin{lem}[Log surface germs with $\cf\leq \frac{1}{2}$]\label{lem:cf<=1/2}
Let $(\ov X,\ov D)$ be a germ of a log surface. Let $E$ be the exceptional divisor of the minimal resolution of singularity $\pi\:X\to \ov X$. Put $D=\pi_*^{-1}D+E$. 
Then the following hold:
\begin{enumerate}[(1)]
\item $\cf(\ov X,\frac{1}{2}\ov D)<\frac{1}{2}$ if and only if 
one of the following holds:
	\begin{enumerate}[(a)]
		\item $\ov D=0$ and $E$ is a $(-2)$-fork or a $(-2)$-rod of $D$ or a rod of $D$ with $E=[3,2,\ldots,2]$,
		\item $\ov D\neq 0$ and $E$ is a $(-2)$-twig of $D$.
		\end{enumerate}
\item $\cf(\ov X,\frac{1}{2}\ov D)=\frac{1}{2}$ if and only if one of the following holds:
	\begin{enumerate}[(a)]
	\item $\ov D=0$ and $E=[4]$, $E=[3,2,\ldots,2,3]$, $E=[2,3,2]$ or $E=\langle [2];[2],[2],[3,2,\ldots,2]\rangle $,
	\item $\ov D\neq 0$ and $E$ is a $(-2)$-segment of $D$ or a twig of $D$ with $E=[3,2,\ldots,2]$.
	\end{enumerate}
\end{enumerate}
\end{lem}
\begin{proof} Assume first that $\ov D=0$. We check that for $E$ is as in (2a) we have $\cf(\ov X,0)=\frac{1}{2}$. Then we infer from Lemma \ref{lem:Alexeev_dlt} that if $\cf(E;\ov X, \frac{1}{2}\ov D)\leq \frac{1}{2}$ then the associated weighted decorated graph of the germ $(\ov X, \frac{1}{2}\ov D)$ does not have a proper subgraph associated with the listed exceptional divisors. In particular, all components have self-intersection numbers not lower than $(-3)$ and there is at most one equal to $(-3)$, necessarily a tip of $D$. It follows easily that the list is complete and that (1a) holds. Assume that $\ov D\neq 0$ and that $E$ is not a $(-2)$-segment of $D$. By Corollary \ref{cor:lc_sing_restrictions_r}(2) $E$ is a twig of $D$. It remains to show that if $E$ contains a curve $E_0$ which is not a $(-2)$-curve then $\cf(E_0;\ov X,r\ov D)\geq\frac{1}{2}$ and the equality holds if only if $E=[3,2,\ldots,2]$. By Lemma \ref{lem:Alexeev_dlt} we may assume that $E=[3,2,\ldots,2]$. Then $\cf(E_0;\ov X,\frac{1}{2}\ov D)=\frac{1}{2}$ by \eqref{eq:Bk_rD}.
\end{proof}

The coefficient of an exceptional prime divisor $U\leq E$ over $(\ov X,r\ov D)$ is its coefficient in $\cf_X(\ov X,r\ov D)$, where by definition $K_X+r\pi_*^{-1}\ov D+\cf_X(\ov X,r\ov D)\sim \pi^*(K_{\ov X}+r\ov D).$ By Lemma \ref{lem:Bk=ld(pi)} if $\pi$ contracts rods, twigs, forks and $(-2)$-segments of $D$ then 
\begin{equation}
\cf_X(\ov X,r\ov D)=E-\Bk_D(E)-(1-r)\Bk\trp T,
\end{equation}
where $T$ is the sum of twigs which are not rods.  Here is a complete computation of coefficients for the above graphs using the above equation and the explicit formulas computing the barks like \eqref{eq:bark}.

\begin{lem}[Formulas for coefficients, $r\leq \frac{1}{2}$] \label{lem:cf_1/2_computations}
Let the notation be as in Lemma \ref{lem:cf<=1/2} and let $r\in [0,\frac{1}{2}]\cap \Q$. The following formulas for the coefficients over $(\ov X,r\ov D)$ hold.
\begin{enumerate}[(1)]
\item[(1a)] If $E$ is a $(-2)$-fork or a $(-2)$-rod of $D$ then $\cf_X=0$.
\item[(1b)] If  $E$ is a rod of $D$ and $E=E_1+\ldots+E_k=[3,(2)_{k-1}]$ then $\cf(E_i)=\frac{k+1-i}{2k+1}$, $1\leq i\leq k$.
\item[(1c)] If $E=E_1+\ldots+E_k$ is a $(-2)$-twig of $D$ then $\cf(E_i)=\frac{ir}{k+1}$, $1\leq i\leq k$. 
\item[(2a)] If $E$ is a rod of $D$ and $E=[4]$ or $E=[3,2,\ldots,2,3]$ then $\cf_X=\frac{1}{2}E$.
\item[(2b)] If  $E$ is a rod or a fork of $D$ and $E=[2,3,2]$ or $E=\langle [2];[2],[2],[3,2,\ldots,2]\rangle $ then $\cf_X=\frac{1}{2}E-\frac{1}{4}T_0$, where $T_0$ is the sum of $(-2)$-tips of $D$ contained in $E$.
\item[(2c)] If $E$ is $(-2)$-segment of $D$ then $\cf_X=rE$.
\item[(2d)] If $E$ is a twig of $D$ and $E=[3,2,\ldots, 2]$ then $\cf_X\geq rE$ and the equality holds only for $r=\frac{1}{2}$.
\end{enumerate}
\end{lem}

\begin{cor}[Peeling for $(X,rD)$ with $r\leq \frac{1}{2}$]\label{cor:peel-squeeze_r=half}
Let $X$ be a smooth projective surface and $D$ a reduced divisor on $X$. Let $r\in [0,\frac{1}{2}]\cap \Q$ and let $\alpha\:(X,rD)\to (\ov X,r\ov D)$ be a birational contraction. Then the following hold.
\begin{enumerate}[(1)]
\item If $r=\frac{1}{2}$ then $\alpha$ is a pure partial peeling of the first (second) kind if and only if every connected component of $\Exc \alpha$ is as in part (1) (respectively, part (1) or (2)) of Lemma \ref{lem:cf<=1/2}.
\item If $r<\frac{1}{2}$ then $\alpha$ is a pure partial peeling of the first kind if and only if every connected component of $\Exc \alpha$ is as in Lemma \ref{lem:cf_1/2_computations} part (1a), (1c) or (1b) with $k<\frac{r}{1-2r}$.
\item If $r<\frac{1}{2}$ then $\alpha$ is a pure partial peeling of the second kind if and only if every connected component of $\Exc \alpha$ is as in Lemma \ref{lem:cf_1/2_computations} part (1a), (1c), (2c) or (1b) with $k\leq \frac{r}{1-2r}$. 
\item A curve $L\leq D$ is $\alpha$-redundant of the first or second kind for some pure partial peeling $\alpha$ if and only if it is a $(-1)$-curve meeting at most one $(-2)$-twig $\Delta$ contracted by $\alpha$ and such that 
\begin{equation}\label{eq:Lambda_k-bound}
\beta_{D-\Delta}(L)\leq  \frac{1}{r}+\frac{1}{d(\Delta)},
\end{equation}
where the equality holds exactly when $L$ is of the second kind. 
\end{enumerate}
\end{cor}

\begin{proof} (1)-(3). We may assume that $E=\Exc \alpha$ is connected and nonzero. If $\alpha$ is a peeling of the second kind then for every component $U$ of $E$ we have $\cf(U;\ov X,r\ov D)\leq \cf(U;X,rD)=r$ by Lemma \ref{lem:ld_increasaes} and the inequality is strict if $\alpha$ is of the first kind. Note that in case (1b) the maximal coefficient is $\cf(E_1)=\frac{k}{k+1}$ and $\cf(E_1)\leq r$ if and only if $k\leq \frac{r}{1-2r}$. Hence we are done by the above computations.  The converse implication follows from Corollary \ref{cor:improving_ld_gives_MMP}.

(4) Let $\alpha\:X\to \ov X$ be a pure partial peeling (of the first kind) and $L\leq D$ be such that $\ov L:=\alpha(L)$ is log exceptional of the first or second kind. By Lemma \ref{lem:a.l.e._is_redundant}, $L$ is a $(-1)$-curve. By parts (1) and (2) we may assume that $\Delta:=\Exc \alpha$ is a sum of $(-2)$-twigs meeting $L$. Since $\Delta+L$ is negative definite, $\Delta$ is a single $(-2)$-twig meeting $L$, possibly zero. We compute 
$\ov L\cdot (K_{\ov X}+r\ov D)= -1+r(\beta_{D-\Delta}(L)-\frac{1}{d(\Delta)})$, which gives (4).
\end{proof}

Note that a partial squeezing of $(X,rD)$ for $r\leq \frac{1}{2}$ is a composition of contractions of $(-1)$-curves as in Corollary \ref{cor:peel-squeeze_r=half}(4).

\begin{rem}
We note that the condition $k<\frac{r}{1-2r}$ is empty for $r=\frac{1}{2}$ and gives $k=0$ for $r\leq \frac{1}{3}$. Also, while for $r=1$ peeling contracts all admissible twigs, for $r\leq \frac{1}{2}$ only $(-2)$-twigs are contracted. On the other hand, while an snc-minimal divisor $D$ is automatically squeezed, it is not so for $r\leq \frac{1}{2}$. In fact, the smaller $r$ is the more $(-1)$-curves in $D$ are contracted by a squeezing morphism for $(X,rD)$.
\end{rem}

\begin{notation}\label{not:Delta_Gamma_Lambda}
Assume that $X$ is smooth, $D$ is reduced and $0<r\leq \frac{1}{2}$. We define: 
\begin{enumerate}[(1)]
	\item\label{item:not_Phi} $\Gamma$ as the sum of all $(-2)$-rods and admissible $(-2)$-forks of $D$,
	\item\label{item:not_Omega} $\Lambda$ as the sum of all $\Ast_k$-rods of $D$ with $k<\frac{r}{1-2r}$. 
	\item\label{item:not_Delta} $\Delta$ as the sum of maximal $(-2)$-twigs in $D-\Gamma-\Lambda$, 
\end{enumerate}
\end{notation}

For a connected component of $\Lambda$, say $\Lambda_0=[3,(2)_{k-1}]$ we fix the order of components as written, that is, $[3]$ is the first component.

\begin{cor}\label{cor:squeezing_r=half}
Assume that $X$ is smooth, $D$ is reduced and $0<r\leq \frac{1}{2}$. Let $\alpha\:(X,rD)\to (\ov X,r\ov D)$ be a (unique) maximal pure partial peeling. Then:
\begin{enumerate}[(1)]
\item $\Exc \alpha=\Gamma+\Lambda+\Delta$,
\item $\cf_X(\ov X,r\ov D)=\Bk\trp \Lambda+r\Bk\trp \Delta$.
\end{enumerate}
\end{cor}

\begin{proof}Part (1) follows from Corollary \ref{cor:peel-squeeze_r=half} and part (2) from Lemma \ref{lem:cf_1/2_computations}.
\end{proof}

\begin{rem}[Multiplicity for $(1-r)$-dlt singularities] Assume that $(\ov X,r\ov D)$, with $\ov D$ reduced, is $(1-r)$-dlt (respectively, $(1-r)$-lc) and $p\in \ov X$ is smooth. Then
\begin{equation}\label{eq:mult_for_eps-dlt}
\mult_p(\ov D)<1+\frac{1}{r} \quad (\text{respectively, \ }\mult_p(\ov D)\leq 1+\frac{1}{r})
\end{equation}
If $\frac{1}{r}$ is an integer then for every infinitely near point $q$ of $p$ we have
\begin{equation}\label{eq:mult_higher_for_eps-dlt}
\mult_q(D)\leq 1+\frac{1}{r},
\end{equation} 
where $D$ is the reduced total transform of $\ov D$. 
\end{rem}

\begin{proof} Let $\sigma$ be a blowup at $p$ with exceptional divisor $E$. Then $$K_X+rD-\sigma^*(K_{\ov X}+r\ov D)=(1-rE\cdot D)E=-r(\mult_p(\ov D)-1-\tfrac{1}{r})E,$$ which gives \eqref{eq:mult_for_eps-dlt}. Assume that $\frac{1}{r}$ is an integer. Let $\alpha\:(X,D)\to (X,D)$ be a minimal log resolution. By Lemma \ref{lem:resolution_is_peeling} $\alpha$ is an peeling of the second kind. Since $p\in \ov X$ is smooth, it is also a squeezing of the second kind. Then by Corollary \ref{cor:singularities_after_squeezing}(2) each step of $\alpha$ considered as squeezing leads to an $(1-r)$-lc log surface. Thus the above computation gives \eqref{eq:mult_higher_for_eps-dlt}.
\end{proof}

If $\frac{1}{r}$ is not an integer then the condition \eqref{eq:mult_higher_for_eps-dlt} may fail for infinitely near points of $p$. 

\begin{ex}
Let $(X,D)\to (\P^2,D)$ be the minimal log resolution of a cuspidal planar cubic. The log surface $(\ov X,r\ov D)$ is $(1-r)$-lc for $r\leq \frac{4}{5}$, see Example \ref{ex:optimal_ass_1}. Since one of the the infinitely near points of the cusp has multiplicity $3$, the inequality  \eqref{eq:mult_higher_for_eps-dlt} fails for $r\in (\frac{1}{2},\frac{4}{5}]$.
\end{ex}

We now describe almost log exceptional curves for $r=\frac{1}{2}$. A similar characterization is possible for smaller $r$, but the number of cases grows as $r$ decreases. Since on squeezed log surfaces of type $(X,rD)$ the peeling morphism is unique by Corollary \ref{cor:peeling_unique_for_squeezed}, we may and will speak about \emph{almost log exceptional curves} meaning that they are almost log exceptional with respect to this unique peeling. 

\begin{lem}[Almost log exceptional curves for $r=\frac{1}{2}$]\label{lem:a.l.e.1_r=1/2}
Let $X$ be a smooth surface and $D$ a reduced divisor. Let $E$ be the exceptional divisor of a (unique) maximal pure partial peeling of $(X,\frac{1}{2}D)$. We have $E=\Gamma+\Lambda+\Delta$ (see Notation \ref{not:Delta_Gamma_Lambda} and Corollary \ref{cor:squeezing_r=half}). A curve $A\not\leq D$ is almost log exceptional of the first kind on $(X,\frac{1}{2}D)$ if and only if it is a $(-1)$-curve such that one of the following holds (see Fig.\ \ref{fig:lem:a.l.e.1_r=1/2-first}): 
\begin{enumerate}[(1)]
\item $A\cdot D\leq 1$. If $A\cdot T=1$ for some component $T\leq E$ then $T$ is a tip of $\Delta$ (not necessarily of $D$) or of a rod of $\Gamma+\Lambda$ or it is the middle component of $[2,2,3]$ - a connected component of $\Lambda$.
\item $A\cdot D=2$, and $A$ meets two different components $T_{1}$, $T_{2}$ of $D$, such that
	\begin{enumerate}[(a)]
	\item \label{item:ale_C*_R} $T_{1}\leq D-E$ and $T_{2}$ is a tip of $\Delta$, of $\Lambda$, or of a $(-2)$-rod of $\Gamma$,
	\item \label{item:ale_C*_OmegaR} $T_{1}\leq D-E$  and $T_{2}$ is the middle curve of a connected component $[2,2,3]$ of $\Lambda$,
	\item \label{item:ale_C*_Omega} $T_{1}$, $T_{2}$ are $(-3)$-curves in $\Lambda$,
	\item \label{item:ale_C*_Delta-Omega} $T_{1}$ is a $(-3)$-curve in $\Lambda$, and $T_{2}=[2]$ is a connected component of $\Delta$ or $\Gamma$. 
	\end{enumerate}
\item $A\cdot D=3$ and $A\cdot (D-E)=1$, $A$ meets a connected component $[2]$ of $\Gamma$ and a $(-3)$-curve in $\Lambda$.
\end{enumerate}
A curve $A\nleq D$ is almost log exceptional of the second kind on $(X,\frac{1}{2}D)$ if and only if it is a $(-1)$-curve such that (see Fig.\ \ref{fig:lem:a.l.e.1_r=1/2-second})
\begin{enumerate}[(1)]
\item[(4)] $A\cdot D=2$ and $A\cdot E=0$ or
\item[(5)] $A\cdot D=3$ and $A$ meets $E$ once, in a tip of a rod in $\Gamma$.
\end{enumerate}
\end{lem}

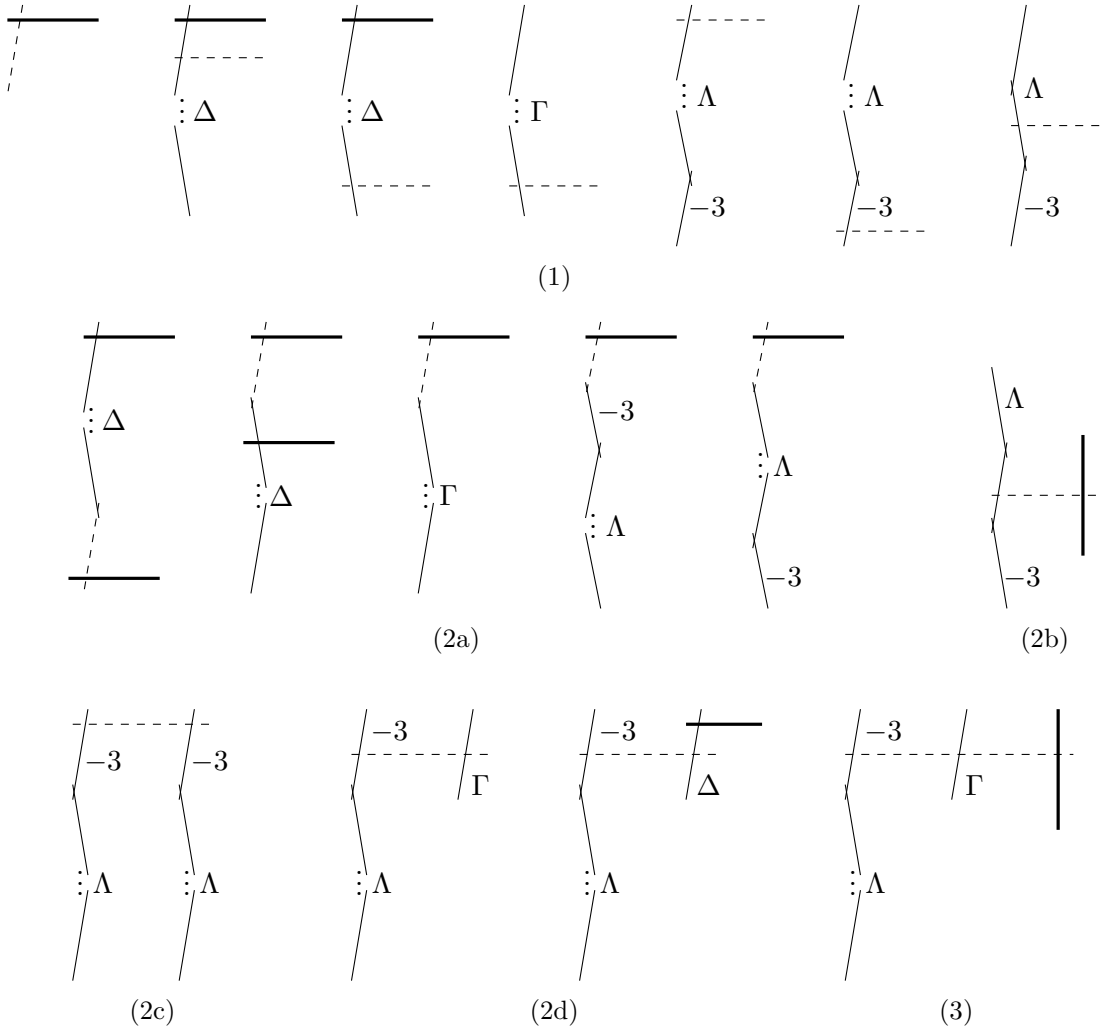
\begin{figure}[!h]
\subcaptionbox*{(1)}[\textwidth]{
\begin{tikzpicture}
	\begin{scope} 	\draw[thick] (0,3) -- (1.2,3); 	\draw[dashed] (0.2,3.2) -- (0,2);
	\end{scope}
	\begin{scope}[shift={(2.2,0)}] 	\draw[thick] (0,3) -- (1.2,3); 	\draw (0.2,3.2) -- (0,2); 	\node at (0.1,1.9) {$\vdots$}; 	\node at (0.4,1.8) {$\Delta$}; 	\draw (0,1.6) -- (0.2,0.4); \draw[dashed] (0,2.5) -- (1.2,2.5);
	\end{scope}
	\begin{scope}[shift={(4.4,0)}] 	\draw[thick] (0,3) -- (1.2,3); 	\draw (0.2,3.2) -- (0,2); 	\node at (0.1,1.9) {$\vdots$}; 	\node at (0.4,1.8) {$\Delta$}; 	\draw (0,1.6) -- (0.2,0.4); \draw[dashed] (0,0.8) -- (1.2,0.8);
	\end{scope}
	\begin{scope}[shift={(6.6,0)}] 	\draw (0.2,3.2) -- (0,2); 	\node at (0.1,1.9) {$\vdots$}; 	\node at (0.4,1.8) {$\Gamma$}; 	\draw (0,1.6) -- (0.2,0.4); \draw[dashed] (0,0.8) -- (1.2,0.8);
	\end{scope}
	\begin{scope}[shift={(8.8,0)}] \draw (0.2,3.2) -- (0,2.2); 	\node at (0.1,2.1) {$\vdots$}; 	\node at (0.4,2) {$\Lambda$}; 	\draw (0,1.8) -- (0.2,0.8); 	\draw (0.2,1) -- (0,0); 	\node at (0.4,0.5) {$-3$}; 	\draw[dashed] (0,3) -- (1.2,3);
	\end{scope}
	\begin{scope}[shift={(11,0)}] 	\draw (0.2,3.2) -- (0,2.2); 	\node at (0.1,2.1) {$\vdots$}; 	\node at (0.4,2) {$\Lambda$}; 	\draw (0,1.8) -- (0.2,0.8); 	\draw (0.2,1) -- (0,0); 	\node at (0.4,0.5) {$-3$}; 	\draw[dashed] (-0.1,0.2) -- (1.1,0.2); 
	\end{scope}
	\begin{scope}[shift={(13.2,0)}] 	\draw (0.2,3.2) -- (0,2); 	\node at (0.3,2.1) {$\Lambda$}; 	\draw (0,2.2) -- (0.2,1); 	\draw (0.2,1.2) -- (0,0); 	\node at (0.4,0.5) {$-3$}; 	\draw[dashed] (0,1.6) -- (1.2,1.6);
	\end{scope}
\end{tikzpicture}
}\bigskip

\subcaptionbox*{(2a)}[0.75\textwidth]{
\begin{tikzpicture}
	\begin{scope} 		\draw[thick] (0,3) -- (1.2,3); 		\draw (0.2,3.2) -- (0,2); 		\node at (0.1,2) {$\vdots$};
		\node at (0.4,1.9) {$\Delta$}; 		\draw (0,1.8) -- (0.2,0.6); 		\draw[dashed] (0.2,0.8) -- (0,-0.4);
		\draw[thick] (-0.2,-0.2) -- (1,-0.2);
	\end{scope}
	\begin{scope}[shift={(2.2,0)}] 		\draw[thick] (0,3) -- (1.2,3); 		\draw[thick] (-0.1,1.6) -- (1.1,1.6); 		\draw[dashed] (0.2,3.2) -- (0,2); 		\draw (0,2.2) -- (0.2,1); 			\node at (0.1,1) {$\vdots$}; 			\node at (0.4,0.9) {$\Delta$}; 		\draw (0.2,0.8) -- (0,-0.4);
	\end{scope} 	
	\begin{scope}[shift={(4.4,0)}] 		\draw[thick] (0,3) -- (1.2,3); 		\draw[dashed] (0.2,3.2) -- (0,2);
		\draw (0,2.2) -- (0.2,1); 		\node at (0.1,1) {$\vdots$}; 		\node at (0.4,0.9) {$\Gamma$};
		\draw (0.2,0.8) -- (0,-0.4); 
	\end{scope}
	\begin{scope}[shift={(6.6,0)}] 		\draw[thick] (0,3) -- (1.2,3); 		\draw[dashed] (0.2,3.2) -- (0,2.2); 		\draw (0,2.4) -- (0.2,1.4); 		\node at (0.4,2) {$-3$}; 		\draw (0.2,1.6) -- (0,0.6); 		\node at (0.1,0.6) {$\vdots$}; 		\node at (0.4,0.5) {$\Lambda$}; 		\draw (0,0.4) -- (0.2,-0.6);
	\end{scope}
	\begin{scope}[shift={(8.8,0)}] 		\draw[thick] (0,3) -- (1.2,3); 		\draw[dashed] (0.2,3.2) -- (0,2.2); 		\draw (0,2.4) -- (0.2,1.4); 		\node at (0.1,1.4) {$\vdots$}; 		\node at (0.4,1.3) {$\Lambda$};  \draw (0.2,1.2) -- (0,0.2); 		\node at (0.4,-0.2) {$-3$}; 		\draw (0,0.4) -- (0.2,-0.6);
	\end{scope}
\end{tikzpicture}
}\bigskip
\subcaptionbox*{(2b)}[0.15\textwidth]{
\begin{tikzpicture}
		\draw[thick] (1.2,1.9) -- (1.2,0.3); 		\draw (0,2.8) -- (0.2,1.6); 		\node at (0.3,2.4) {$\Lambda$}; 		\draw (0.2,1.8) -- (0,0.6); 		\node at (0.4,0) {$-3$}; 		\draw (0,0.8) -- (0.2,-0.4); 		\draw[dashed] (0,1.1) -- (1.4,1.1); 
	\end{tikzpicture}
}\bigskip

\subcaptionbox*{(2c)}[0.2\textwidth]{
	\begin{tikzpicture}
			\draw[dashed] (0,3) -- (1.8,3); 			\draw (0.2,3.2) -- (0,2); 			\node at (0.4,2.5) {$-3$}; 			\draw (0,2.2) -- (0.2,1); 			\node at (0.1,1) {$\vdots$}; 			\node at (0.4,0.9) {$\Lambda$}; 			\draw (0.2,0.8) -- (0,-0.4); 			\draw (1.6,3.2) -- (1.4,2); 			\node at (1.8,2.5) {$-3$}; 			\draw (1.4,2.2) -- (1.6,1); 			\node at (1.5,1) {$\vdots$}; 			\node at (1.8,0.9) {$\Lambda$}; 			\draw (1.6,0.8) -- (1.4,-0.4);
	\end{tikzpicture}
}
\subcaptionbox*{(2d)}[0.4\textwidth]{
	\begin{tikzpicture}
		\begin{scope} 			\draw[dashed] (0,2.6) -- (1.8,2.6); 			\draw (0.2,3.2) -- (0,2); 			\node at (0.5,2.9) {$-3$}; 			\draw (0,2.2) -- (0.2,1);  			\node at (0.1,1) {$\vdots$};			\node at (0.4,0.9) {$\Lambda$}; 			\draw (0.2,0.8) -- (0,-0.4); 			\draw (1.6,3.2) -- (1.4,2); 			\node at (1.7,2.2) {$\Gamma$};
		\end{scope}
		\begin{scope}[shift={(3,0)}] 			\draw[dashed] (0,2.6) -- (1.8,2.6); 			\draw (0.2,3.2) -- (0,2); 			\node at (0.5,2.9) {$-3$}; 			\draw (0,2.2) -- (0.2,1); 			\node at (0.1,1) {$\vdots$}; 			\node at (0.4,0.9) {$\Lambda$}; 			\draw (0.2,0.8) -- (0,-0.4); 			\draw (1.6,3.2) -- (1.4,2); 	\node at (1.7,2.2) {$\Delta$}; \draw[thick] (1.4,3) -- (2.4,3);
		\end{scope}
	\end{tikzpicture}
}
\subcaptionbox*{(3)}[0.2\textwidth]{
	\begin{tikzpicture}
		\draw[dashed] (0,2.6) -- (3,2.6); 		\draw (0.2,3.2) -- (0,2); 		\node at (0.5,2.9) {$-3$}; 		\draw (0,2.2) -- (0.2,1); 		\node at (0.1,1) {$\vdots$}; 		\node at (0.4,0.9) {$\Lambda$}; 		\draw (0.2,0.8) -- (0,-0.4); 	\draw (1.6,3.2) -- (1.4,2); 		\node at (1.7,2.2) {$\Gamma$}; 		\draw[thick] (2.8,3.2) -- (2.8,1.6);
	\end{tikzpicture}
}\medskip
\caption{Lemma \ref{lem:a.l.e.1_r=1/2}(1)-(3). Thick line is $R-E$, dashed line is $A$.}\label{fig:lem:a.l.e.1_r=1/2-first}
\end{figure}

\begin{figure}[h]
\subcaptionbox*{(4)}[\textwidth]{
	\begin{tikzpicture}
		\begin{scope} 			\draw[thick] (0,0) -- (1.6,0); 			\draw[dashed] (0.2,-0.6) -- (0.2,0.2) to[out=90,in=90] (1.4,0.2) -- (1.4,-0.6); 
		\end{scope} 
		\begin{scope}[shift={(2.5,0)}] 			\draw[thick] (0,0) -- (1.6,0); 			\draw[dashed] (0.2,-0.6) to[out=90,in=180] (0.8,0) to[out=0,in=90] (1.4,-0.6);
		\end{scope}
		\begin{scope}[shift={(5,0)}] 			\draw[thick] (0,0.4) -- (1.6,-0.4); 			\draw[thick] (0,-0.4) -- (1.6,0.4); 			\draw[dashed] (0.8,0.6) -- (0.8,-0.6);
		\end{scope}
		\begin{scope}[shift={(7.5,0)}] 			\draw[thick] (0,0.4) to[out=-60,in=180] (0.8,0) to[out=180,in=60] (0,-0.4); 			\draw[dashed] (0.8,0.6) -- (0.8,-0.6);
		\end{scope}
	\end{tikzpicture}
}
\medskip
\subcaptionbox*{(5)}[\textwidth]{
	\begin{tikzpicture}
		\begin{scope} 			\draw[thick] (0,0) -- (1.6,0); 			\draw[dashed] (0.2,-0.6) -- (0.2,0) -- (0.2,0.2) to[out=90,in=90] (1.4,0.2) -- (1.4,0) to[out=-90,in=100] (1.5,-0.7); 			\draw (1.5,-0.4) -- (1.3,-1.4); 			\node at (1.4,-1.45) {$\vdots$};  			\node at (1,-1.5) {$\Gamma$}; 			\draw (1.3,-1.7) -- (1.5,-2.7); 
		\end{scope}
		\begin{scope}[shift={(2.5,0)}] 			\draw[thick] (0,0) -- (1.6,0); 			\draw[dashed] (0.2,-0.6) to[out=90,in=180] (0.8,0) to[out=0,in=100] (1.5,-0.7); 			\draw (1.5,-0.4) -- (1.3,-1.4); 			\node at (1.4,-1.45) {$\vdots$}; 			\node at (1,-1.5) {$\Gamma$}; 			\draw (1.3,-1.7) -- (1.5,-2.7); 
		\end{scope}
		\begin{scope}[shift={(5,0)}]
			\draw[thick] (0,0.4) -- (1.6,-0.4); 			\draw[thick] (0,-0.4) -- (1.6,0.4); 			\draw[dashed] (0.8,0.6) -- (0.8,0) to[out=-90,in=100] (0.9,-0.7); 			\draw (0.9,-0.4) -- (0.7,-1.4); 			\node at (0.8,-1.45) {$\vdots$}; 			\node at (0.4,-1.5) {$\Gamma$}; 			\draw (0.7,-1.7) -- (0.9,-2.7);  
		\end{scope}
		\begin{scope}[shift={(7.5,0)}]
			\draw[thick] (0,0.4) to[out=-60,in=180] (0.8,0) to[out=180,in=60] (0,-0.4); 			\draw[dashed] (0.8,0.6) -- (0.8,0) to[out=-90,in=100] (0.9,-0.7); 			\draw (0.9,-0.4) -- (0.7,-1.4); 			\node at (0.8,-1.45) {$\vdots$}; 			\node at (0.4,-1.5) {$\Gamma$}; 			\draw (0.7,-1.7) -- (0.9,-2.7);
		\end{scope}
	\end{tikzpicture}
}
\caption{Lemma \ref{lem:a.l.e.1_r=1/2}(4)-(5). Thick line is $R-E$, dashed line is $A$.}\label{fig:lem:a.l.e.1_r=1/2-second}
\end{figure}
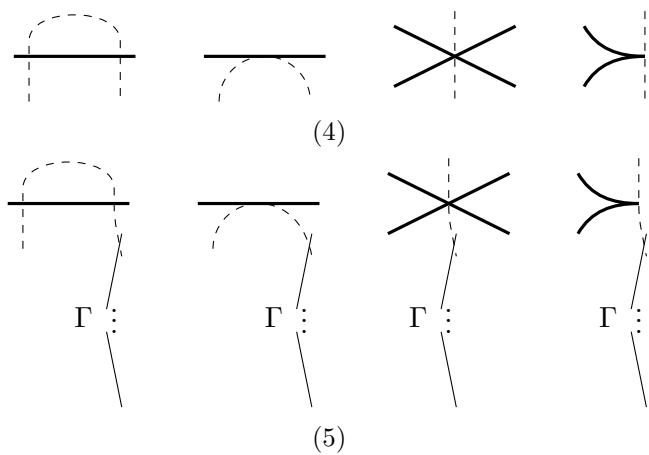

\begin{proof}
Let $\alpha\:(X,\frac{1}{2}D)\to (\ov X,\frac{1}{2}\ov D)$ be the unique maximal pure partial peeling morphism. Put $R=D-E$ and $E^\flat=\cf_X(\ov X,\frac{1}{2}\ov D)$. Let $A$ be an almost log exceptional curve of the first or second kind on $(X,\frac{1}{2}D)$. By Lemma \ref{lem:a.l.e._is_redundant}(3), $A$ is an $\alpha$-redundant $(-1)$-curve of the first kind on $(X,\frac{1}{2}(D+A))$.  By Corollary \ref{cor:peel-squeeze_r=half}(4) it meets at most one $(-2)$-twig $\Delta_A$ of $D+A$ and this twig satisfies $$A\cdot (D-\Delta_A)< 2+\frac{1}{d(\Delta_A)}.$$ It follows that $A\cdot D\leq 2+A\cdot \Delta_A\leq 3$. Moreover, since $\Delta_A$ is a twig of $D+A$ met by $A$, it is a $(-2)$-rod of $D$ (in particular a connected component of $\Gamma$) met by $A$ in a tip.  

The negative definiteness of $E+A$ implies that $A$ does not meet a $(-2)$-fork in $\Gamma$ and if it meets $\Delta+\Gamma$ then only once, in a tip. Similarly, if it meets $\Lambda$ then each connected component at most once, either in a tip or in the middle component of $[2,2,3]$. For $A\cdot D\leq 1$ we get (1). We may thus assume that $A\cdot D\in \{2,3\}$. If $A\cdot E=0$ then $A\cdot (K_X+\frac{1}{2}R+E^\flat)=-1+\frac{1}{2}A\cdot D\geq 0$, which gives (4). We may thus assume that $A\cdot E\neq 0$, too.

Consider the case $A\cdot D=2$. Note that the coefficients of $E^\flat$ are all smaller than $\frac{1}{2}$. We infer that $A$ is of the first kind, because otherwise $A\cdot (R+2E^\flat)=-2A\cdot K_X=2$, hence $A\cdot E>2A\cdot E^\flat=2-A\cdot R=A\cdot E$, which is impossible. Let $T_1$, $T_2$ be the components of $D$ meeting $A$. By the negative definiteness of $E+A$, we see that $T_1\neq T_2$ and that in case $A\cdot R=1$ we have (2a) or (2b). Similarly, in case $A\cdot R=0$ we get (2c) or (2d).

Consider the case $A\cdot D=3$. Then $A\cdot \Delta_A=1$ and $A\cdot (D-\Delta_A)=2$. In particular, $A\cdot R\leq 2$. The negative definiteness of $E+A$ implies that $A\cdot (\Gamma-\Delta_A)=A\cdot \Delta=0$ and $A\cdot R\in \{1,2\}$, hence we get (3) or (5). The $(-1)$-curve $A$ is log exceptional of the second kind on $(X,\frac{1}{2}D)$ if and only if $A\cdot E^\flat=1-\frac{1}{2}A\cdot R$, which holds for (5) and fails for (3).
\end{proof}

The characterization in Lemma \ref{lem:a.l.e.1_r=1/2} gives the following corollary.

\begin{cor}\label{cor:Euler(ALE)}
Let $X$ be a smooth projective surface and $D$ a reduced divisor which contains no superfluous $(-1)$-curve (for instance, $D$ is snc-minimal or $(X,\frac{1}{2}D)$ is squeezed). If a curve $A\not\leq D$ is almost log exceptional of the first kind on $(X,\frac{1}{2}D)$ then $A\cap (X\setminus D)$ is isomorphic to $\P^1$, $\A^1$, $\A^*$ or $\A^{**}=\A^1\setminus\{0,1\}$. In the last case $A$ meets three connected components of $D$ (in particular, $X\setminus D$ is not affine), each once in the sense of intersection theory; one of them is $[2]$, and the second one is $[3,2,\ldots,2]$ and $A$ meets the $(-3)$-curve.
\end{cor}

\medskip
\section{Appendix: Graphs, discriminants and coefficients}\label{sec:d_and_ld}

For the convenience of the reader we recall some terminology and elementary results concerning discriminants and coefficients (equivalently, log discrepancies) in graph-theoretic terms. 

Let $D$ be an undirected graph with vertices $|D|=\{D_1,\ldots,D_n\}$ of weight $w_i\in \Z$ and unique edges between $D_i$ and $D_j$ for (not necessarily all) $i\neq j$ of weight $w_{i,j}\in \Z_{>0}$. We put $D_i\cdot D_j=w_{i,j}$ if $i\neq j$ and $D_i^2:=D_i\cdot D_i=-w_i$ otherwise and we extend this product linearly to $\span_\Q\{|D|\}$. For instance, $D$ could be the dual graph of a reduced divisor on a smooth projective surface with '$\cdot$' being the intersection product. Put $Q(D)=[D_i\cdot D_j]_{i,j\leq n}$ and $d(\emptyset)=1$ and call $$d(D):=\det(-Q(D))$$ the \emph{discriminant of $D$}; it does not depend on the chosen order of vertices. By a \emph{weighted subgraph} of $D$ we mean a subgraph of the underlying unweighted graph of $D$ together with the associated integral weights of vertices and positive weights of edges not exceeding the weights of respective vertices and edges of $D$. Given $T$ - a subset of vertices of $|D|$ or a subgraph of $D$, we denote by $D-T$ the subgraph remaining from $D$ after removing vertices belonging to $T$ and edges incident to them. 
We call 
\begin{equation*}
\beta_D(T)=\sum_{U\in T}U\cdot \sum_{V\in |D-T|} V
\end{equation*}
the \emph{branching number of $T$ in $D$}. A vertex with $\beta_D\leq 1$ is called a \emph{tip} of $D$. A \emph{twig} of $D$ is a connected subgraph $T$ of $D$ containing a tip of $D$ and such that $\beta_D\leq 2$ for all vertices of $T$. It is a \emph{maximal twig} of $D$ if $|T|$ is a maximal subset of $|D|$ with this property. Assume $D$ is connected. If $\beta_D\leq 2$ for every vertex of $D$ then $D$ is called a \emph{chain} in case at least one of these inequalities is strict and is \emph{circular} otherwise. Clearly, a twig is a chain. If $D$ has no circular subgraph (in particular all its edges have weight $1$) then it is called a \emph{tree}. A tree with only one branching component is a \emph{fork}. 

The following lemma follows from the additivity of the determinant function with respect to column addition and its behavior on block-triangular matrices. 

\begin{lem} \label{lem:d(D1+D2)_graph}
Assume that for a weighted graph $D$ as above we have two subgraphs $D_1$ and $D_2$ such that $|D|$ is a disjoint sum of $|D_1|$ and $|D_2|$ and such that there exist unique vertices $T_1\in |D_1|$ in $T_2\in |D_2|$ joined by an edge in $D$. Then
\begin{equation}\label{eq:d(D1+D2)_graph}
d(D)=d(D_1)d(D_2)-(T_1\cdot T_2) d(D_1-T_1)d(D_2-T_2).
\end{equation}
\end{lem}
In particular, if $T_1$ is a tip of $D$ and $T_2$ is the vertex adjacent to $T_1$ then
\begin{equation}\label{lema:d(D-tip)_graph}
d(D)=(-T_1^2)d(D-T_1)-d(D-T_1-T_2).
\end{equation}

\medskip
To study graphs associated with exceptional divisors of resolutions of surface singularities we use some additional notions. A \emph{decorated weighted graph} is a triple $(|D|, \cdot, \vartheta)$ where $(|D|,\cdot)$ is a weighted graph as above and $\vartheta$, the \emph{decoration}, is a function $\vartheta\:|D|\to \Q_{\geq 0}$. By a weighted decorated subgraph we mean a weighted subgraph together with a decoration not exceeding the original decoration. 

\begin{rem*}
Given a resolution $\pi\:X\to \ov X$ of a singular projective surface $\ov X$ and an effective divisor $\ov D$ on $\ov X$, on the weighted graph $E$ associated to the exceptional divisor of $\pi$ we have a decoration function defined on vertices $E_j\in |E|$ by $$\vartheta(E_j)=2p_a(E_j)+E_j\cdot \pi^{-1}_*\ov D,$$ where $p_a$ denotes the arithmetic genus.
\end{rem*}

\begin{dfn}\label{dfn:resolution_graph} A weighted decorated graph $E=(|E|,\cdot,\vartheta)$ is called a \emph{resolution graph} if all its weights and decorations are non-negative and $Q(E)$ is negative definite.
\end{dfn}

A resolution graph is \emph{minimal} if all vertex weights are at least $2$. If they are all equal to $2$ and additionally $\vartheta=0$ (as a function) then the graph is called \emph{du Val}. Every connected du Val resolution graph has Dynkin type A-D-E and comes from a du Val surface singularity. Any minimal chain $E$ is a resolution graph by \eqref{lema:d(D-tip)_graph}. The latter formula implies also that if $E_1$ is a tip of a minimal chain $E$ then $d(E-E_1)< d(E)$. A minimal twig is called \emph{admissible}. A minimal fork $E$ with branching component $B$ and maximal twigs $T_i$, $i=1,2,3$ is \emph{admissible} if $\beta_E(B)=3$ and $\sum_{i=1}^3\frac{1}{d(T_i)}>1$. We note that an admissible fork $E$ has a negative definite intersection matrix. Indeed, by Sylvester's criterion it is sufficient to argue that $d(E)>0$. Denoting by $C_i$ the vertex of $T_i$ adjacent to $B$, by Lemma \ref{lem:d(D1+D2)_graph}
\begin{equation}\label{eq:fork_d_v2}
d(E)=d(T_1)d(T_2)d(T_3)(-B^2-\sum_{i=1}^3\frac{d(T_i-C_i)}{d(T_i)}).
\end{equation}
Since $d(T_i-C_i)\leq d(T_i)-1$, we have $\sum_{i=1}^3\frac{d(T_i-C_i)}{d(T_i)}\leq 3-\sum_{i=1}^3\frac{1}{d(T_i)}<2\leq -B^2$, hence $d(E)>0$.

\medskip

To every resolution graph $E$ we associate the \emph{coefficient function} $\cf_E$ on the set of vertices defined uniquely by the equations
\begin{equation}\label{eq:discrep_equations_with_d}
\sum_i \cf_E(E_i)(-E_i\cdot E_j)=k_{E}(E_j), \text{\quad where\quad} k_{E}(E_j)=\vartheta(E_j)-E_j\cdot E_j-2,\ \ j=1,\ldots, n.
\end{equation}
This agrees with equations for coefficients of prime divisors for resolution graphs coming from birational morphisms of log surfaces. For a minimal connected resolution graph we have $k_{E}\geq 0$ and the inequality is strict unless $E$ is du Val. We write the equations equivalently as 
\begin{equation}\label{eq:discrep_equations}
\sum_i (1-\cf_E(E_i))(-E_i\cdot E_j)=u_{E}(E_j), \text{\quad where\quad} u_{E}(E_j)=2-\beta_E(E_j)-\vartheta(E_j),\ \ j=1,\ldots, n.
\end{equation}

\begin{lem}[Negativity Lemma, {\cite[3.39]{KollarMori-bir_geom}}] \label{lem:minimal_res} Let $E$ be a connected resolution graph with vertices $E_i$, $i=1,\ldots,n$.
\begin{enumerate}[(1)]
\item If $A=\sum_{j=1}^n a_j E_j$ is such that $(-A)\cdot E_i\geq0 $ for every $i=1,\ldots,n$ then either $A=0$ or $a_j>0$ for every $j=1,\ldots,n$.
\item All entries of $(-Q(E))^{-1}$ are positive. 
\item If $E$ is minimal then either $\cf_E>0$ or $E$ is du Val and then $\cf_E=0$.
\end{enumerate}
\end{lem}

\begin{proof}(1) Write $A=A^+-A^-$ where $A^+$ and $A^-$ have non-negative coefficients and contain no common vertex. We have $(A^-)^2=(A^+-A)\cdot A^- \geq (-A)\cdot A^-\geq 0$, so $A^-=0$, because $Q(E)$ is negative definite. Thus $A\geq 0$. Suppose that $A\neq 0$ and $A$ does not contain some vertex of $E$. By the connectedness of $E$ there is a vertex $E_{i_0}$ of $E$ not contained in $A$ but meeting $A$. Then $A\cdot E_{i_0}>0$; a contradiction.  

(2) Apply (1) to $A$ such that $[a_1,\ldots,a_n]^T$ is a column of the matrix $-Q(E)$.  

(3) Since $k_E(E_j)\geq 0$, the statement is a consequence of part (2) and  \eqref{eq:discrep_equations_with_d}.
\end{proof}

If $E$ is a tree then between any two vertices $E_i$, $E_j$, $i\neq j$ there exists a unique shortest path in the graph joining them, which we denote by $\path(E_i,E_j)$.

\begin{lem}[{\cite[3.1.10]{Flips_and_abundance}}]
 If a resolution graph  $E=(\{E_1,\ldots,E_n\},\cdot,\vartheta)$ is a tree then
 \begin{equation}\label{ex:Alexeev_formula}
 d(E)(1-\cf_E(E_j))=\sum_{i=1}^n u_E(E_i)\cdot d(E-\path(E_i,E_j)),
 \end{equation}
where $u_{E}(E_j)=2-\beta_E(E_j)-\vartheta(E_j).$
\end{lem}

\begin{proof}
Denote by $M^{[i,j]}$ the $(i,j)$-th cofactor of a square matrix $M$. Given \eqref{eq:discrep_equations} it is sufficient to argue that 
\begin{equation}\label{eq:cofactor}
(-Q(E))^{[i,j]}=d(E-\path(E_i,E_j)).
\end{equation}
The proof is by induction on $n$. We may assume that $n\geq 2$, $E_n$ is a tip of $E$ joined by an edge with $E_{n-1}$ and $i\leq j$. If $j=n$ then by induction $(-Q(E))^{[i,n]}=(-Q(E-E_n))^{[i,n-1]}=d(E-E_n-\path(E_i,E_{n-1}))=d(E-\path(E_i,E_n))$. If $i,j<n$ then the formula is obtained by applying successively expansion of the determinant along a column, inductive assumption and then  \eqref{lema:d(D-tip)_graph}.
\end{proof}

We say that $E$ is $\epsilon$-dlt if $\cf_E<1-\epsilon$. It is \emph{terminal} if it is $1$-dlt. The following lemma follows from \cite[3.1.3]{Flips_and_abundance} and \cite[3.7(ii)-(iv)]{Alexeev-Fractional_del_Pezzo}, cf.\ \cite[L.1(2) p.147]{Keel-McKernan_rational_curves}. We add some details in the proof. 

\begin{lem}\label{lem:Alexeev} Let $E$ be a connected minimal resolution graph and $F$ a proper minimal subgraph of $E$. Assume that one of the following holds:
\begin{enumerate}[(1)]
\item $F$ has the same vertex weights as $E$,
\item $F$ arises from $E$ by decreasing the weight of a single vertex $E_0$ and we have $d(F)>0$ and ${\cf_F(E_0)\leq 1}$.
\end{enumerate} 
Then $F$ is a resolution graph and $\cf_F\leq \cf_E.$ Moreover, either the inequality is strict or $\cf_E=\cf_F$ and $E$ (and hence $F$) is du Val in case (1) or $\cf_F(E_0)=1$ in case (2).
\end{lem}

\begin{proof} Let $|E|=\{E_1,\ldots, E_n\}$ and $|F|=\{E_1,\ldots, E_{n'}\}$, $n'\leq n$. Put $x_i=\cf_E(E_i)$, $e_{ij}:=-Q(E)_{i,j}=-E_i\cdot_E E_j$, $k_i=k_E(E_i)$ and $\vartheta_i=\vartheta_E(E_i)$. Similarly, put $x_i'=\cf_F(E_i)$, $e_{ij}'=-E_i\cdot_F E_j$, $k_i'=k_F(E_i)$  and $\vartheta_i'=\vartheta_F(E_i)$. Due to Sylvester's criterion, to prove that $-Q(F)$ is positive definite it is sufficient to argue that $d(F)>d(E)$ when $F$ arises from $E$ by decreasing the $(n-1,n)$-th (and $(n,n-1)$-th) entry by $u>0$. Via symmetric Gaussian elimination we change all other non-diagonal entries to zero. The operation keeps non-diagonal entries non-negative and (due to the positivity of leading principal minors) keeps diagonal ones positive, hence reduces the proof to the easy case $n=2$.

For $E$ the equations \eqref{eq:discrep_equations} give $$\sum_{i=1}^n  x_ie_{ij}=k_j, \quad j=1,\ldots,n $$ and analogously for $F$. Put $x_i'=k_i'=e_{ij}'=0$ if $i>n'$. Subtracting the equations we get:
$$\sum_{i=1}^{n}e_{ij}(x_i-x_i')=\Delta_j,\quad \text{where} \quad \Delta_j:=k_j-k_j'+\sum_{i=1}^{n'}x_i'(e_{ij}'-e_{ij}), \quad j=1,\ldots,n.$$ 
If $E$ is du Val then $k_E=0$, hence $k_F=0$ and $F$ is du Val, so $\cf_F=\cf_E=0$. We may therefore assume that $E$ is not du Val. By Lemma \ref{lem:minimal_res} all entries of $(-Q(E))^{-1}$ are positive. Since $E$ is connected, we get $\cf_E>0$, so $x_i>0$ for each $i$. To prove the inequality $\cf_F<\cf_E$, which is equivalent to the conjunction of inequalities $x_i-x_i'>0$ for each $i$, it is sufficient to show that the numbers $\Delta_j$ are non-negative and not all of them vanish.

If $F$ arises from $E$ by decreasing some decoration, then $\Delta_j=k_j-k_j'=\vartheta_j-\vartheta_j'\geq 0$ and the latter inequality is strict for some $j$, so we are done. 

Assume that $F$ arises from $E$ by removing some vertex, say the last one. Then $\Delta_j=k_j-k_j'=0$ for $j<n$ and $\Delta_n=k_n+\sum_{i=1}^{n-1}x_i' (-e_{in})$. We have $k_n, x_i', -e_{in}\geq 0$, so $\Delta_n\geq 0$. Suppose that the equality holds. Then $k_n=0$ and $x_i'\cdot e_{in}=0$ for each $i<n$. Since $E$ is not du Val and since $k_n=0$, $F$ has a connected component which is not du Val. Since $E$ is connected, this component contains some $E_i$ with $e_{in}\neq 0$. Then $x_i'\cdot e_{in}\neq 0$ by Lemma \ref{lem:minimal_res}; a contradiction. 

Assume that $F$ arises from $E$ by decreasing the weight of an edge, say $e_{n-1,n}-e'_{n-1,n}=u<0$. Then $\Delta_j=0$ for $j\leq n-2$, $\Delta_{n-1}=-ux'_n\geq 0$ and $\Delta_n=-ux'_{n-1}\geq 0$. Suppose that $x'_{n}=x'_{n-1}=0$. The definition implies that $F$ has at most two connected components and if it has two then one contains $E_{n-1}$ and the other $E_n$. Since $x'_{n}=x'_{n-1}=0$, we infer that $F$ is du Val. But then $E$ is not du Val; a contradiction. 

Finally, assume that $F$ arises by decreasing the weight of a vertex. Put $E_n:=E_0$ and $e_{n,n}-e'_{n,n}=u>0$. Then $\Delta_j=0$ for $j\leq n-1$ and $\Delta_n=k_n-k'_n-ux'_n=u(1-x'_n)=u(1-\cf_F(E_n))$. By assumption $\cf_F(E_n)\leq 1$, so we have $\cf_F\leq \cf_E$. Moreover, the inequality is strict if $\cf_F(E_n)<1$ and it is an equality if $\cf_F(E_n)=1$.
\end{proof}

\begin{lem}\label{lem:log_terminal} Let $E$ be a connected minimal resolution graph which is log terminal, i.e.\ $\cf_E<1$. Then $\vartheta_E<2$ and $E$ is a chain or an admissible fork. For every proper minimal subgraph $F$ one has $\cf_F\leq \cf_E$ and the inequality is strict, unless $E$ is du Val. Moreover, if $F$ arises from $E$ by lowering some vertex weights only, then $d(F)<d(E)$. 
\end{lem}

\begin{proof} The classification of log terminal graphs is well-known and follows from the above Lemmas \ref{lem:d(D_1+D_2)} and \ref{lem:Alexeev}. (Note that for a circular graph resolution with $\vartheta=0$ one has $\cf=1$ by \eqref{eq:discrep_equations}, which implies that $E$ is a tree. Then using \eqref{ex:Alexeev_formula} one checks that $E$ does not contain a subgraph with two branching vertices). We infer that $E$ is a chain or an admissible fork and hence the same holds for every connected component of $F$, too. In particular, $d(F)>0$. The inequality $\vartheta<2$ holds, because it holds for a graph with one vertex. Assume that $E$ is not du Val and let $F$ be a subgraph of $E$. By Lemma \ref{lem:Alexeev}, to show that $\cf_F<\cf_E$ we may assume that $F$ arises from $E$ be decreasing the weight of a single vertex $E_0$. By \eqref{ex:Alexeev_formula} the product $d(E)(1-\cf_E(E_0))$ does not depend on the weight of $E_0$, so $$d(F)(1-\cf_F(E_0))=d(E)(1-\cf_E(E_0))>0.$$ We get $\cf_F(E_0)<1$ and hence $\cf_F<\cf_E<1$ by Lemma \ref{lem:Alexeev}. We have $d(F)<d(E)$ in this case, too.
\end{proof}

\medskip
\bibliographystyle{amsalpha} \bibliography{C:/KAROL/PRACA/PUBLIKACJE/BIBL/bibl2023}
\end{document}